\documentclass{article}

\usepackage{microtype}
\usepackage{graphicx}
\usepackage{subfigure}
\usepackage{booktabs} % for professional tables

\usepackage{hyperref}

\usepackage[accepted]{icml2024}

\usepackage{amsmath}
\usepackage{amssymb}
\usepackage{mathtools}
\usepackage{amsthm}

\usepackage[capitalize,noabbrev]{cleveref}
\crefname{equation}{}{}
\crefname{figure}{Figure}{Figures}
\creflabelformat{equation}{\textup{(#2#1#3)}}
\crefname{assumption}{Assumption}{Assumptions}
\crefname{condition}{Condition}{Conditions}

\theoremstyle{plain}
\newtheorem{theorem}{Theorem}[section]

\newtheorem{lemma}[theorem]{Lemma}
\newtheorem{corollary}[theorem]{Corollary}
\theoremstyle{definition}
\newtheorem{definition}[theorem]{Definition}
\newtheorem{assumption}[theorem]{Assumption}
\theoremstyle{remark}
\newtheorem{remark}[theorem]{Remark}

\usepackage[textsize=tiny]{todonotes}

\usepackage{derivative}
\usepackage{ulem}

\newcommand{\real}{\mathbb{R}}

\DeclareMathOperator*{\argmin}{arg\,min}

\newcommand*\tageq{\refstepcounter{equation}\tag{\theequation}}

\newcommand{\sA}{\mathcal{A}}

\newcommand{\sK}{\mathcal{K}}

\newcommand{\sN}{\mathcal{N}}
\newcommand{\sO}{\mathcal{O}}
\newcommand{\sP}{\mathcal{P}}

\newcommand {\RR}  { {\mathbf{R}} }
\newcommand {\BB}  { {\mathbf{B}} }

\newcommand {\DD}  { {\mathbf{D}} }

\newcommand {\HH}  { {\mathbf{H}} }

\newcommand {\VV}  { {\mathbf{V}} }
\newcommand {\WW}  { {\mathbf{W}} }
\newcommand {\UU}  { {\mathbf{U}} }

\newcommand {\QQ}  { {\mathbf{Q}} }

\newcommand {\YY}  { {\mathbf{Y}} }

\newcommand{\eye}{\mathbf{I}}

\renewcommand {\aa}  { {\bf a} }

\newcommand {\ee}  { {\bf e} }

\newcommand {\bgg}  { {\bf g} }

\newcommand {\yy}  { {\bf y} }
\newcommand {\hh}  { {\bf h} }

\newcommand {\rr}  { {\bf r} }

\newcommand {\qq}  { {\bf q} }
\newcommand {\pp}  { {\bf p} }
\newcommand {\bss}  { {\bf s} }
\newcommand {\btt}  { {\bf t} }

\newcommand {\vv}  { {\bf v} }
\newcommand {\ww}  { {\bf w} }
\newcommand {\xx}  { {\bf x} }
\newcommand {\zz}  { {\bf z} }

\newcommand {\zero}  { {\bf 0} }
\newcommand {\one}  { {\bf 1} }

\newcommand {\HHk}  { {\HH_{k}} }

\newcommand {\bggk}  { {{\bgg}_{k}} }

\newcommand {\xxk}  { {{\xx}_{k}} }

\newcommand {\ppk}  { {{\pp}_{k}} }

\newcommand {\range}  { {\textnormal{Range}} }
\newcommand {\Null}  { {\textnormal{Null}} }
\newcommand {\Span}  { {\textnormal{Span}} }

\newcommand {\diag}  { {\textnormal{diag}} }

\newcommand{\hvec}[2]{
\begin{pmatrix} #1 & #2 \end{pmatrix}
}
\newcommand{\vvec}[2]{
\begin{pmatrix} #1 \\ #2 \end{pmatrix}
}

\makeatletter
\newcommand*{\transpose}{%
	{\mathpalette\@transpose{}}%
}
\newcommand*{\@transpose}[2]{%
	\raisebox{\depth}{$\m@th#1\intercal$}%
}
\makeatother

\newcommand {\tbeta}  { {\tilde{\beta}} }
\newcommand {\talpha}  { {\tilde{\alpha}} }

\newcommand {\TT}  { {\mathbf{T}} }

\newcommand {\tTT}  { {\tilde{\mathbf{T}}} }

\newcommand {\tRR}  { {\tilde{\mathbf{R}}} }

\renewcommand{\vec}[1]{\ensuremath{\mathbf{#1}}}
\newcommand{\grad}{\ensuremath {\vec \nabla}}

\newcommand{\defeq}{\triangleq}
\usepackage{enumitem}
\setlist[enumerate]{leftmargin=*,wide=0em, noitemsep,nolistsep, label = {\bfseries \arabic*.}}
\setlist[itemize]{leftmargin=*,wide=0em, noitemsep,nolistsep}

\newcommand{\sI}{\mathcal{I}}
\newcommand{\sJ}{\mathcal{J}}

\newcommand{\Dtype}{\text{D}_\text{type}}
\newcommand{\SOL}{\Dtype = \text{SOL}}
\newcommand{\NPC}{\Dtype = \text{NPC}}

\newcommand*\dotprod[1]{\left\langle #1\right\rangle}
\newcommand*\vnorm[1]{\left\| #1\right\|}

\usepackage{xspace}

\renewcommand\th{\textsuperscript{th}\xspace}
\icmltitlerunning{Inexact Newton-type Methods for Optimisation with Nonnegativity Constraints}

\begin{document}

\twocolumn[
\icmltitle{Inexact Newton-type Methods for Optimisation with Nonnegativity Constraints}



\icmlsetsymbol{equal}{*}

\begin{icmlauthorlist}
\icmlauthor{Oscar Smee}{yyy}
\icmlauthor{Fred Roosta}{yyy,xxx}
\end{icmlauthorlist}

\icmlaffiliation{yyy}{School of Mathematics and Physics, University of Queensland, Brisbane, Australia}
\icmlaffiliation{xxx}{ARC Training Centre for Information Resilience, Brisbane, Australia}
%\icmlaffiliation{comp}{Company Name, Location, Country}
%\icmlaffiliation{sch}{School of ZZZ, Institute of WWW, Location, Country}

\icmlcorrespondingauthor{Oscar Smee}{o.smee@uq.edu.au}
%\icmlcorrespondingauthor{Firstname2 Lastname2}{first2.last2@www.uk}

\icmlkeywords{Machine Learning, ICML}

\vskip 0.3in
]



\printAffiliationsAndNotice{}  % leave blank if no need to mention equal contribution

\begin{abstract}
We consider solving large scale nonconvex optimisation problems with nonnegativity constraints. Such problems arise frequently in machine learning, such as nonnegative least-squares, nonnegative matrix factorisation, as well as problems with sparsity-inducing  regularisation. In such settings, first-order methods, despite their simplicity, can be prohibitively slow on ill-conditioned problems or become trapped near saddle regions, while most second-order alternatives involve non-trivially challenging subproblems. The two-metric projection framework, initially proposed by  \citet{ProjectionNewtonMethodsForOptimizationProblemsBertsekas}, alleviates these issues and achieves the best of both worlds by combining projected gradient steps at the boundary of the feasible region with Newton steps in the interior in such a way that feasibility can be maintained by simple projection onto the nonnegative orthant. We develop extensions of the two-metric projection framework, which by inexactly solving the subproblems as well as employing non-positive curvature directions, are suitable for large scale and nonconvex settings. We obtain state-of-the-art  convergence rates for various classes of non-convex problems and demonstrate competitive practical performance on a variety of problems.
\end{abstract}

\section{Introduction}\label{sec:introduction}

We consider high-dimensional problems of the form 
\begin{align}
    \min_{\xx \in \real^{d}} f(\xx), \quad \text{subject to } \quad \xx \geq \zero, \tageq\label{eqn:nonnegative constrained problem}
\end{align}
where $d \gg 1$ and $f:\real^d \to \real$ is twice continuously differentiable and possibly nonconvex function. Despite the simplicity of its formulation, such problems arise in many applications in science, engineering, and machine learning (ML). Typical examples in ML include nonnegative formulations of least-squares and matrix factorisation \citep{Lee1999LearningThePartsOfObjectsByNNMF, Lee2000AlgorithmsForNNMF,gillis2020NNMF}. Additionally, problems involving sparsity inducing regularisation such as $\ell_{1}$ norm, which are typically non-smooth, can be reformulated into a differentiable objective with nonnegativity constraints \citep{Schmidz2007FastOptimizationMethodsForL1Regularization}. 

Many methods have been developed to solve \cref{eqn:nonnegative constrained problem}. First-order methods \citep{lan2020first}, such as projected gradient descent, can be very simple to implement and as such are popular in ML. However, they come with well-known deficiencies, including relatively-slow convergence on ill-conditioned problems, sensitivity to hyper-parameter settings such as learning rate, and difficulty in escaping flat regions and saddle points. On the other hand, general purpose second-order algorithms, e.g., projected Newton method \citep{ ProjectedNewtonTypeMethodsInMachineLearningSchmidt,Lee2014ProximalNewtonTypeMethods} and interior point methods \citep{NocedalJorgeNO}, alleviate some of these issues such as susceptibility to ill-conditioning and/or stagnation near flat regions. However, due to not leveraging the simplicity of the constraint, this advantages come at the cost of introducing highly non-trivial and challenging subproblems. 

By exploiting the structure of the constraint in  \cref{eqn:nonnegative constrained problem}, \citet{ProjectionNewtonMethodsForOptimizationProblemsBertsekas} proposed the two-metric projection framework as a natural and simple adaptation of the classical Newton's method for unconstrained problems.  By judicious modification of the Hessian matrix, this framework can be effectively seen as projecting Newton's step onto the nonnegative orthant. This allows for the best of both worlds, blending the efficiency of classical Newton's method with the simplicity of projected gradient descent. Indeed, similar to the classical Newton's method, the subproblem amounts to solving a linear system, while like projected gradient-descent, the projection step is straightforward. 

\textbf{Contribution}. In this paper, we design, theoretically analyse, and empirically evaluate novel two-metric projection type algorithms (\cref{alg:newton-mr-two-metric,alg:newton-mr-two-metric-minimal-assumptions-case})  with desirable complexity guarantees for solving large scale and nonconvex optimisation problems with nonnegativity constraints \cref{eqn:nonnegative constrained problem}. Both \cref{alg:newton-mr-two-metric,alg:newton-mr-two-metric-minimal-assumptions-case} are Hessian-free in that the subproblems are solved inexactly using  the minimum residual (MINRES) method \citep{PaigeSaunders1975SolutionOfSparseIndefiniteSystemsOfLinearEquations} and only require Hessian-vector product evaluations. To achieve approximate first-order optimality (see \cref{defn:optimal point}), we leverage the theoretical properties of MINRES, as recently established in \cite{LiuMinres}, e.g., nonnegative curvature detection and monotonicity properties, and we show the following:
\begin{enumerate}[label = {\bfseries (\Roman*)}]
    \item Under minimal assumptions, \cref{alg:newton-mr-two-metric-minimal-assumptions-case} achieves global iteration complexity that matches those of first-order alternatives (\cref{thm:minimal assumptions convergence theorem}).
    \item Under stronger assumptions, \cref{alg:newton-mr-two-metric} enjoys a global iteration complexity guarantee with an improved rate that matches the state of the art for second-order methods (\cref{thm:first-order iteration complexity}).
    \item Both variants obtain competitive oracle complexities, i.e., the total number of gradient and Hessian-vector product evaluations (\cref{cor:first-order operational complexity,cor:first-order operational complexity minimal assumptions}). 
    \item Our approach enjoys fast local convergence guarantees (\cref{thm:active set identification,cor:local convergence guarantee}).  
    \item Our approach exhibit highly competitive empirical performance on several machine learning problems (\cref{sec:numerical results}).
\end{enumerate}
To our knowledge, the complexity guarantees outlined in this paper are the first to be established for two-metric projection type algorithms in nonconvex settings.

\textbf{Notation}. Vectors and matrices are denoted, respectively, by bold lowercase and uppercase letters. Denote the nonnegative orthant by $\real^d_+$. The open ball of radius $r$ around $\xx$ is denoted by $B(\xx, r) \defeq \{ \zz \in \real^d \ | \  \|\zz - \xx\| < r\}$. The inequalities,``$\geq$'' and ``$\leq$'', are often applied element-wise. Big-$\sO$ complexity is denoted by $\sO$ with hidden logarithmic factors indicated by $\tilde{\sO}$. Denote components of vectors by superscript and iteration counters as subscripts, e.g., $\xx^i_k$ is $i^\text{th}$ component of the $k^\text{th}$ iterate of $\xx$. As a natural extension, a set of indices in the superscript denotes the subvector corresponding to those components, e.g., letting $[d] = \{1, \ldots, d\}$, if $\sI \subseteq [d]$ and $\vv \in \real^d$ then $\vv^\sI = (\vv^i \ | \ i \in \sI) \in \real^{|\sI|}$. Let $\bgg(\xx) =\nabla f(\xx)$ and $\HH(\xx) = \nabla^2 f(\xx)$ denote the gradient and Hessian of $f$, respectively. Denote the $\delta_k$-active and $\delta_k$-inactive sets, respectively, by  
\begin{subequations} \label{eqn:inactive and active set definition}
    \begin{align} 
    \sA(\xx_k, \delta_k) &= \{ i \in [d] \mid 0 \leq \xx_k^i \leq \delta_k \}, \tageq\label{eqn:active set} \\
    \sI(\xx_k, \delta_k) &= \{ i \in [d] \mid \xx_k^i > \delta_k \}. \tageq\label{eqn:inactive set}
    \end{align} 
\end{subequations}
When the context is clear, we suppress the dependence on $\xx_k$ and $\delta_k$, e.g., $\bggk$ and $\HHk$ for $\bgg(\xxk)$ and $\HH(\xxk)$ and $\xx_k^\sI $ or $\xx_k^{\sI_k}$ instead of $\xx_k^{\sI(\xx_k, \delta_k)}$. We also denote $\HH_k^{\sI} = \{ (\HH_k)_{ij} \ | \ i,j \in \sI(\xx_k, \delta_k) \} $. 

\section{Background and Related Work}
We now briefly review related works for solving \cref{eqn:nonnegative constrained problem} and some essential background necessary for our presentation. 

\textbf{First-order Methods}. The projected gradient method \citep{lan2020first} is among the simplest techniques for solving optimisation problems involving convex constraints. Indeed, the projected gradient iteration for minimisation over a convex set $\Omega$ is simply given by $\xx_{k+1} = \sP_\Omega (\xx_k - \alpha_k \bgg_k)$ where $\sP_\Omega: \real^{d} \to \real^{d}$ is the orthogonal projection onto $\Omega$ defined by $\sP_\Omega(\xx) = \argmin_{\zz \in \Omega} \|\zz - \xx\|$. When $\alpha_k$ is chosen appropriately, e.g., via line search, the projected gradient method is known to converge under essentially the same conditions and at the same rate as the unconstrained variant  \citep{bertsekasNonlinearProgramming, beck2017FirstOrderMethodsInOptimization}. Many variations of this method have also been considered, e.g., spectral projected gradient \citep{Birgin2014SpectralProjectedGradientMethods}, proximal gradient \citep{Parikh2014ProximalAlgorithms,beck2017FirstOrderMethodsInOptimization}, and accelerated proximal gradient \citep{Nesterov2013GradientMethodsForMinimizingCompositeFunction, Beck2009FISTA} with its extensions to non-convex settings \citep{lin2020accelerated,li2017convergence}.

Of course, the effectiveness of the projected gradient method relies heavily on the computational cost associated with computing $\sP_\Omega(\xx)$. While this can be challenging for general convex sets, in the case of $\Omega = \real^d_+$, it is simply given by $[\sP(\xx)]^{i} = \xx^i$, if $\xx^i > 0$, and $[\sP(\xx)]^{i} = 0$, otherwise. Note that, for notational simplicity, we omit the dependence of $\sP$ on $\Omega$ in our context. Nonetheless, while the projected gradient method is a simple choice for solving \eqref{eqn:nonnegative constrained problem}, it shares the common drawbacks of first-order methods alluded to earlier, e.g., susceptibility to ill-conditioning. 

\textbf{Second-order Methods}. By incorporating Hessian information, second-order methods hold the promise to alleviate many of the well-known deficiencies of first-order alternatives, e.g., they are typically better suited to ill-conditioned problems \citep{xu2020second}. For constrained problems, generic projected (quasi) Newton methods involve iterations of the form $\xx_{k+1} = \xxk + \alpha_{k} \ppk$ where 
\begin{align}
    \label{eqn:proximal Newton method subproblem}
    \pp_{k} = \argmin_{\xx \in \Omega} \dotprod{\bgg_k, \pp} + \dotprod{\pp, \BB_k\pp}/2,
\end{align}
where $\alpha_{k}$ is an appropriately chosen step-size, e.g., backtracking line search, and $\BB_k$ captures some curvature information of $f$ at $\xx_k$ (and also potentially the step-length as in the proximal arc search). For $\BB_k = \eye $ we recover a projected gradient variant, whereas for $\BB_k = \HH_k$, or some approximation, we obtain projected (or more generally proximal) Newton-type methods \citep{schmidt2009optimizing,ProjectedNewtonTypeMethodsInMachineLearningSchmidt,becker2012quasi,Lee2014ProximalNewtonTypeMethods,shi2015large}. The main drawback of this framework is that the subproblem, \cref{eqn:proximal Newton method subproblem}, may no longer be a simple projection even when $\Omega$ is a simple, and one has to resort to an optimisation subroutine to (approximately) solve \cref{eqn:proximal Newton method subproblem}.

An alternative is the interior point framework \citep{NocedalJorgeNO}, where the constraints are directly integrated into the objective as ``barrier'' functions. While the subproblems in this framework amount to solving linear systems, to produce accurate solutions the barrier function must approach the constraint, which can lead to highly ill conditioned subproblems. Some recent works \citep{bian2014ComplexityAnalysisOfInteriorPoint,haeser2017OptimalityConditionAndComplexityForLinearlyConstrainedOptimizationWithoutDiffOnBoundary,LogBarrierNewtonCGOneillWright} consider interior point methods for \cref{eqn:nonnegative constrained problem}. In particular, in \citep{LogBarrierNewtonCGOneillWright}, capped Newton-CG with a preconditioned Hessian is used to optimise a log barrier augmented objective. Due to issues arising from increasingly ill-conditioned subproblems, the practical efficacy of this method seems to be inferior when compared to projection-based methods, including those of first-order \citep{XieWright2021ComplexityOfProjectedNewtonCG}.

The issue with the general purpose second-order methods discussed so far is that, unlike projected gradient, they do not leverage the simplicity of the nonnegativity constraints and the corresponding projection. In this light, a na\"{i}ve adaptation of the projected gradient would imply directly projecting the Newton step on the constraints, e.g., $\xx_{k+1} = \sP_\Omega (\xx_k - \alpha_k \HH_{k}^{-1}\bgg_k)$. Unfortunately, such a direct adaptation may lead to ascent directions for the objective function at the boundary. To that end, the \textit{two-metric projection (TMP) framework} \citep{ProjectionNewtonMethodsForOptimizationProblemsBertsekas, TwoMetricProjectionGafiniBertsekas} offers an ingenious solution. Specifically, at each iteration, the component indices, $[d]$, are divided into the approximately bound, $\sJ^+_k$, and free sets, $\sJ^-_k$, given by
\begin{align*}
    \sJ^+_k = \{ i \in [d] \mid  \xx_k^i\leq \delta, \bgg_k^i > 0 \}, \; \sJ^-_k =[d] \setminus \sJ_k^+. \tageq\label{eqn:bertsekas active and inactive set}
\end{align*}
where $\delta > 0$. A matrix, $\DD_k$, is then chosen to be ``diagonal'' with respect to set $\sJ_k^+$, that is,
\begin{align*}
    (\DD_k)_{ij} = 0, \;\; i \in \sJ_k^+, \;\; j \in [d] \setminus\{i\},  
\end{align*}
and the update is simply given by  
\begin{align*}
    \xx_{k+1} = \sP(\xx_k - \alpha_k \DD_k \bgg_k). \tageq\label{eqn:bertsekas two metric projection}
\end{align*}
It has been shown that TMP is asymptotically convergent under certain conditions and reasonable choices of $\DD_k$. For example, for strongly convex problems, the non-diagonal portion of $\DD_k$ can consist of the inverse of the Hessian submatrix corresponding to the indices in $\sJ_k^-$. In this case, \cref{eqn:bertsekas two metric projection} reduces to a scaled gradient in $\sJ_k^+$ and a Newton step in $\sJ_k^-$. \citet{ProjectionNewtonMethodsForOptimizationProblemsBertsekas} also shows that,  under certain conditions, TMP can preserve fast ``Newton like'' local convergence. Practically, TMP type algorithms has been successfully applied to a range of problems \citep{TwoMetricProjectionGafiniBertsekas,Schmidz2007FastOptimizationMethodsForL1Regularization,Sra2010TacklingBoxConstrainedOptimization,HaberComputationalMethodsinGeophysicalElectromagnetics,Kuang2012SymmetricNNMF,Cai2023FastProjectedNewtonLikeMethodForPrecisionMatrix}. In large scale and nonconvex settings, employing the Newton step as part of \cref{eqn:bertsekas two metric projection} may be infeasible or even undesirable. Indeed, not only can Hessian storage and inversion costs be prohibitive, the existence of negative curvature can lead to ascent directions.

With a view to eliminate the necessity of forming and inverting the Hessian, \citet{Sra2010TacklingBoxConstrainedOptimization} extend TMP to utilise a quasi-Newton update with asymptotic convergence guarantees in the convex setting. Also in this vein, \citet{XieWright2021ComplexityOfProjectedNewtonCG} considered ``projected Newton-CG'', which entails a combination of the projected gradient and the inexact Newton steps that preserve the simplicity of projection onto $\real^{d}_{+}$. In particular, Newton-CG steps are based on the capped CG procedure of \citet{RoyerWright2018NewtonCG}. Unfortunately, the gradient and Newton-CG steps are not taken simultaneously. Instead, the algorithm employs projected gradient steps across all components until optimality is attained in the approximately active set. Only at that point is the Newton-CG step applied in the approximately inactive set. This implies that the algorithm may take projected gradient steps at most iterations, potentially impeding its practical performance.

\textbf{Hessian-free Inexact Methods}. In high-dimensional settings, storing the Hessian matrix may be impractical. Moreover, an approximate direction can often be computed at a fraction of the cost of a full Newton step. In this context, Hessian-free inexact Newton-type algorithms leverage Krylov subspace methods \citep{SaadIterativeMF}, which are particularly well-suited for these scenarios. Krylov subspace solvers can recover a reasonable approximate direction in just a few iterations and only require access to the Hessian-vector product mapping, $\vv \mapsto \HH(\xx) \vv$. The computational cost of a Hessian-vector product is comparable to that of a gradient evaluation and does not require the explicit formation of $\HH$. Indeed, $\HH(\xx) \vv$ can be computed by obtaining the gradient of the map $\xx \mapsto \langle \bgg(\xx), \vv \rangle$ using automatic differentiation, leading to one additional back propagation compared to computing $\bgg(\xx)$.

\textbf{Complexity in Optimisation}. Recently, there has been a growing interest in obtaining global worst case \textit{iteration complexity} guarantees for optimisation methods, namely a bound on the number of iterates required for the algorithm to compute an approximate solution. For instance, in unconstrained and nonconvex settings,  gradient descent produces an approximate first-order optimal point satisfying $\| \bgg(\xx) \| \leq \epsilon_g $ in at most $\sO(\epsilon_g^{-2})$ iterations for objectives with Lipschitz continuous gradients \citep{Nesterov2013introductoryLecturesOnConvex}. This rate has been shown to be tight \citep{cartisComplexityOfSteepestDescent}. Without additional assumptions, similar rates have also been shown for second-order methods \cite{cartis2022evaluationComplexityOfAlgorithmsForNonconvexOptimization}. However, for objectives with both Lipschitz continuous gradient and Hessian, this rate can be improved to $\sO(\epsilon_g^{-3/2})$, which is also shown to be tight over a wide class of second-order algorithms \citep{cartis2011OptimalNewtonTypeMethodsForNonConvex}. Second-order methods which achieve this rate include cubic regularised Newton's method and its adaptive variants \citep{Nesterov2006CubicRegularisationOfNewtonsMethod,cartis2011AdaptiveCubicRegularisationPartI,cartis2011AdaptiveCubicRegularisationPartII,xu2020newton}, modified trust region based methods \citep{curtis2016ATrustRegionAlgorithmWithWorstCaseIterativeComplexity,curtis2021TrustRegionNewtonCG,curtis2023WorstCaseComplexityOfTRACE} and line search methods including Newton-CG \citep{RoyerWright2018NewtonCG} and Newton-MR \citep{LiuNewtonMR} as well as their inexact variants \citep{Yao2022InexactNewtonCG,lim2023complexityForNewtonMRUnderInexactHessian}. Many of the above works also provide explicit bounds on the \textit{operational complexity}, that is, a bound on the number of fundamental computational units (e.g. gradient evaluations, Hessian vector products) to obtain an approximate solution.

In the constrained setting, direct comparison between bounds is difficult due to differences in approximate optimality conditions; see discussion in \citet[Section 3]{XieWright2021ComplexityOfProjectedNewtonCG} for the bound constraint case. However, the algorithms in \citet{cartis2020SharpWorstCaseComplexityBoundsForNonConvexOptimizationWithInexpensiveConstraints,birgin2018OnRegularizationAndActiveSetMEthodsWithComplexity} achieve $\sO(\epsilon_g^{-3/2})$ for a first-order optimal point with certain types of constraints, which is shown to be tight in  \citet{cartis2020SharpWorstCaseComplexityBoundsForNonConvexOptimizationWithInexpensiveConstraints}. More specific to the bound constraint case, the Newton-CG log barrier method of \citet{LogBarrierNewtonCGOneillWright} achieves a complexity of $\sO(d\epsilon_g^{-1/2} + \epsilon_g^{-3/2})$, while the projected Newton-CG algorithm of \citet{XieWright2021ComplexityOfProjectedNewtonCG} obtains a rate of $\sO(\epsilon_g^{-3/2})$ under a set of approximate optimality conditions similar to this work.

\textbf{Optimality Conditions}. Recall that $\xx_*$ satisfies the first-order necessary conditions for \cref{eqn:nonnegative constrained problem} if
\begin{align}
\label{eqn:nonnegative KKT conditions}
 \xx_* \geq \zero, \quad \text{and} \quad  \begin{cases}
        [\grad f(\xx_*)]^i = 0, \ &\text{if } \xx_*^i > 0, \\
        [\grad f(\xx_*)]^i \geq 0, \ &\text{if } \xx_*^i = 0.
    \end{cases}
\end{align}
We seek a point which satisfies these conditions to some ``$\epsilon$'' tolerance. There are a number of ways to adapt \eqref{eqn:nonnegative KKT conditions} into an approximate condition \citep[Section 3]{XieWright2021ComplexityOfProjectedNewtonCG}. In this work we adopt \citet[Definition 1]{XieWright2021ComplexityOfProjectedNewtonCG}.  
\begin{definition}[$\epsilon$-Optimal Point]\label{defn:optimal point}
A point, $\xx$, is called $\epsilon$-approximate first-order optimal ($\epsilon$-FO) if 
\begin{subequations}\label{eqn:epsilon optimality conditions}
    \begin{align}
    &\bgg^i \geq - \sqrt{\epsilon}, \quad \forall i \in \sA(\xx, \sqrt{\epsilon}) \tageq\label{eqn:active gradient negativity condition}\\
    &\| \diag(\xx^\sA) \bgg^\sA \| \leq \epsilon, \tageq\label{eqn:active gradient norm termination condition} \\
    &\| \bgg^\sI \| \leq \epsilon. \tageq\label{eqn:inactive set termination condition}
\end{align} 
\end{subequations}
We take \cref{eqn:active gradient negativity condition} and \cref{eqn:active gradient norm termination condition} to be trivially satisfied if $\sA(\xx, \sqrt{\epsilon}) = \emptyset$ and similar for \cref{eqn:inactive set termination condition} if $\sI(\xx, \sqrt{\epsilon}) = \emptyset$.
\end{definition} 

This definition has been shown to be asymptotically exact.
\begin{lemma}{\citep[Lemma 1]{XieWright2021ComplexityOfProjectedNewtonCG}}
    Suppose that $\epsilon_k \downarrow 0$ and we have a sequence $\{\xx_k\}_{k=1}^\infty$ where each $\xx_k$ satisfies the corresponding $\epsilon_k$-FO optimality condition. If $\xx_k \to \xx_*$ then $\xx_*$ satisfies \eqref{eqn:nonnegative KKT conditions}.
\end{lemma}

\section{Newton-MR Two-Metric Projection}
We now propose and theoretically study our extensions of the TMP framework, which involves simultaneously employing gradient and inexact Newton steps, which are, respectively, restricted to the active and inactive sets.

\subsection{MINRES and Its Properties}
The inexact Newton step is based on the recently proposed Newton-MR framework \cite{LiuNewtonMR,RoostaNewtonWithMinimumResidual}, where instead of CG, subproblems are approximately solved using the minimum residual (MINRES) method \citep{PaigeSaunders1975SolutionOfSparseIndefiniteSystemsOfLinearEquations}. Recall that the $t\th$ iteration of MINRES is formulated as 
\begin{align}
    \bss^{(t)} = \argmin_{\bss \in \sK_t(\HH, \bgg)}\| \HH\bss + \bgg \|^2. \label{eqn:MINRES subproblem}
\end{align}
where $\sK_t(\HH, \bgg) = \Span\{ \bgg, \HH\bgg, \ldots, \HH^{t-1} \bgg\}$ is the Krylov subspace of degree $t$ generated from $\HH$ and $\bgg$. On each iteration MINRES minimises the squared norm of the residual of the Newton system over the corresponding Krylov subspace. Note that, from an optimisation perspective, the residual itself can be viewed as the \textit{gradient} of the second-order Taylor approximation typically considered by second-order methods (e.g. Newton-CG), that is, $\rr \defeq -\HH\bss - \bgg = - \grad_\bss (\langle \bgg, \bss\rangle + \frac{1}{2} \langle \bss, \HH \bss \rangle)$. This highlights an advantage of MINRES over CG. Indeed, unlike CG, which aims to minimise the second order Taylor approximation, minimisation of the residual norm remains well defined even if $\HH$ is indefinite. For more theoretical and empirical comparisons between CG and MINRES, see  \citet{lim2024conjugate}.

Recently, \citet{LiuMinres} established several properties of MINRES that make it particularly well-suited for nonconvex settings. For example, to assess the availability of a nonpositive curvature (NPC) direction in MINRES, one merely needs to monitor the condition 
\begin{align*}
    \langle \rr^{(t-1)}, \HH \rr^{(t-1)} \rangle \leq 0, \tageq\label{eqn:NPC condition} 
\end{align*}
This condition is shown to be both necessary and sufficient for the existence of NPC directions in $\sK_t(\HH, \bgg)$ \citep[Theorem 3.3]{LiuMinres}. In addition, MINRES enjoys a natural termination condition in non-convex settings. More specifically, for any user specified tolerance $\eta > 0$, the termination condition
\begin{align}
    \| \HH \rr^{(t-1)} \| \leq  \eta \|\HH \bss^{(t-1)}\|, \tageq\label{eqn:MINRES termination tolerance}
\end{align}
is satisfied at some iteration. Note that the left hand side, $\HH \rr^{(t-1)}$, is simply the residual of the normal equation $\HH^{2} \bss = -\HH \bgg$. Condition \cref{eqn:MINRES termination tolerance} is particularly appealing in non-convex settings where we might have $\bgg \notin \range{(\HH)}$ and therefore $\|\rr\| > 0$ for all $\bss \in \real^d$. In this case a more typical termination condition $\| \rr^{(t-1)} \| \leq \eta$ may never be satisfied for a given $\eta > 0$. By contrast, \cref{eqn:MINRES termination tolerance} is applicable in all situations since $\| \HH \rr^{(t-1)} \| $ is guaranteed to monotonically decrease to zero, while $\| \HH \bss^{(t-1)}\|$ is monotonically increasing \citep[Lemma 3.11]{LiuInexactNewtonMR}. Remarkably, both Conditions \cref{eqn:NPC condition,eqn:MINRES termination tolerance} can be computed with a scalar update directly from the MINRES iterates without any additional Hessian-vector products; see \cref{lemma:MINRES scalar updates}. 

A Newton-MR step is computed by running MINRES until \cref{eqn:NPC condition} is detected, in which case $\rr^{(t-1)}$ is returned. Since $\rr^{(t-1)}$ is a nonpositive curvature direction, we label this case as a ``NPC'' step. Otherwise, when the termination condition \cref{eqn:MINRES termination tolerance} is satisfied, $\bss^{(t-1)}$ is returned. This step serves as an approximate solution to \cref{eqn:MINRES subproblem} and so we label this case as a ``SOL'' step. Let $ \pp $ denote the direction returned by negative curvature detecting MINRES. \citet{LiuMinres} shows that $ \pp $ serves as a direction of first and second-order descent for the function $ f $, namely $\dotprod{\pp, \bgg} < 0$ and $\dotprod{\pp,\bgg} + \dotprod{\pp,\HH \pp}/2 < 0$ \citep[Theorem 3.8]{LiuMinres}, as well as a direction of non-ascent for the norm of its gradient $ \vnorm{\bgg}^{2} $, that is, $\dotprod{\pp,\HH \bgg} < 0$ for a SOL step and $\dotprod{\pp,\HH \bgg} = 0$ for a NPC step \citep[Lemma 3.1]{LiuMinres}.

We include the full MINRES algorithm (\cref{alg:MINRES}) as well as some additional properties in \cref{apx:MINRES}.

\subsection{Global Convergence: Minimal Assumptions}
We first present a variant that is globally convergent under minimal assumptions. \cref{alg:newton-mr-two-metric-minimal-assumptions-case} is our simplest variant of the Newton-MR two-metric projection method. Recalling the definition of the $\delta_k$-active and $\delta_k$-inactive sets as in \cref{eqn:inactive and active set definition}\footnote{Note that \cref{eqn:inactive and active set definition} differs from \cref{eqn:bertsekas active and inactive set} as it does not include a gradient positivity condition. This helps with tractability of the global analysis but leads to a relatively smaller inactive set.}, \cref{alg:newton-mr-two-metric-minimal-assumptions-case} combines an active set gradient step (i.e., $\pp_k^\sA = - \bgg_k^{\sA}$) with an inactive set Newton-MR step (i.e., $\pp_k^\sI = \bss_k^\sI$, if $\SOL$, and $\pp_k^\sI = \rr_k^\sI$, if $\NPC$). In \cref{alg:newton-mr-two-metric-minimal-assumptions-case}, the curvature condition \cref{eqn:NPC condition} is considered with a positive tolerance, $\overline{\varsigma} = (d+1)\varsigma > 0$, i.e., $\langle \rr^{(t-1)}, \HH \rr^{(t-1)} \rangle \leq \overline{\varsigma} \| \rr ^{(t-1)}\|^2$. \cref{lemma:strong positive curvature certification} demonstrates that $\langle \rr^{(i)}, \HH \rr^{(i)} \rangle > \overline{\varsigma} \| \rr^{(i)}\|^2 $ for all $0 \leq i \leq t-1$ is a certificate that $\HH$ is $\varsigma$-strongly positive definite over $\sK_t(\HH, \bgg)$.

\begin{algorithm}[htbp]
    \begin{algorithmic}[1]
        \STATE \textbf{Input} Initial point $\xx_0 \geq \zero $, active set tol  $\{ \delta_k\}$, optimality tol  $\{ \epsilon_k \}$, MINRES inexactness tol  $\eta > 0$, NPC tol  $\overline{\varsigma} = (d+1)\varsigma$ for $ \varsigma > 0$, Line search parameter $\rho < 1/2$.
        \smallskip
        \FOR{$ k= 0, 1, \ldots$}
        \smallskip
        \STATE Update sets $\sA(\xx_k, \delta_k)$ and $\sI(\xx_k, \delta_k)$ as in \cref{eqn:inactive and active set definition}.
        \smallskip
        \IF{\cref{eqn:epsilon optimality conditions} is satisfied}
            \smallskip
            \STATE \textbf{Terminate}.
            \smallskip
        \ENDIF
        \smallskip
        \STATE $\ppk: \left\{\begin{array}{ll}
             \pp^{\sA}_k \gets - \bgg^{\sA}_k,&\\\\
            (\pp_k^\sI, \ \Dtype) \gets \text{MINRES}(\HH_k^\sI, \bgg_k^\sI, \eta, \overline{\varsigma})&
        \end{array}\right.$
        \smallskip
        \IF{ $\SOL$}
            \smallskip
            \STATE $\alpha_{k} \gets$ \cref{alg:back tracking line search} with $\alpha_0=1$ and \eqref{eqn:line search criteria}.
            \smallskip
        \ELSIF{$\NPC$}
            \smallskip
            \STATE $\alpha_{k} \gets$ \cref{alg:forward tracking line search} with $\alpha_0 = 1$ and \eqref{eqn:line search criteria}.
            \smallskip
        \ENDIF
        \STATE $\xx_{k+1} = \sP(\xx_k + \alpha_k \pp_k)$
        \ENDFOR
    \end{algorithmic}
    \caption{Newton-MR TMP (Minimal Assumptions)}
    \label{alg:newton-mr-two-metric-minimal-assumptions-case}
\end{algorithm}

Once the step direction is computed, the step size is selected with a line search criteria similar to that of \citet{ProjectionNewtonMethodsForOptimizationProblemsBertsekas}. Specifically, letting $\xx_k(\alpha) = \sP(\xx_k + \alpha\pp_k)$, we find $\alpha$ that, for some $\rho \in (0, 1/2)$, satisfies
\begin{equation}
    \begin{aligned}
        f(\xx_k(\alpha)) - f(\xx_k)\leq \rho \dotprod{\bgg_{k}^\sA, \sP(\xx_k^\sA + \alpha \pp_k^\sA) - \xx_k^\sA} \\ + \alpha \rho \dotprod{\bgg_{k}^\sI, \pp_k^\sI} ,
    \end{aligned} \tageq\label{eqn:line search criteria}
\end{equation}
Note that the term corresponding to the inexact set in \cref{eqn:line search criteria} is negative due to the descent properties of $\pp_k^\sI$ discussed earlier. On the other hand, the active set term in \cref{eqn:line search criteria} is negative due the descent properties of the gradient mapping \citep[Proposition 3.3.1]{bertsekasNonlinearProgramming}. This is crucial for our analysis as it allows us to consider the decrease in the inactive and active sets independently of each other. The two terms are unified since 
\begin{equation*}
    \begin{aligned}
    \langle \bgg_{k}, \sP(\xx_k + \alpha \pp_k) - \xx_k \rangle = \langle \bgg_{k}^\sA, \sP(\xx_k^\sA + \alpha \pp_k^\sA) - \xx_k^\sA \rangle \\ + \alpha \langle \bgg_{k}^\sI, \pp_k^\sI \rangle,
    \end{aligned} 
\end{equation*}
so long as $\alpha$ is chosen small enough. This is a direct consequence of $\sI(\xx_k, \delta_k)$ containing only \textit{strictly} feasible indices.

In \citet{LiuNewtonMR}, it was shown that when MINRES algorithm the returns an NPC step, the line search for $\alpha$ could run in a forward tracking mode (cf. \cref{alg:forward tracking line search}). In numerical experiments, it was demonstrated that the forward tracking line search was beneficial because it allowed for very large steps to be taken, particularly in flat regions where progress would otherwise be slow. Our theoretical analysis in \cref{apx:minimal assumptions convergence} demonstrates that a forward tracking line search can also be used in \cref{alg:newton-mr-two-metric-minimal-assumptions-case} for NPC type steps.

To analyse the global complexity of \cref{alg:newton-mr-two-metric-minimal-assumptions-case}, we only require  typical assumptions on Lipschitz continuity of the gradient and lower-boundedness of the objective.
\begin{assumption}
    \label{ass:LipschitzGradient}
    There exists $0 \leq L_{g} < \infty$ such that for all $\xx, \yy \in \real^d_+$, $\|\bgg(\xx) - \bgg(\yy) \| \leq L_{g} \|\xx - \yy\|.$
\end{assumption}
\begin{assumption}
    \label{ass:bounded below}
    We have $-\infty < f_* \leq f(\xx)$, $\forall \xx \in \real^d_+$.
\end{assumption}

With these minimal assumptions we can provide a guarantee of convergence of \cref{alg:newton-mr-two-metric-minimal-assumptions-case} in \cref{thm:minimal assumptions convergence theorem}, the proof of which we deferred to \cref{apx:minimal assumptions convergence}.
\begin{theorem}[Global Complexity of \cref{alg:newton-mr-two-metric-minimal-assumptions-case}]\label{thm:minimal assumptions convergence theorem}
    Let $\epsilon_g \in (0, 1)$ and $\varsigma > 0$. Under \cref{ass:LipschitzGradient,ass:bounded below}, if we choose $\delta_k =  \epsilon_k = \epsilon_g^{1/2}$ and $\overline{\varsigma}=(d+1)\varsigma$,   \cref{alg:newton-mr-two-metric-minimal-assumptions-case} produces an $\epsilon_g$-FO point in at most $\sO(\epsilon_g^{-2})$ iterations.  
\end{theorem}
\begin{remark}
    The ``big-$\sO$'' rate obtained in \cref{thm:minimal assumptions convergence theorem} hides a dependence on the problem constants and algorithm parameters $\rho$, $\varsigma$, $L_{g}$, $\eta$, which are, in particular, independent of $d$. However, the proof of \cref{thm:minimal assumptions convergence theorem} (and, indeed, \cref{thm:first-order iteration complexity}) implies that the worst case constant hidden by the big-$\sO$ notation could have an unfortunate dependence on the problem constants (e.g., $L_g^3$). This could suggest poor \textit{practical} performance despite the desirable dependence on $\epsilon_{g}$. However, as we show numerically in \cref{sec:numerical results}, such worst case analyses are rarely indicative of typical performance in practice.
\end{remark}

\subsection{Global Convergence: Improved Rate}
It is possible to modify \cref{alg:newton-mr-two-metric-minimal-assumptions-case} to improve upon the convergence rate of \cref{thm:minimal assumptions convergence theorem}, albeit under stronger assumptions.  
This is done in \cref{alg:newton-mr-two-metric} where, by appropriate use of curvature information, we can obtain an improved complexity rate. \cref{alg:newton-mr-two-metric} shares the same inactive/active sets, line search strategies, and projection based feasibility with \cref{alg:newton-mr-two-metric-minimal-assumptions-case}.There are, however, some main differences. A key distinction lies in the certification of \textit{strictly} positive curvature \cref{eqn:NPC condition} rather than strongly positive curvature, i.e., unlike \cref{alg:newton-mr-two-metric-minimal-assumptions-case} where we set $\overline{\varsigma} > 0$, in \cref{alg:newton-mr-two-metric} we set the NPC tolerance to $\overline{\varsigma} = 0$. Another notable difference is the introduction of \textit{Type II} steps. \textit{Type II} steps set the active portion of the step to zero and occur when the active set optimality conditions \cref{eqn:active gradient negativity condition,eqn:active gradient norm termination condition} are satisfied (otherwise \textit{Type I} steps, i.e., steps similar to \cref{alg:newton-mr-two-metric-minimal-assumptions-case}, are used) but the inactive set tolerance \cref{eqn:inactive set termination condition} is unsatisfied. Because the active set termination conditions are satisfied, removing the active portion of the step is not expected to significantly impede the algorithm's progress. By the same token, we can analyse \textit{Type II} steps using second-order curvature information, similar to the unconstrained Newton-MR algorithm, without having to account for the curvature related to the projected gradient portion of the step. Additionally, to achieve an improved rate over \cref{alg:newton-mr-two-metric-minimal-assumptions-case}, MINRES inexactness tolerance must scale with $\epsilon_k$ in  \cref{alg:newton-mr-two-metric}.

\begin{algorithm}[ht]
    \begin{algorithmic}[1]
       \STATE \textbf{Input} Initial point $\xx_0 \geq \zero $, active set tol  $\{ \delta_k\}$, optimality tol  $\{ \epsilon_k \}$, MINRES inexactness tol $\eta = \epsilon_k \theta$ and $\theta > 0$, Line search parameter $\rho < 1/2$, NPC tol  $\overline{\varsigma} = 0$.
        \smallskip
        \FOR{$ k= 0, 1, \ldots$}
        \smallskip
        \STATE Update sets $\sA(\xx_k, \delta_k)$ and $\sI(\xx_k, \delta_k)$ as in \cref{eqn:inactive and active set definition}.  
        \smallskip
        \IF{ $\sA(\xx_k, \delta_k) \neq \emptyset$ and (not \cref{eqn:active gradient negativity condition} or not \cref{eqn:active gradient norm termination condition}) }
            \smallskip
            \STATE Flag = \textit{Type I}.
            \smallskip
        \ELSIF{ $\sI(\xx_k, \delta_k) \neq \emptyset$ and not \cref{eqn:inactive set termination condition}}
            \smallskip
            \STATE Flag = \textit{Type II}.
            \smallskip
        \ELSE
            \smallskip
            \STATE \textbf{Terminate}.
            \smallskip
        \ENDIF
        \smallskip
        \STATE $\ppk: \left\{\begin{array}{ll}
             \pp^{\sA}_k \gets \left\{\begin{array}{ll}
             - \bgg^{\sA}_k,& \text{If Flag = \textit{Type I},}\\\\
            \zero, & \text{If Flag = \textit{Type II},}
        \end{array}\right. &\\\\
            \hspace{-1mm}(\pp_k^\sI, \ \Dtype) \gets \text{MINRES}(\HH_k^\sI, \bgg_k^\sI, \eta, \overline{\varsigma})&
        \end{array}\right.$
        \smallskip
        \IF{ $\SOL$}
            \smallskip
            \STATE $\alpha_{k} \gets$ \cref{alg:back tracking line search} with $\alpha_0=1$ and \eqref{eqn:line search criteria}.
            \smallskip
        \ELSIF{$\NPC$}
            \smallskip
            \STATE $\alpha_{k} \gets$ \cref{alg:forward tracking line search} with $\alpha_0=1$ and \eqref{eqn:line search criteria}.
            \smallskip
        \ENDIF
        \STATE $\xx_{k+1} = \sP(\xx_k + \alpha_k \pp_k)$
        \ENDFOR
    \end{algorithmic}
    \caption{Newton-MR TMP (Improved Rate)}
    \label{alg:newton-mr-two-metric}
\end{algorithm}

For our analysis, we need additional assumptions including the Lipschitz continuity of the Hessian.
\begin{assumption}
    \label{ass:LipschitzHessian}
    There exists $0 \leq L_{H} < \infty$ such that for all $\xx, \yy \in \real^d_+$, $\|\HH(\xx) - \HH(\yy) \| \leq L_{H} \|\xx - \yy\|$.
\end{assumption}
Additionally, we make some regularity assumptions on the output of the MINRES iterations.
\begin{assumption}\label{ass:ResidualLowerBoundRegularity}
    There exists a contant $\omega > 0 $, independent of $\xx$, such that the NPC direction from MINRES, $\pp = \rr^{(t-1)}$, satifies $\| \rr^{(t-1)} \| \geq \omega \| \bgg\|$.
\end{assumption}
We note that a lower bound for the relative residual is available directly \textit{prior to termination}. In fact, recall that if an NPC direction is returned, the termination condition \cref{eqn:MINRES termination tolerance} must not yet be satisfied. In this case, \cref{ass:LipschitzGradient,lemma:MINRES scalar updates} together imply that $\| \rr^{(t-1)} \| \geq \eta \| \bgg \|/\sqrt{\eta^2 + L_g^2}$. For \cref{alg:newton-mr-two-metric-minimal-assumptions-case}, this lower bound is directly utilised to establish convergence with no requirement for \cref{ass:ResidualLowerBoundRegularity}. However, for \cref{alg:newton-mr-two-metric}, $\eta$ depends on $\epsilon_k$, which could lead us to believe that the lower bound on the relative residual prior to termination does too. In particular, at first glance, this might suggest that the smaller the inexactness tolerance $\eta$, the more iterations MINRES is expected to perform before NPC detection. We argue that this is not the case. Firstly, an upper bound on the number of MINRES iterations until a NPC direction is encountered is \textit{independent} of the termination criteria $\eta$ \citet[Corollary 2]{LiuNewtonMR}. In fact, by construction, the MINRES iterates are independent of the termination tolerance $\eta$ and the magnitude of $\| \bgg\|$; see discussion and numerical examples around \citet[Assumption 4]{LiuNewtonMR}. Additionally, in the case where $\bgg \notin \range{(\HH)}$, we always have $\| \rr^{(t-1)} \| \geq \|(\eye - \HH\HH^{\dagger}) \bgg \|$, which is clearly independent of $\eta$. Together, these lines of argumentation constitute our justification for \cref{ass:ResidualLowerBoundRegularity}. 

Recall that \cref{alg:newton-mr-two-metric-minimal-assumptions-case} includes a manual verification of user specified strongly positive curvature over $\sK_t(\HH, \bgg)$ in $\SOL$ case, while \cref{alg:newton-mr-two-metric} only certifies strict positive curvature through the NPC condition \cref{eqn:NPC condition}. \citet{LiuMinres} demonstrated that as long as  the NPC condition \cref{eqn:NPC condition} has not been detected,  we have $\TT_t \succ \zero$ where $\TT_t \in \real^{t \times t}$ is the symmetric tridiagonal matrix obtained in the $t\th$ iteration of MINRES (see \cref{apx:MINRES} for more details). Our next assumption strengthens this notion. 
\begin{assumption}
    \label{ass:KrylovSubspaceRegularity}
    There exists $\sigma > 0$ such that for any $\xx_k$ in the sequence of SOL type iterates returned by \cref{alg:newton-mr-two-metric}, we have  $\TT_t  \succeq \sigma \eye$.  
\end{assumption}

\cref{ass:KrylovSubspaceRegularity} implies that, as long as the NPC condition \cref{eqn:NPC condition} has not been detected, for any $\vv \in \sK_t(\HH, \bgg)$ we have $\langle \vv, \HH \vv\rangle \geq \sigma \| \vv\|^2 $. \cref{ass:KrylovSubspaceRegularity} is satisfied by an objective function whose Hessian contains positive $\bgg$-relevant eigenvalues (eigenspaces not orthogonal to the gradient) uniformly separated from zero. A simple example is an under-determined least-squares problem. 

Together, \cref{ass:ResidualLowerBoundRegularity,ass:KrylovSubspaceRegularity} allow us to control the curvature of our step, which is necessary to obtain an improved rate over \cref{alg:newton-mr-two-metric-minimal-assumptions-case} using a Lipschitz Hessian upper bound. We now present the convergence result for \cref{alg:newton-mr-two-metric}. We defer the proof to \cref{sec:global convergence proofs}.
\begin{theorem}[Global Complexity of \cref{alg:newton-mr-two-metric}]\label{thm:first-order iteration complexity}
    Let $\epsilon_g \in (0, 1)$. Under \cref{ass:LipschitzGradient,ass:LipschitzHessian,ass:KrylovSubspaceRegularity,ass:ResidualLowerBoundRegularity,ass:bounded below}, if we choose $\delta_k =  \epsilon_k = \epsilon_g^{1/2}$,  Algorithm \ref{alg:newton-mr-two-metric} produces an $\epsilon_g$-FO point in at most $\sO(\epsilon_g^{-3/2})$ iterations.
\end{theorem}

\begin{remark}
A direct corollary to \cref{thm:first-order iteration complexity,thm:minimal assumptions convergence theorem}, under some mild additional assumptions, is a bound on the operational complexity in terms of gradient and Hessian-vector product evaluations. In particular, to produce a $\epsilon_g$-FO point, the operation complexity for \cref{alg:newton-mr-two-metric-minimal-assumptions-case,alg:newton-mr-two-metric} is,  respectively, $\sO(\epsilon_g^{-2})$ and $\tilde{\sO}(\epsilon_g^{-3/2})$; see  \cref{apx:operational complexity}.
\end{remark}

\begin{remark}
In all our algorithms, each step includes the Newton-MR component. The integration of the gradient and Newton-MR step is feasible in our algorithm due to the properties of the MINRES iterates (\cref{lemma:SOL type step properties,lemma:NPC type step properties}), allowing for a more flexible analysis with only first-order information. In contrast, it appears that second-order information is crucial for achieving descent with the capped-CG procedure, a central aspect of \citet{XieWright2021ComplexityOfProjectedNewtonCG}. This constraint prevents the algorithm from taking a step simultaneously comprised of gradient and Newton-CG components.    
\end{remark}

\subsection{Local Convergence}

An advantage of the original TMP method of \citet{ProjectionNewtonMethodsForOptimizationProblemsBertsekas} is that we get fast local convergence, a property that is shared by many Newton-type methods. We now show that our algorithm, in a slightly modified form, also exhibits this property. The basis for the local convergence is the fact that, under certain conditions, projected gradient algorithms are capable of identifying the true set of active constraints in a \textit{finite} number of iterations. This result was first establish for projected gradient with bound constraints in \citet{Bertsekas1976GradientProjection} but has been extended to a variety of constraints \citep{burke1988OnTheIdentificationofActiveConstraints,Burke1990OnTheIdentificationofActiveConstraintsII,Wright1993IdentifiableSurfaces,Sun2019AreWeThereYet}. In the case of two-metric projection, once the active set is identified, the combined step reduces to an unconstrained Newton step in the inactive set.

For the analysis, we consider a ``local phase'' variant of \cref{alg:newton-mr-two-metric-minimal-assumptions-case}. Specifically, we maintain flexibility in defining the outer and inner termination conditions and tolerances, eliminate the strongly positive curvature validation, and only perform backtracking line search from $\alpha_0=1$ to ensure the step length remains bounded. The pseudo-code for this local phase version is given in \cref{alg:newton-mr-two-metric-local-phase} for completeness. To show that the active set is identified in finite number of iterations, we need non-degeneracy and second-order sufficiency assumptions, which are standard in this context.

\begin{assumption}\label{ass:local minima nondegeneracy}
    A local minima, $\xx_*$, is non-degenerate if $[\bgg(\xx_*)]^i  > 0$, $\forall i \in \sA(\xx_*, 0)$.  
\end{assumption}
\begin{assumption}\label{ass:local minima second-order sufficiency}
    A local minima, $\xx_*$, satisfies the second-order sufficiency condition if $0 <  \dotprod{\zz, \HH(\xx_*)\zz}$ for all $\zz \neq 0$ such that $\zz^i = 0$ if $i \in \sA(\xx_*, 0)$.
\end{assumption}
\begin{theorem}[Active Set Identification]\label{thm:active set identification}
Let $f$ satisfy \cref{ass:LipschitzGradient} and $\xx_*$ be a local minima satisfying \cref{ass:local minima second-order sufficiency,ass:local minima nondegeneracy}. Let $\{\xx_k\}$ be the sequence of iterates generated by \cref{alg:newton-mr-two-metric-local-phase} with $\delta$ chosen according to \cref{eqn:local convergence delta choice}. There exists $\Delta_{\text{actv}} > 0$ such that if $\xx_{\bar{k}} \in B(\xx_*, \Delta_{\text{actv}})$, then $\sA(\xx_k, \delta) = \sA(\xx_k, 0) = \sA(\xx_*, 0)$ for all $k \geq \bar{k} + 1$. 
\end{theorem}
We defer the proof to \cref{apx:local convergence}. Once the active set is identified, our method reduces to unconstrained Newton-MR on the inactive set. Local convergence is therefore a simple corollary of \cref{thm:active set identification}. 
\begin{corollary}[Local Convergence]\label{cor:local convergence guarantee}
    For $k \geq \bar{k} + 1$ (cf. \cref{thm:active set identification}), the convergence of \cref{alg:newton-mr-two-metric-local-phase} is driven by the local properties of the Newton-MR portion of the step.
\end{corollary}

\smallskip
\begin{remark} \label{remark:local convergence} The local convergence of Newton-MR is similar to that of other inexact Newton methods. Suppose that we use a relative residual tolerance, $\|\rr^{\sI} \| \leq \eta \| \bgg^\sI \| $, as the criteria for the MINRES termination. Under \cref{ass:local minima second-order sufficiency}, we know that $\HH(\xx_*)$ is positive definite on the inactive indices. Therefore, by applying \citet[Theorem 7.1 and 7.2]{NocedalJorgeNO}, we obtain a superlinear convergence if we choose $\eta = \sO(1)$ and let $\xx_k$ be close enough to $\xx_*$. If we choose $\eta=\sO(\| \bgg_k \|) $ and the Hessian is Lipschitz then we can improve the rate to quadratic.
\end{remark}

\smallskip
\begin{remark}
    A central ingredient in the projected Newton-CG of \citet{XieWright2021ComplexityOfProjectedNewtonCG} is the damping of the Hessian in the form of diagonal perturbation (i.e., $\HH + \epsilon\eye$) for all Newton-CG steps in the inactive set. While this facilitates an optimal global complexity, an unfortunate consequence, at least in theory, is that the algorithm no longer enjoys a guaranteed fast ``Newton-type'' local convergence rate. In other words, one can at best show linear rates in local regimes.
\end{remark}

\section{Numerical Experiments} \label{sec:numerical results}
We now compare the performance of our method for solving \cref{eqn:nonnegative constrained problem} with several alternatives using various convex and non-convex examples. Specifically, we consider \cref{alg:newton-mr-two-metric} (denoted by \textbf{MR}), projected Newton-CG (denoted by \textbf{CG}) as in \citet[Algorithm 1]{XieWright2021ComplexityOfProjectedNewtonCG}, and projected gradient with line search (denoted by \textbf{PG}) \cite{beck2017FirstOrderMethodsInOptimization}. For convex problems, we also include \textbf{FISTA} with line search  \cite{Beck2009FISTA}, while for non-convex settings, we compare against the proximal gradient with momentum and fine-tuned constant step size (denoted by \textbf{PGM}) from \citet[Algorithm 4.1]{lin2020accelerated}. We exclude proximal Newton methods due to the difficulty of solving its subproblems at each iteration. We also do not consider the Newton-CG log barrier method \citep{LogBarrierNewtonCGOneillWright} due to poor practical performance observed in  \citet{XieWright2021ComplexityOfProjectedNewtonCG}.

For all applicable methods we terminate according to \cref{eqn:epsilon optimality conditions} with $\epsilon_g = 10^{-8}$. Instead of the highly implementation dependent ``wall-clock'' time, here we plot the objective value against the number of \textit{oracle calls}, i.e., the number of equivalent function evaluations. For completeness, however, we also include plots of objective value against wall-clock time in \cref{apx:timing results}. The PyTorch \citep{pytorch} implementation for our experiments is available \href{https://github.com/oscar-99/ProjectedNewton}{here}. All experiments were performed on a GPU cluster. See \cref{apx:parameter settings} for further experimental details.

\subsection{Sparse Regularisation With $\ell_1$ Norm}
We first consider sparse regression using $\ell_1$-regularisation
\begin{align*}
    \min_{\xx \in \real^d} f(\xx) + \lambda \| \xx \|_1, \tageq\label{eqn:L1 objective}
\end{align*}
where $f$ is a smooth function. Although the objective function in \cref {eqn:L1 objective} is nonsmooth, it can be reformulated into a smooth optimisation problem with nonnegativity constraints; see \cref{apx:smooth L1 regularisation} for details. We consider two examples in this context.

\textbf{Multinomial Regression}. In \cref{fig:CIFAR multinomial,fig:MNIST logistic regression}, we consider convex multinomial regression with $C$ classes where  $f$ is given by \cref{eqn:multinomial regression objective}. The FISTA method is  applied directly to \cref{eqn:L1 objective}. While FISTA clearly outperforms the others, our method is competitive. Further simulations showing fast local convergence of our method on these examples are given in \cref{apx:numerical local convergence}.

\begin{figure}[ht]
    \centering
    \includegraphics[scale=0.5]{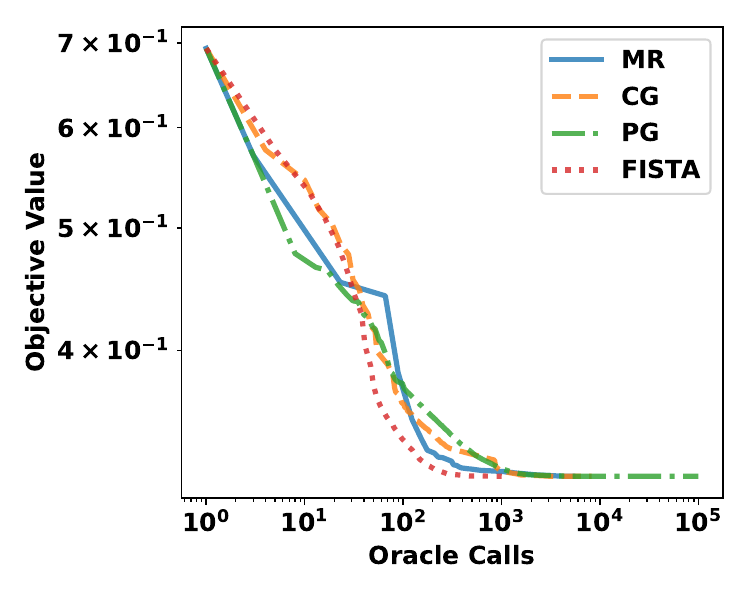}
    \caption{Logistic regression ($C=2$) on the binarised \texttt{MNIST} dataset \citep{Lecun1998MNIST} ($d=785$) with $\lambda = 10^{-3}$.}
    \label{fig:MNIST logistic regression}
\end{figure}

\begin{figure}[ht]
    \centering
    \includegraphics[scale=0.5]{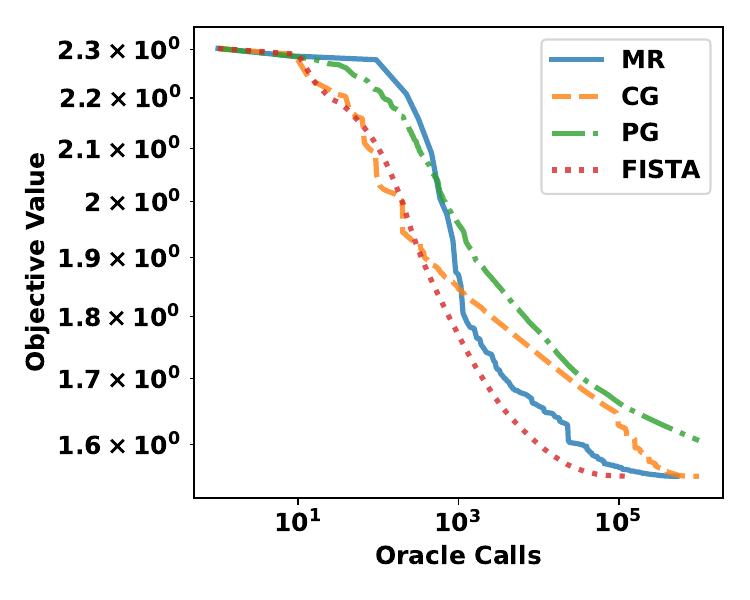}
    \caption{Multinomial regression ($C=10$) on \texttt{CIFAR10} dataset \citep{Krizhevsky2009CIFAR10} ($d=27,657$) with $\lambda = 10^{-4}$. }
    \label{fig:CIFAR multinomial}
\end{figure}

\textbf{Neural Network}. \cref{fig:MLP fashion MNIST} shows the results using a two layer neural network where $f$ is non-convex and defined by \cref{eqn:MLP objective function}. Again, PGM is applied directly to \cref{eqn:L1 objective} and its step size is fine-tuned for best performance. We once again observe superior performance of our method compared with the alternatives.

\begin{figure}
    \centering
    \includegraphics[scale=0.5]{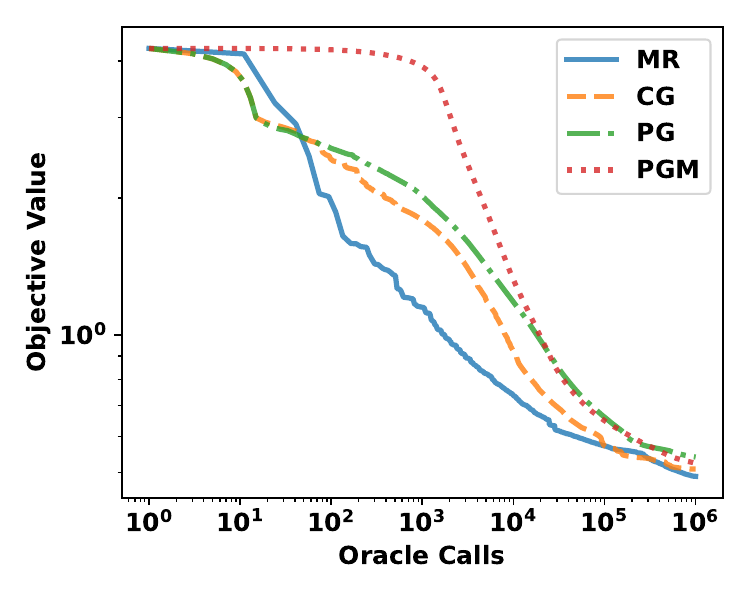}
    \caption{Training a two-layer neural network on the \texttt{Fashion MNIST} dataset \citep{xiao2017fashionmnist} ($d=89,610$) with $\lambda = 10^{-3}$.}
    \label{fig:MLP fashion MNIST}
\end{figure}

\subsection{Nonnegative Matrix Factorisation}

Given a nonnegative data matrix $\YY\in \real^{n \times m}_+$, nonnegative matrix factorisation (NNMF) aims to produce two low rank, say $r$, nonnegative matrices $\WW \in \real^{n \times r}_+$ and $\HH \in \real^{r \times m}_+$ such that $\YY \approx \WW \HH$. This can be formulated as
\begin{align*}
    \min_{\WW \geq 0, \ \HH \geq 0} D(\YY, \WW \HH)  + R_\lambda(\WW, \HH), \tageq\label{eqn:NNMF}
\end{align*}
where $D(\cdot, \cdot)$ is a `distance' and $R_\lambda(\cdot, \cdot)$ is a regularisation term. In \cref{fig:NNMF cosine text}, we consider a text dataset and cosine similarity based distance function, while in \cref{fig:NNMF nonconvex regularisation}, we use an image dataset and a Euclidean distance function with a nonconvex regulariser;  see \cref{apx:parameter settings} for details. Clearly, our method  outperforms all others across both problems.

\begin{figure}
    \centering
    \includegraphics[scale=0.5]{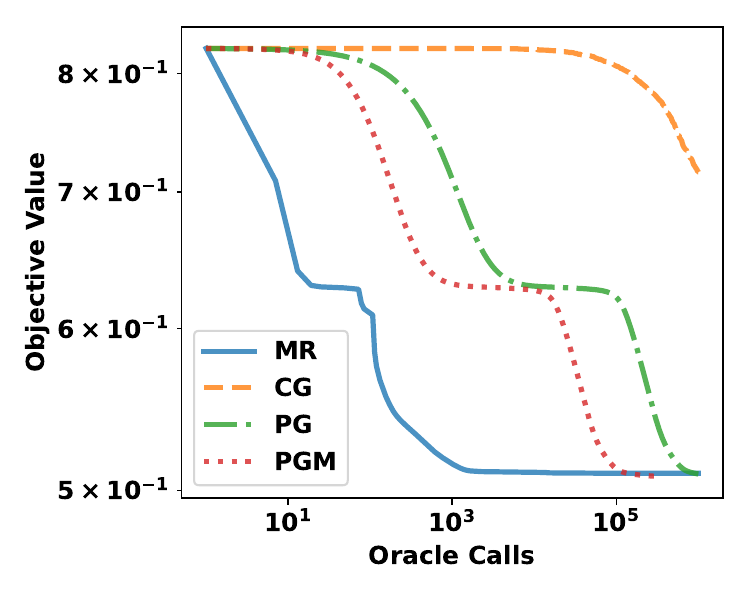}
    \caption{NNMF ($r=20$) with cosine distance on top 1000 TF-IDF features of the \texttt{20 Newsgroup} dataset \citep{20newsgroups} ($d=385,220$).}
    \label{fig:NNMF cosine text}
\end{figure}

\begin{figure}
    \centering
    \includegraphics[scale=0.5]{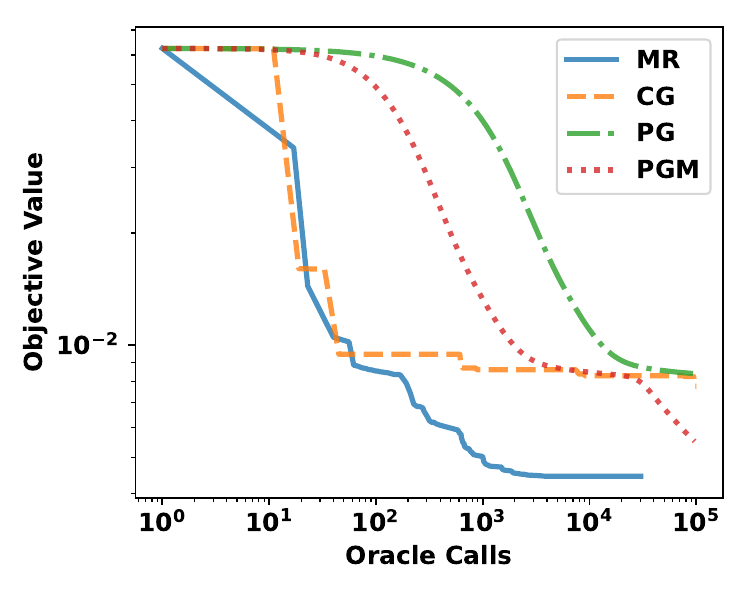}
    \caption{NNMF ($r=10$) with nonconvex TSCAD regulariser on the \texttt{Olivetti faces} dataset \citep{scikit-learn} ($d=44,960$). We used $a=3$ and $\lambda=10^{-4}$ for the TSCAD regulariser.}
    \label{fig:NNMF nonconvex regularisation}
\end{figure}

\section{Conclusions and Future Directions}
We developed Newton-MR variants of the two-metric projection framework. By inexactly solving the subproblems using MINRES as well as employing non-positive curvature directions, our proposed variants are suitable for large scale and nonconvex settings. We demonstrated that, under certain assumptions, the convergence rates of our methods match the state-of-the-art and showcased competitive practical performance across a variety of problems.

Possible avenues for future research include extensions to box constraints, variants with second-order complexity guarantees, and the  development of stochastic algorithms.

\section*{Acknowledgements}
This research was partially supported by the Australian Research Council through an Industrial Transformation Training Centre for Information Resilience (IC200100022).

\section*{Impact Statement}

This paper presents work whose goal is to advance the field of Machine Learning. There are many potential societal consequences of our work, none which we feel must be specifically highlighted here.

\bibliography{bibliography}
\bibliographystyle{icml2024}

%%%%%%%%%%%%%%%%%%%%%%%%%%%%%%%%%%%%%%%%%%%%%%%%%%%%%%%%%%%%%%%%%%%%%%%%%%%%%%%
%%%%%%%%%%%%%%%%%%%%%%%%%%%%%%%%%%%%%%%%%%%%%%%%%%%%%%%%%%%%%%%%%%%%%%%%%%%%%%%
% APPENDIX
%%%%%%%%%%%%%%%%%%%%%%%%%%%%%%%%%%%%%%%%%%%%%%%%%%%%%%%%%%%%%%%%%%%%%%%%%%%%%%%
%%%%%%%%%%%%%%%%%%%%%%%%%%%%%%%%%%%%%%%%%%%%%%%%%%%%%%%%%%%%%%%%%%%%%%%%%%%%%%%
\newpage
\appendix
\onecolumn

\section{MINRES and Newton-MR}\label{apx:MINRES}

In this section, for completeness, we discuss MINRES (\cref{alg:MINRES}) and provide some of its fundamental  properties. We note that our presentation is essentially that of \citet[Appendix A]{LiuNewtonMR} as the notation and implementation is well adapted to our setting. Recall that  MINRES  combines the Lanczos process, a QR decomposition, and an updating formula to iteratively solve a symmetric linear least-squares problem of the form 
\begin{align*}
    \min_{\bss \in \real^d } \| \HH \bss + \bgg \|^2.
\end{align*}
We now discuss each of these aspects in detail.

\paragraph{Lanczos Process.} Recall that, starting from $\vv_1 = \bgg/\| \bgg \| $, after $t$ iterations of the Lanczos process, the Lanczos vectors $\{\vv_1, \vv_{2}, \ldots, \vv_{t+1}\}$, form a basis for the Krylov subspace $\sK_{t+1}(\HH, \bgg)$. Collecting these vectors into an orthogonal matrix
\begin{align*}
    \VV_{t+1} = [\vv_1, \ldots \vv_{t+1}] \in \real^{d \times (t+1)},
\end{align*}
we can write 
\begin{align*}
    \HH \VV_t = \VV_{t+1} \tilde{\TT}_t,
\end{align*}
where $\tilde{\TT}_t \in \real^{(t+1), t}$ is an upper Hessenberg matrix of the form
\begin{align*}
    \TT_t = \begin{pmatrix}
        \talpha_1 & \tbeta_2 & & & \\
        \tbeta_2 & \talpha_2 & \tbeta_3 & & \\
        & \tilde{\beta}_3 & \talpha_3 & \ddots &  \\
        & & \ddots & \ddots  &  \tbeta_t   \\
        & & & \tbeta_t & \talpha_t 
    \end{pmatrix}, \ \tilde{\TT}_t \defeq \begin{pmatrix}
        \TT_t \\
        \tilde{\beta}_{t+1} \ee_t^\transpose
    \end{pmatrix}.
\end{align*}
This relation yields the underlying update process of the MINRES iterations for $t \geq 2$ as,
\begin{align*}
    \HH \vv_t = \tbeta_t \vv_{t-1} + \talpha_t \vv_t + \tbeta_{t+1} \vv_{t+1}.
\end{align*}
The Lanczos process terminates when $\tbeta_{t+1} = 0$. We remark that computing an expansion of the basis requires a single Hessian-vector product, $\HH \vv_t$. The basis for the Krylov subspace allows us to significantly simplify \cref{eqn:MINRES subproblem}. Indeed, let $\bss_t$ be a solution to \cref{eqn:MINRES subproblem} at iteration $t$. By $\bss_t \in \sK_t(\HH, \bgg)$, we have $\bss_t = \VV_t \yy_t$ for some $\yy_t \in \real^t$. Hence, the residual can be written as 
\begin{align*}
    \rr_t = - \bgg - \HH \bss_t = - \bgg - \HH \VV_t \yy_t = - \bgg - \VV_{t+1} \tilde{\TT}_t \yy_t = -\VV_{t+1}(\| \bgg \| \ee_1 + \tilde{\TT}_t \yy_t).  
\end{align*}
In the final equality, we applied the orthogonality of the basis vectors and $\vv_1 = \bgg/\| \bgg\|$. Applying this expression to \cref{eqn:MINRES subproblem} and using the orthogonality of $\VV_{t+1}$, we obtain the reduced tridiagonal least-squares problem 
\begin{align*}
    \min_{\yy_{t} \in \real^t} \vnorm{ \tbeta_1\ee_1 + \tTT_t \yy_t}, \tageq\label{eqn:tridiagonal least-squares problem}
\end{align*}
where $\tilde{\beta_1} = \| \bgg \| $.

\paragraph{QR Factorisation.} The next step in the MINRES procedure is to solve \cref{eqn:tridiagonal least-squares problem} by computing the full QR factorisation $\QQ_t \tTT_t = \tRR_t$ where $\QQ_t \in \real^{(t+1) \times (t+1)}$ and $\tRR_t \in \real^{(t+1) \times t}$. Because $\tTT_t$ is already close to being upper triangular, we form the QR factorisation using a series of Householder reflections to annihilate the sub-diagonal elements. Each Householder reflection affects only two rows of $\tTT_t$. We can summarise the effect of two successive Householder reflections for $3 \leq i \leq t-1$ as  
\begin{align*}
    \begin{pmatrix}
        1 & 0 & 0 \\
        0 & c_{i-1} & s_{i-1} \\
        0 & s_{i-1} & -c_{i-1} 
    \end{pmatrix}
    \begin{pmatrix}
        c_{i-2} & s_{i-2} & 0 \\
        s_{i-2} & -c_{i-2} & 0 \\
        0 & 0 & 1
    \end{pmatrix}
    \begin{pmatrix}
        \gamma_{i-2} & \delta_{i-1} & 0 & 0 \\
        \tbeta_{i-1} & \talpha_{i-1} & \tbeta_i & 0 \\
        0 & \tbeta_i & \talpha_i & \tbeta_{i+1}
    \end{pmatrix} \\
    = \begin{pmatrix}
        1 & 0 & 0 \\
        0 & c_{i-1} & s_{i-1} \\
        0 & s_{i-1} & -c_{i-1} 
    \end{pmatrix}
    \begin{pmatrix}
        \gamma_{i-2}^{[2]} & \delta_{i-1}^{[2]} & \epsilon_i & 0 \\
        0 & \gamma_{i-1} & \delta_i & 0 \\
        0 & \tbeta_i & \talpha_i & \tbeta_{i+1}
    \end{pmatrix} \\
    = \begin{pmatrix}
        \gamma_{i-2}^{[2]} & \delta_{i-1}^{[2]} & \epsilon_i & 0 \\
        0 & \gamma_{i-1}^{[2]} & \delta_i^{[2]} & \epsilon_{i+1} \\
        0 & 0 & \gamma_i & \delta_{i+1}
    \end{pmatrix}, \\
\end{align*}
where for $1 \leq j \leq t$ we have
\begin{align*}
    c_j = \frac{\gamma_j}{\gamma_j^{[2]}}, \quad s_j = \frac{\tbeta_{j+1}}{\gamma_j^{[2]}}, \quad \gamma_j^{[2]} = \sqrt{(\gamma_j)^2 + \tbeta_{j+1}^2} = c_j \gamma_j + s_j \tbeta_{j+1}.
\end{align*}
We therefore form $\QQ_t$ as a product of the Householder reflection matrices
\begin{align*}
    \QQ_t = \prod_{i=1}^t \QQ_{i, i+1}, \quad \QQ_{i, i+1} \defeq \begin{pmatrix}
        \eye_{i-1} & & &  \\
         & c_t & s_t & \\
         & s_t & - c_t & \\
         & & & \eye_{t-i}
    \end{pmatrix}.
\end{align*}
It is also clear that $\tRR_t$ is given by
\begin{align*}
    \RR_t \defeq \begin{pmatrix}
        \gamma_{1}^{[2]} & \delta_{2}^{[2]} & \epsilon_3 & &  \\
         & \gamma_{2}^{[2]} & \delta_3^{[2]} &  \ddots & \\
         &  & \ddots & \ddots & \epsilon_t \\ 
         & & & \gamma_{t-1}^{[2]} &  \delta_{t}^{[2]} \\
        & & & &  \gamma_{t}^{[2]}
    \end{pmatrix} , \quad \tRR_t = \begin{pmatrix}
        \RR_t \\
        \zero^\transpose
    \end{pmatrix}.
\end{align*}
Applying $\QQ_t$ to $\tbeta_1 \ee_1$, we obtain
\begin{align*}
    \QQ_t \tbeta_1 \ee_1 = \tbeta_1 \begin{pmatrix}
        c_1 \\
        s_1 c_2 \\ 
        \vdots \\
        s_1 s_2 \cdots s_{t-1} c_t \\
        s_{1} s_2 \cdots s_{t-1} s_t
    \end{pmatrix} \defeq
    \begin{pmatrix}
        \tau_1 \\
        \tau_2 \\
        \ldots 
        \tau_t \\
        \phi_t
    \end{pmatrix} \defeq
    \begin{pmatrix}
        \btt_t \\
        \phi_t
    \end{pmatrix}.
\end{align*}
Applying the QR factorisation to solve \cref{eqn:tridiagonal least-squares problem} gives
\begin{align*}
    \min_{\yy_t} \vnorm{ \tbeta_1 \ee_1 + \tTT_t \yy_t} &= \min_{\yy_t} \vnorm{ \QQ_t^\transpose ( \tbeta_1 \QQ_t \ee_1 + \QQ_t\tTT_t \yy_t)} \\
    &= \min_{\yy_t} \vnorm{\vvec{\btt_t}{\phi_t} + \vvec{\RR_t}{\zero^\transpose} \yy_t}.
\end{align*}
An immediate implication of this result is $\phi_t = \| \rr_t\| $.

\paragraph{Update.} The key to the computational efficiency of MINRES is the existence of vector update formula, which eliminates the requirement to form or store the matrices involved in the Lanczos and QR factorisation processes, i.e., $\VV_t$, $\QQ_t$, $\tRR_t$, and $\tTT_t$. Define $\WW_t$ from the upper triangular system $\WW_t \RR_t = \VV_t$ as 
\begin{align*}
    \begin{pmatrix}
        \vv_1 & \vv_2 & \ldots & \vv_t 
    \end{pmatrix} = \begin{pmatrix}
        \ww_1 & \ww_2 & \ldots & \ww_t 
    \end{pmatrix}
    \begin{pmatrix}
        \gamma_{1}^{[2]} & \delta_{2}^{[2]} & \epsilon_3 & &  \\
         & \gamma_{2}^{[2]} & \delta_3^{[2]} &  \ddots & \\
         &  & \ddots & \ddots & \epsilon_t \\ 
         & & & \gamma_{t-1}^{[2]} &  \delta_{t}^{[2]} \\
        & & & &  \gamma_{t}^{[2]}
    \end{pmatrix}. \tageq\label{eqn:RW = V system}
\end{align*}
By reading off \cref{eqn:RW = V system}, we see that 
\begin{align*}
    \vv_t = \epsilon_t \ww_{t-2} + \delta_t^{[2]} \ww_{t-1} + \gamma_t^{[2]} \ww_t.
\end{align*}
The computation for the MINRES iterate can now be written as
\begin{align*}
    \bss_t = \VV_t \yy_t = \WW_t \RR_t \yy_t  = \WW_t \btt_t = \hvec{\WW_{t-1}}{\ww_t} \vvec{\btt_{t-1}}{\tau_t} = \bss_{t-1} + \tau_t \ww_t, 
\end{align*}
where we set $\bss_0 = 0$. With this result in mind, we give the full MINRES method in \cref{alg:MINRES}. 
We remark that, in \cref{alg:MINRES}, we have also included steps for verifying the inexactness condition \cref{eqn:MINRES termination tolerance} (\cref{alg:MINRES}-\hyperref[eq:Inexactness_Step_Minres]{Line 10}) as well as certifying $\langle \rr_t, \HH \rr_t \rangle  \geq \vartheta \| \rr_t \|^2$ for some user specified $\vartheta \geq 0$ (\hyperref[eq:NOC_Step_Minres]{\cref{alg:MINRES}-Line 7}). 

\begin{algorithm}[ht]
    \begin{algorithmic}[1]
        \STATE \textbf{Input} Hessian $\HH$, gradient $\bgg$, inexactness tolerance $\eta > 0$, and NPC tolerance $\vartheta \geq 0$. 
        \STATE $\phi_0 = \tilde{\beta}_0 = \| \bgg\|$, $\rr_0 = - \bgg$, $\vv_1 = \rr_0/\phi_0$, $\vv_0 = \bss_0 = \ww_0 = \ww_{-1} = 0$.
        \STATE $s_0 = 0$, $c_0 = -1$, $\delta_1 = \tau_0 = 0$, $t=1$.
        \WHILE{ \text{True}}
            \STATE $\qq_t = \HH \vv_t$, $\tilde{\alpha}_t = \langle \vv_t, \qq_t \rangle$, $\qq_t = \qq_t - \tilde{\beta}_t \vv_{t-1}$, $\qq_t = \qq_t - \tilde{\alpha}_t \vv_t$, $\tilde{\beta}_{t+1} = \| \qq_{t} \| $.
            \STATE $ \begin{pmatrix}
                \delta_t^{[2]} & \epsilon_{t+1} \\
                \gamma_t & \delta_{t+1}
            \end{pmatrix} = \begin{pmatrix}
                c_{t-1} & s_{t-1} \\
                s_{t-1} & -c_{t-1}
            \end{pmatrix}\begin{pmatrix}
                \delta_t & 0 \\
                \tilde{\alpha}_t & \tilde{\beta}_{t+1}
            \end{pmatrix}$
            \IF{ $-c_{t-1} \gamma_t \leq \vartheta$} \label{eq:NOC_Step_Minres}
                \STATE \textbf{return} ($\rr_{t-1}$, $\text{D}_{\text{type}} = \text{NPC}$).
            \ENDIF
            \IF{$ \phi_{t-1} \sqrt{\gamma_t^2 + \delta_{t+1}^2} \leq \eta \sqrt{\phi_0^2 - \phi_{t-1}^2}$} \label{eq:Inexactness_Step_Minres}
                \STATE \textbf{return} ($\bss_{t-1}$, $\text{D}_{\text{type}} = \text{SOL}$).
            \ENDIF
            \STATE $\delta_t^{[2]} = \sqrt{\gamma_t^2 + \tilde{\beta}_{t+1}^2}$.
            \IF{ $\delta_t^{[2]} \neq 0$}
                \STATE $c_t = \gamma_t/\gamma_t^{[2]}$, $s_t = \tilde{\beta}_{t+1}/\gamma_t^{[2]}$, $\tau_t = c_t \phi_{t-1}$, $\phi_t = s_t \phi_{t-1}$.
                \STATE $\ww_t = (\vv_t - \gamma_t^{[2]} \ww_{t-1} - \epsilon_t \ww_{t-2})/\gamma_t^{[2]}$, $\bss_t = \bss_{t-1} + \tau_t \ww_t$.
                \IF{$\tilde{\beta}_{t+1} \neq 0$}
                    \STATE $\vv_{t+1} = \qq_t/\tilde{\beta}_{t+1}$, $\rr_t = s_t^2 \rr_{t-1} - \phi_t c_t \vv_{t+1}$.
                \ENDIF
            \ELSE
                \STATE $c_t = 0$, $s_t = 1$, $\tau_t =0$, $\phi_t = \phi_{t-1}$, $\rr_t = \rr_{t-1}$, $\bss_{t} = \bss_{t-1}$.
            \ENDIF
            \STATE $t \gets t+1$.
        \ENDWHILE
    \end{algorithmic}
    \caption{MINRES($\HH$, $\bgg$, $\eta$, $\vartheta$)}
    \label{alg:MINRES}
\end{algorithm}

We now collect several properties of the MINRES for reference; see  \citet{LiuMinres,LiuNewtonMR} for more  details and properties. Firstly, we give some scalar expressions for the quantities of interest in \cref{eqn:NPC condition,eqn:MINRES termination tolerance} in the MINRES algorithm

\begin{lemma}[MINRES scalar updates]\label{lemma:MINRES scalar updates}
    We have the following
    \begin{subequations}
    \begin{align}
        \| \rr^{(t)} \| &= \phi_t \label{eqn:r norm scalar} \\ 
        \langle \rr^{(t-1)}, \HH \rr^{(t-1)} \rangle &= -c_{t-1} \gamma_t \| \rr^{(t-1)}\|^2 , \label{eqn:r curvature scalar} \\ 
        \| \HH \bss^{(t-1)} \| &= \sqrt{\phi_0^2 - \phi_{t-1}^2}, \tageq\label{eqn:Hs norm scalar} \\
        \| \HH \rr^{(t-1)} \| &= \phi_{t-1} \sqrt{\gamma_t^2 + \delta_{t+1}^2}. \tageq\label{eqn:Hr norm scalar}
    \end{align}  
    \end{subequations}
\end{lemma}
\begin{proof}
    \cref{eqn:r norm scalar} follows from the construction of the MINRES algorithm. The proof of \cref{eqn:Hr norm scalar,eqn:r curvature scalar,eqn:Hs norm scalar} is given in \citet[Lemma 11]{LiuNewtonMR}.
\end{proof}

Next we give some helpful properties of the SOL and NPC steps.

\begin{lemma}[$\text{D}_{\text{type}} = \text{SOL}$] \label{lemma:SOL type step properties}
    Any iterate of MINRES, $\bss^{(t)}$, satisfies  
    \begin{align*}
        \| \HH \bss^{(t)} \| \leq \| \bgg \|, \tageq\label{eqn:Hessian step gradient bound}
    \end{align*}
    and 
    \begin{align*}
        \langle \bss^{(t)}, \HH\bgg  \rangle \leq 0. \tageq\label{eqn:descent for ||g||}
    \end{align*}
    Suppose that negative curvature has not been detected up to iteration $t$. Then, 
    \begin{align*}
        \langle \bss^{(t)} , \bgg \rangle \leq -\langle \bss^{(t)}, \HH \bss^{(t)} \rangle. \tageq\label{eqn:SOL MINRES descent curvature condition}
    \end{align*}
    Further, consider \cref{ass:LipschitzGradient} and suppose there exists some $\varrho > 0$ such that for any $\vv \in \sK_t(\HH, \bgg)$ we have $\langle \vv, \HH \vv\rangle \geq \varrho \| \vv \|^2 $. Then, 
    \begin{align*}
        C_{\varrho,L_g}   \| \bgg \|  \leq \| \bss^{(t)} \| \leq \frac{\| \bgg \|}{\varrho},  \tageq\label{eqn:SOL step gradient related}
    \end{align*}
    where $C_{\varrho,L_g}   \defeq \varrho/L_g^2$.
\end{lemma}

\begin{proof}
    The relation \cref{eqn:Hessian step gradient bound} follows from \citet[Lemma 3.11]{LiuInexactNewtonMR}, while \cref{eqn:descent for ||g||} follows from the fact that $\zero \in \sK_t(\HH, \bgg)$ and $\bss^{(t)}$ minimises \cref{eqn:MINRES subproblem}. Also, \cref{eqn:SOL MINRES descent curvature condition} follows from \citet[Theorem 3.8]{LiuMinres}. For the right-hand-side of \cref{eqn:SOL step gradient related}, we use   \cref{eqn:Hessian step gradient bound} and the fact that $\bss^{(t)} \in \sK_t(\HH, \bgg) $ to get 
    \begin{align*}
        \varrho \| \bss^{(t)} \|^2 \leq \langle \bss^{(t)}, \HH \bss^{(t)} \rangle \leq \|  \bss^{(t)}\| \| \HH \bss^{(t)}\|  \leq \| \bss^{(t)} \| \| \bgg \| \implies \| \bss^{(t)} \| \leq \| \bgg\|/\varrho.
    \end{align*}
    We show the left-hand-side of \cref{eqn:SOL step gradient related}  using a monotonicity argument. In particular, consider the first iterate $\bss^{(1)}$. It is easy to see that the solution to \cref{eqn:MINRES subproblem} over the Krylov subspace $\sK_1(\HH, \bgg) = \Span\{ \bgg \}$ is given by
    \begin{align*}
        \min_{\bss \in \sK_1(\HH, \bgg)}\| \HH \bss + \bgg \|^2 = \min_{\beta \in \real} \left\| \beta \HH \bgg + \bgg \right\|^2 \implies \beta = -\frac{\langle \bgg, \HH \bgg \rangle}{\| \HH \bgg \|^2}.
    \end{align*}
    The step is therefore given by 
    \begin{align*}
        \bss^{(1)} = -\frac{\langle \bgg, \HH \bgg \rangle}{\| \HH \bgg \|^2} \bgg.
    \end{align*}
    We can apply $\langle \bgg, \HH \bgg \rangle \geq \varrho \| \bgg \|^2 $ and $\|\HH \bgg \| \leq L_g \| \bgg \| $ to obtain
    \begin{align*}
        \| \bss^{(1)}\| &= \frac{\langle \bgg, \HH \bgg \rangle}{\| \HH \bgg \|^2} \| \bgg\| \\
        &\geq \frac{\varrho}{L_g^2} \| \bgg \|.
    \end{align*}
    The full results follows from the monotonicity of the MINRES iterates \citep[Theorem 3.11]{LiuMinres}, that is, as long as negative curvature remains undetected up to iteration $t \geq 1$ we have
    \begin{align*}
        \|\bss^{(t)}\|  \geq \| \bss^{(1)}\| \geq \frac{\varrho}{L_g^2} \| \bgg \| .
    \end{align*}
\end{proof}

\begin{lemma}[$\NPC$]\label{lemma:NPC type step properties}
    Suppose that the MINRES algorithm returns $\NPC$ so that our step is $\rr^{(t-1)}$. Then, 
    \begin{align*}
        \langle \rr^{(t-1)}, \bgg \rangle = - \left\| \rr^{(t-1)} \right\|^2. \tageq\label{eqn:NPC step descent}
    \end{align*}
    Additionally, the residual norm is upper bounded by the gradient.
    \begin{align*}
        \| \rr^{(t-1)} \| \leq \| \bgg \|. \tageq\label{eqn:NPC residual upper bound gradient}
    \end{align*}
\end{lemma}
\begin{proof}
    The relation \cref{eqn:NPC step descent} follows from the MINRES properties directly \citep[Lemma 3.1]{LiuMinres}. We get \cref{eqn:NPC residual upper bound gradient} by noting that
    \begin{align*}
        \|\HH \bss^{(t-1)} \|^2  &= \| \rr^{(t-1)}  + \bgg \|^2 = \| \rr^{(t-1)} \|^2 + 2\langle \rr^{(t-1)}, \bgg \rangle + \| \bgg \|^2 \\
        &= \| \rr^{(t-1)} \|^2 - 2\| \rr^{(t-1)} \|^2 + \| \bgg \|^2 = \| \bgg \|^2-\| \rr^{(t-1)} \|^2.
    \end{align*}
    For the third line, we applied \cref{eqn:NPC step descent}. The final equality and the nonnegativity of the norm implies the result. 
\end{proof}

\section{Global Convergence - Minimal Assumptions} \label{apx:minimal assumptions convergence}

In this section, we detail the proof of the global convergence of \cref{alg:newton-mr-two-metric-minimal-assumptions-case}, i.e., \cref{thm:minimal assumptions convergence theorem}. We first demonstrate that the uniform positive curvature certification of the residuals, $\rr^{(i)}$, provides a bound on the curvature of the Hessian over the corresponding Krylov subspace.
\begin{lemma}[Strong Positive Curvature Certification]\label{lemma:strong positive curvature certification}
    By verifying 
    \begin{align*}
        \langle \rr^{(t-1)} , \HH \rr^{(i)} \rangle  > \overline{\varsigma} \|\rr^{(i)}\|^2,
    \end{align*}
    for $i = 0, \ldots, t-1$, we obtain 
    \begin{align*}
        \langle \vv, \HH\vv \rangle \geq \overline{\varsigma}/(t+1)\|\vv\|^2, \tageq\label{eqn:minimal assumptions case curvature llower bound}
    \end{align*}
    for any $\vv \in \sK_t(\HH, \bgg)$.
\end{lemma}

\begin{proof}
    Let $\vv \in \sK_t(\HH, \bgg)$. We can write \citep[Lemma A.1]{LiuMinres}
    \begin{align*}
        \sK_t(\HH, \bgg) = \Span \left\{ \rr^{(0)}, \rr^{(1)}, \ldots, \rr^{(t-1)} \right\},
    \end{align*}
    and therefore there exists a set of scalars, $\{\beta_i\}_{i=0}^{t-1}$, such that
    \begin{align*}
        \vv = \sum_{i = 0}^{t-1} \beta_i \rr^{(i)}.
    \end{align*}
    Using this fact and the certificates $\langle \rr^{(i)}, \HH \rr^{(i)} \rangle \geq \overline{\varsigma} \| \rr^{(i)}\|^2$ gathered for $i=0, \ldots, t-1$, we obtain
    \begin{align*}
        \langle \vv, \HH \vv \rangle &= \left\langle \sum_{i = 0}^{t-1} \beta_i \rr^{(i)}, \sum_{i = 0}^{t-1} \beta_i \HH \rr^{(i)} \right\rangle = \sum_{i = 0}^{t-1} \sum_{j = 0}^{t-1} \beta_i \beta_j \left\langle   \rr^{(i)},   \HH \rr^{(j)} \right\rangle \\
        &= \sum_{i=0}^{t-1} \beta_i^2 \left\langle   \rr^{(i)},   \HH \rr^{(i)} \right\rangle \geq \sum_{i=0}^{t-1} \beta_i^2 \overline{\varsigma} \|\rr^{(i)}\|^2, \tageq\label{eqn:Krylov subspace lower bound}
    \end{align*}
    where the second to last equality follows from the $\HH$-conjugacy of the residuals \citep[Lemma 11]{LiuNewtonMR}. Using \citet[Fact 9.7.9]{Bernstein2009MatrixMathematics}, we get 
    \begin{align*}
         \frac{1}{t+1} \left\| \sum_{i=0}^{t-1}  \beta_i \rr^{(i)} \right\|^2 \leq  \sum_{i=0}^{t-1}\beta_i^2 \|\rr^{(i)}\|^2,
    \end{align*}
    which gives the desired result.
\end{proof}
Note that since $t$ appears in the lower bound \cref{eqn:minimal assumptions case curvature llower bound}, there is a dependence on the number of MINRES iterations undertaken and hence $\xx$. However, $t$ is bounded above by $d$. For this reason, in the sequel, we choose $\overline{\varsigma} = (d+1) \varsigma$ for some $\varsigma > 0$. Indeed, this choice implies that, under the conditions of \cref{lemma:strong positive curvature certification}, for any $\vv \in \sK_t(\HH, \bgg)$ we have
\begin{align*}
    \langle \vv, \HH \vv \rangle \geq \varsigma\| \vv\|^2. \tageq\label{eqn:minimal assumptions SOL curvature}
\end{align*}

We now demonstrate that the line search procedure \cref{eqn:line search criteria} terminates for a small enough step size.
\begin{lemma}[Step-size Lower Bound] \label{lemma:minimal assumptions line search termination}
Suppose $f$ satisfies  \cref{ass:LipschitzGradient}. If at iteration $k$ of \cref{alg:newton-mr-two-metric-minimal-assumptions-case}, we have $\sI(\xx_k, \delta_k) \neq \emptyset$, then the largest step size, $\alpha_k$, that satisfies the line search criteria \eqref{eqn:line search criteria}, also satisfies the following lower bound  
    \begin{align*}
        \alpha_k \geq \min \left\{\frac{2(1 - \rho)}{L_{g}} \min \{1, \varsigma \},   \frac{\delta_k}{\| \pp_k^\sI \|} \right\}.\tageq\label{eqn:minimal assumptions line search lower bound}
    \end{align*}
    On the other hand, if $\sI(\xx_k, \delta_k) = \emptyset$, the bound is given by
    \begin{align*}
        \alpha_k \geq \frac{2(1- \rho)}{L_g}. \tageq\label{eqn:minimal assumptions line search lower bound Ik empty}
    \end{align*}
\end{lemma}
\begin{proof}
    We note that the proof of \cref{lemma:type 1 line search termination} utilises no curvature properties of the residual. With this fact in mind, the proof is entirely the same as \cref{lemma:type 1 line search termination} in \cref{sec:global convergence proofs} with $\varsigma$ taking the place of $\sigma$.
\end{proof}

The following lemma gives the amount of decrease obtained from the inactive set step whenever the inactive set is nonempty and the inactive set termination condition \cref{eqn:inactive set termination condition} is not satisfied.

\begin{lemma}[Sufficient Decrease: Inactive Set Case] \label{lemma:minimal assumptions inactive set decrease}
    Suppose $f$ satisfies \cref{ass:LipschitzGradient}. Let $\xx_{k+1} = \sP(\xx_k + \alpha_k \pp_k)$ be the update computed at iteration $k$ of \cref{alg:newton-mr-two-metric-minimal-assumptions-case}, where $\alpha_k$ satisfies the line search criterion \cref{eqn:line search criteria}. Suppose $\sI(\xx_k, \delta_k) \neq \emptyset$ and \cref{eqn:inactive set termination condition} is not satisfied. If $\SOL$, then 
    \begin{align*}
        f(\xx_{k+1}) - f(\xx_k) &< - \rho \varsigma \min \left\{ \frac{2(1 - \rho)\min \{1, \varsigma \} C_{\varsigma,L_g}   ^2}{L_{g}} \epsilon_k^{4} , C_{\varsigma,L_g}    \delta_k \epsilon_k^{2} \right\},
    \end{align*}
    where $C_{\varsigma,L_g}$ is as in \cref{eqn:SOL step gradient related}.
    Otherwise, with $\NPC$, 
    \begin{align*}
        f(\xx_{k+1}) - f(\xx_k) &< - \rho \min \left\{\frac{2(1 - \rho)\eta^{2}}{L_{g}(\eta^2 + L_g^2)}  \epsilon_k^{4} , \frac{\eta \delta_k}{\sqrt{\eta^2 + L_g^2}} \epsilon_k^{2} \right\}.
    \end{align*}
\end{lemma}
\begin{proof}
    Since $\alpha_k$ satisfies the line search condition, we have
    \begin{align*}
        f(\xx_{k+1}) - f(\xx_k) \leq \rho \langle \bgg_{k}^\sA, \sP(\xx_k^\sA + \alpha_k \pp_k^\sA) - \xx_k^\sA \rangle + \alpha_k \rho \langle \bgg_{k}^\sI, \pp_k^\sI \rangle \leq \alpha_k \rho \langle \bgg_{k}^\sI, \pp_k^\sI \rangle, 
    \end{align*}
    where we use the fact that $\langle \bgg_{k}^\sA, \sP(\xx_k^\sA + \alpha \pp_k^\sA) - \xx_k^\sA \rangle \leq 0$. We now consider $\SOL$ and $\NPC$ cases.

    When $\SOL$,  we have $ \pp_k^\sI= \bss_k^\sI$. Using the line search condition, \cref{eqn:SOL MINRES descent curvature condition}, \cref{eqn:minimal assumptions SOL curvature}, \cref{eqn:minimal assumptions line search lower bound}, and the left-hand-side inequality in \cref{eqn:SOL step gradient related} with $\varrho = \varsigma$, we have
    \begin{align*}
        f(\xx_{k+1}) - f(\xx_k) 
        &\leq \rho \alpha_k \langle \bgg_{k}^\sI, \bss_k^\sI \rangle \\
        &\leq -\rho \alpha_k \langle \bss_k^\sI , \HH_k^\sI \bss_k^\sI \rangle \\
        &\leq -\rho \varsigma \alpha_k \| \bss_k^\sI \|^2 \\
        &\leq - \rho \varsigma \min \left\{\frac{2(1 - \rho)}{L_{g}} \min \{1, \varsigma \},   \frac{\delta_k}{\| \bss_k^\sI \|} \right\}  \| \bss_k^\sI \|^2  \\
        &\leq - \rho \varsigma \min \left\{\frac{2(1 - \rho)}{L_{g}} \min \{1, \varsigma \} \| \bss_k^\sI\|^2 ,  \delta_k \| \bss_k^\sI \| \right\}    \\
        &\leq - \rho \varsigma\min \left\{\frac{2(1 - \rho)\min \{1, \varsigma \} C_{\varsigma,L_g}   ^2  }{L_{g}}   \| \bgg_k^\sI \|^2 ,  C_{\varsigma,L_g}    \delta_k \| \bgg_k^\sI \| \right\}  \\
        &< - \rho \varsigma \min \left\{ \frac{2(1 - \rho)\min \{1, \varsigma \} C_{\varsigma,L_g}   ^2}{L_{g}} \epsilon_k^{4} , C_{\varsigma,L_g}    \delta_k \epsilon_k^{2} \right\},
    \end{align*}
    where we applied $\| \bgg_k^\sI \| > \epsilon_k^{2}$ on the final line.

    When $\NPC$,  we have, $\pp_k^\sI = \rr_k^\sI$. We first note that, since the inexactness condition \cref{eqn:MINRES termination tolerance} has not been met, by applying \cref{ass:LipschitzGradient} and using the fact that 
    \begin{align*}
        \|\HH \bss^{(t-1)} \|^2  = \| \bgg \|^2-\| \rr^{(t-1)} \|^2,
    \end{align*}
    we get 
    \begin{align*}
        \| \rr^{(t-1)} \| \geq \frac{\eta}{\sqrt{\eta^2 + L_g^2}} \| \bgg \|.
    \end{align*}
    Let $\omega = {\eta}/{\sqrt{\eta^2 + L_g^2}}$. Proceeding similarly to the SOL case but using \eqref{eqn:NPC step descent}, we have
    \begin{align*}
        f(\xx_{k+1}) - f(\xx_k) 
        &\leq \rho \alpha_k \langle \bgg_{k}^\sI, \rr_k^\sI \rangle \\
        &\leq - \rho \alpha_k \| \rr_k^\sI \|^2  \\
        &\leq - \rho \min \left\{\frac{2(1 - \rho)}{L_{g}}  \| \rr_k^\sI \|^2 ,  \delta_k \| \rr_k^\sI \|\right\} \\
        &\leq - \rho \min \left\{\frac{2(1 - \rho)\omega^2}{L_{g}}  \| \bgg_k^\sI \|^2 , \delta_k \omega \| \bgg_k^\sI \| \right\} \\
        &< - \rho \min \left\{\frac{2(1 - \rho)\omega^2}{L_{g}}  \epsilon_k^{4} ,\delta_k \omega \epsilon_k^{2} \right\},
    \end{align*}
    again, making use of $\| \bgg_k^\sI \| > \epsilon_k^{2}$ in the final line.
\end{proof}

The following lemma covers the case when the inactive set termination condition is satisfied, that is, $\sI(\xx_k, \delta_k) = \emptyset$ or \cref{eqn:inactive set termination condition} holds. In this case, we expect the inactive set step to be small (cf. \cref{eqn:SOL step gradient related}) and so we analyse the decrease due to the active set portion of the step, using the fact that at lease one of the active set termination conditions \cref{eqn:active gradient negativity condition} or \cref{eqn:active gradient norm termination condition} must be unsatisfied.

\begin{lemma}[Sufficient Decrease: Active Set Case] \label{lemma:minimal assumptions active set decrease}
    Suppose that $f$ satisfies \cref{ass:LipschitzGradient}. Let $\xx_{k+1} = \sP(\xx_k + \alpha_k \pp_k)$ be the update computed at iteration $k$ of \cref{alg:newton-mr-two-metric-minimal-assumptions-case}, where $\alpha_k$ satisfies the line search criterion \cref{eqn:line search criteria}. Suppose that at least one of the active set termination conditions, \cref{eqn:active gradient negativity condition} or \cref{eqn:active gradient norm termination condition}, is not satisfied. If $\sI(\xx_k, \delta_k) = \emptyset$, then
    \begin{align*}
        f(\xx_{k+1}) - f(\xx_k) < - \rho \min \left\{ \frac{1}{2}, \frac{2(1- \rho)}{L_g} \min\left\{ 1 ,   \frac{\epsilon_k^2}{2\delta_k^2} \right\} \right\}\epsilon_k^2.
    \end{align*}
    However, if $\sI(\xx_k, \delta_k) \neq \emptyset$ and \cref{eqn:inactive set termination condition} is satisfied, we have
    \begin{align*}
        f(\xx_{k+1}) - f(\xx_k) < - \rho \min \left\{ \frac{1}{2}, \min \left\{\frac{2(1 - \rho)}{L_{g}}, \frac{\delta_k}{ \epsilon_k^{2}} \right\}\min\{1, \varsigma\}  \min\left\{ 1 ,   \frac{\epsilon_k^2}{2\delta_k^2} \right\} \right\}\epsilon_k^2.
    \end{align*}
\end{lemma}
\begin{proof}
    Since $\alpha_k $ satisfies the line search criterion we have
    \begin{align*}
        f(\xx_{k+1}) - f(\xx_k) &\leq \rho \langle \bgg_{k}^\sA, \sP(\xx_k^\sA + \alpha_k \pp_k^\sA) - \xx_k^\sA \rangle + \alpha_k \rho \langle \bgg_{k}^\sI, \pp_k^\sI \rangle \\
        &\leq \rho \langle \bgg_{k}^\sA, \sP(\xx_k^\sA + \alpha_k \pp_k^\sA) - \xx_k^\sA \rangle,
    \end{align*}
    where we apply $\langle \bgg_{k}^\sI, \pp_k^\sI \rangle \leq 0$. From here the proof proceeds similarly to \cref{lemma:type 1 active set decrease}. Indeed, the if \cref{eqn:active gradient negativity condition} or \cref{eqn:active gradient norm termination condition} are unsatisfied, \cref{eqn:type 1 function decrease} gives 
    \begin{align*}
        f(\xx_{k+1}) - f(\xx_k) &< - \rho \min \left\{ \frac{1}{2}, \alpha_k \min\left\{ 1,   \frac{\epsilon_k^2}{2\delta_k^2} \right\} \right\} \epsilon_k^2, \tageq\label{eqn:minimal assumptions active set line search condition}
    \end{align*}
    and it only remains to apply a bound on $\alpha_k$. If $\sI(\xx_k, \delta_k) = \emptyset$, we use \cref{eqn:minimal assumptions line search lower bound Ik empty} to obtain 
    \begin{align*}
        f(\xx_{k+1}) - f(\xx_k) < - \rho \min \left\{ \frac{1}{2}, \frac{2(1- \rho)}{L_g} \min\left\{ 1 ,   \frac{\epsilon_k^2}{2\delta_k^2} \right\} \right\}\epsilon_k^2.
    \end{align*}
    Otherwise, we have $\sI(\xx_k, \delta_k) \neq \emptyset$. In this case, we must lower bound $\delta_k/\| \pp_k^\sI \|$ in \cref{eqn:minimal assumptions line search lower bound}. We therefore use \cref{eqn:SOL step gradient related} with $\varrho=\varsigma$, \cref{eqn:NPC residual upper bound gradient}, as well as the fact that \cref{eqn:inactive set termination condition} is unsatisfied to obtain 
    \begin{align*}
        \min\{\varsigma, 1 \}\| \pp_k^\sI \| \leq \| \bgg_k \| \leq \epsilon_k^2 \implies \frac{\delta_k\min\{1, \varsigma\} }{ \epsilon_k^{2}} \leq \frac{\delta_k }{\| \pp_k^\sI \| }.
    \end{align*}
    We now apply this bound to \cref{eqn:minimal assumptions line search lower bound} and combine with \cref{eqn:minimal assumptions active set line search condition} to get
    \begin{align*}
        f(\xx_{k+1}) - f(\xx_k) < - \rho \min \left\{ \frac{1}{2}, \min \left\{\frac{2(1 - \rho)}{L_{g}}, \frac{\delta_k}{ \epsilon_k^{2}} \right\}\min\{1, \varsigma\}  \min\left\{ 1 ,   \frac{\epsilon_k^2}{2\delta_k^2} \right\} \right\}\epsilon_k^2.
    \end{align*}
\end{proof}

\begin{proof}[Proof of \cref{thm:minimal assumptions convergence theorem}]
    We posit that the algorithm must terminate in at most 
    \begin{align*}
        K = \left\lceil \frac{(f_0 - f_*) \epsilon_g^{-2}}{\min\{c_1, c_2 \}} \right\rceil,
    \end{align*}
    iterations, where 
    \begin{align*}
        c_1 &\defeq \rho  \min \left\{ \frac{2 \varsigma (1 - \rho)\min \{1, \varsigma \} C_{\varsigma,L_g}^2}{L_{g}}, \varsigma C_{\varsigma,L_g},\frac{2(1 - \rho)\omega^2}{L_{g}}, \omega \right\}, \quad \text{with} \quad \omega \defeq \frac{\eta}{\sqrt{\eta^2 + L_g^2}},\\
        c_2 &\defeq \rho \min \left\{ \frac{1}{2}, \frac{1}{2}\min \left\{\frac{2(1 - \rho)}{L_{g}}, 1 \right\}\min\{1, \varsigma\} \right\},
    \end{align*}
    and $C_{\varsigma,L_g}$ is as in \cref{eqn:SOL step gradient related}. 
    Suppose otherwise, that is, the algorithm fails to terminate until at least iteration $K+1$. For iterations $k = 1, \ldots, K$, the termination conditions must be unsatisfied. We divide the iterates up in the following manner 
    \begin{align*}
        \sK_1 = \{ k \in [K] \ | \ \sI(\xx_k, \epsilon_g^{1/2}) \neq \emptyset, \ \| \bgg_k^\sI \| \geq \epsilon_g\},
    \end{align*}
    and 
    \begin{align*}
        \sK_2 = \{ k \in [K]\setminus \sK_1 \ | \ \sA(\xx_k, \epsilon_g^{1/2}) \neq \emptyset, \ (\exists i \in \sA(\xx_k, \epsilon_g^{1/2}),  \ \bgg_k^i < - \sqrt{\epsilon_g} \ \text{or} \ \| \diag(\xx_k^\sA) \bgg_k^\sA \| \geq \epsilon_g )\}.
    \end{align*}
    Since the algorithm has not terminated, $[K] = \sK_1 \cup \sK_2$. If $k \in \sK_1$ we apply \cref{lemma:minimal assumptions inactive set decrease} and combine the SOL and NPC cases with $\epsilon_g < 1$ to obtain
    \begin{align*}
        f(\xx_{k+1}) - f(\xx_k) &< - \rho  \min \left\{ \frac{2 \varsigma (1 - \rho)\min \{1, \varsigma \} C_{\varsigma,L_g}^2}{L_{g}}, \varsigma C_{\varsigma,L_g},\frac{2(1 - \rho)\omega^2}{L_{g}}, \omega \right\} \epsilon_g^2 = -c_1 \epsilon_g^2.
    \end{align*}
    If $k \in \sK_2$, we instead combine the results of \cref{lemma:minimal assumptions active set decrease} to obtain
    \begin{align*}
        f(\xx_{k+1}) - f(\xx_k) &< - \rho \min \left\{ \frac{1}{2}, \frac{1}{2}\min \left\{\frac{2(1 - \rho)}{L_{g}}, 1 \right\}\min\{1, \varsigma\} \right\}\epsilon_g \leq -c_2 \epsilon_g^2.
    \end{align*}
    Finally, we obtain 
    \begin{align*}
        f_0 - f_* &\geq f_0 - f(\xx_K) = \sum_{k=0}^{K-1} f(\xx_{k}) - f(\xx_{k+1}) > |\sK_1| c_1 \epsilon_g^2 + |\sK_2| c_2 \epsilon_g^2 \\ 
        &\geq (|\sK_1| + |\sK_2|) \min\{c_1, c_2 \}\epsilon_g^2 = K \min\{c_1, c_2 \}\epsilon_g^2, 
    \end{align*}
    which contradicts the definition of $K$.
\end{proof}

\section{Global Convergence - Improved Rate} \label{sec:global convergence proofs}

In this section, we provide the proof of  \cref{thm:first-order iteration complexity}. Recall that we denote the update to $\xx_k$ for some step size, $\alpha$, by  
\begin{align*}
    \xx_k(\alpha) = \sP(\xx_k + \alpha\pp_k).
\end{align*}
Recall that \cref{alg:newton-mr-two-metric} involves two types of steps: \textit{Type I} and \textit{Type II}. 
We summarise the step types, the optimality conditions, as well as the corresponding lemmas in \cref{table:summary}.
\begin{table}[htbp]
    \centering
    \caption{The step types, the optimality conditions, as well as the corresponding lemmas involved in the proof of \cref{thm:first-order iteration complexity}.\label{table:summary}}
    \bigskip
    
    \scriptsize
    \begin{tabular}{|c|c|c| c|c|c| }
        \hline
         Type & Termination condition & Active Step & Inactive Step & Step size & Sufficient Decrease  \\ \hline
         \textit{I} & $\sA\neq\emptyset$ and (not \eqref{eqn:active gradient negativity condition} or not \eqref{eqn:active gradient norm termination condition})  & Gradient & Newton-MR & \cref{lemma:type 1 line search termination} & \cref{lemma:type 1 active set decrease,lemma:type 1 inactive set decrease} \\ \hline
         
         \textit{II} & ($\sA = \emptyset$ or (\cref{eqn:active gradient negativity condition,eqn:active gradient norm termination condition})) and ($\sI \neq \emptyset$ and (not \eqref{eqn:inactive set termination condition}))  & None & Newton-MR & \cref{lemma:type 2 line search termination} & \cref{lemma:SOL alpha=1 gradient lower bound,lemma:type 2 sufficient decrease} \\ \hline
    \end{tabular}
     
\end{table}

Our first three lemmas (\cref{lemma:type 1 line search termination,lemma:type 1 active set decrease,lemma:type 1 inactive set decrease}) will demonstrate that \textit{Type I} steps produce sufficient decrease in the function value.  The analysis of \textit{Type I} steps builds off of \citet{XieWright2021ComplexityOfProjectedNewtonCG} which demonstrated that projected gradient can achieve good progress (in terms of guaranteed decrease) when the active termination conditions \cref{eqn:active gradient negativity condition,eqn:active gradient norm termination condition} are unsatisfied. However, unlike \citet{XieWright2021ComplexityOfProjectedNewtonCG}, which only uses a first-order step, we also incorporate second-order update in the form of Newton-MR step in the inactive set of indices. 

As shown in \cref{lemma:type 1 line search termination}, combining the steps in this manner suggests that the lower bound on the step size may depend inversely on the length of the Newton-MR step. This, in turn, could lead to small step sizes, if the Newton-MR step is large. We deal with this issue by splitting our analysis into two cases. The first case (\cref{lemma:type 1 inactive set decrease}) deals with large gradients on the inactive set where we expect good progress due to the corresponding large Newton-MR step on the inactive set (cf.\ \cref{eqn:SOL step gradient related}). By contrast, the second case (\cref{lemma:type 1 active set decrease}) deals with small gradients on the inactive set where we can expect to see small inactive set steps (cf.\ \cref{eqn:SOL step gradient related}) and therefore lower bounded step sizes. In this way, we trade off the convergence due to the inactive and active sets to always ensure sufficient decrease at the required rate.

Recall that \cref{ass:LipschitzGradient} implies that, for any $\yy, \xx \in \real^d_+$, we have
\begin{align*}
    f(\yy) \leq f(\xx) + \langle \grad f(\xx), \yy - \xx \rangle + \frac{L_g}{2} \| \xx -\yy\|^2. \tageq\label{eqn:Lipschitz gradient upper bound}
\end{align*}
We now give the proof of \cref{lemma:type 1 line search termination,lemma:type 1 active set decrease,lemma:type 1 inactive set decrease}.
\begin{lemma}[\textit{Type I} Step: Step-size Lower Bound] \label{lemma:type 1 line search termination} 
Assume that $f$ satisfies \cref{ass:LipschitzGradient,ass:KrylovSubspaceRegularity}. Suppose a \textit{Type I} step is taken at iteration $k$ of \cref{alg:newton-mr-two-metric}. If $\sI(\xx_k, \delta_k) \neq \emptyset$ ,then the largest step size which satisfies the line search criteria \eqref{eqn:line search criteria}, $\alpha_k$, satisfies the following lower bound  
    \begin{align*}
        \alpha_k \geq \min \left\{\frac{2(1 - \rho)}{L_{g}} \min \{1, \sigma \},   \frac{\delta_k}{\| \pp_k^\sI \|} \right\}.\tageq\label{eqn:type 1 line search lower bound}
    \end{align*}
    However, if $\sI(\xx_k, \delta_k) = \emptyset$, then
    \begin{align*}
        \alpha_k \geq \frac{2(1- \rho)}{L_g}. \tageq\label{eqn:type 1 line search lower bound Ik empty}
    \end{align*}
\end{lemma}
\begin{proof}
    If $\sI(\xx_k, \delta_k) \neq \emptyset$, suppose
    \begin{align*}
        \alpha \leq \frac{\delta_k}{\| \pp_k^\sI \|} \leq \frac{\delta_k}{\| \pp_k^\sI \|_\infty},
    \end{align*}
    so that for each $i \in \sI(\xx_k, \delta_k) $ we have $\sP(\xx_k^i + \alpha \pp_k^i) = \xx_k^i + \alpha \pp_k^i$. The Lipschitz gradient upper bound \cref{eqn:Lipschitz gradient upper bound} yields
    \begin{align*}
         f(\xx_k(\alpha))  &\leq f(\xx_k) + \langle \bgg_k, \sP(\xx_k + \alpha \pp_k) - \xx_k  \rangle + \frac{L_{g}}{2} \| P(\xx_k + \alpha \pp_k) - \xx_k \|^2  \\ 
        &= f(\xx_k) +  \langle \bgg_k^\sA, P(\xx_k^\sA - \alpha \bgg_k^\sA) - \xx_k^\sA \rangle  + \alpha \langle \bgg_k^\sI, \pp_k^\sI  \rangle + \frac{L_{g}}{2} \| \sP(\xx_k^\sA - \alpha \bgg_k^\sA) - \xx_k^\sA\|^2 +  \frac{L_{g} \alpha^2}{2} \| \pp_k^\sI \|^2.  
    \end{align*}
    It is clear from this bound that the line search will terminate for any $\alpha$ such that
    \begin{align*}
        \langle \bgg_k^\sA, P(\xx_k^\sA - \alpha \bgg_k^\sA) - \xx_k^\sA \rangle  + \alpha \langle \bgg_k^\sI, \pp_k^\sI  \rangle + \frac{L_{g}}{2} \| \sP(\xx_k^\sA - \alpha \bgg_k^\sA) - \xx_k^\sA\|^2 +  \frac{L_{g} \alpha^2}{2} \| \pp_k^\sI \|^2 \\
        - \rho \left( \langle \bgg_{k}^\sA, \sP(\xx_k^\sA - \alpha \bgg_k^\sA) - \xx_k^\sA \rangle + \alpha \langle \bgg_{k}^\sI, \pp_k^\sI \rangle  \right), \tageq\label{eqn:stage1-termination-condition}
    \end{align*}
    is nonpositive. Starting with the active set terms of \eqref{eqn:stage1-termination-condition}. We use the projection inequality $\| \sP(\xx) - \sP(\yy) \|^2 \leq \langle \xx - \yy, \sP(\xx) - \sP(\yy) \rangle$ combined with the feasibility of $\xx_k^\sA$ (which implies $\sP(\xx_k^\sA) = \xx_k^\sA$) to obtain
    \begin{align*}
        &(1- \rho) \langle \bgg_k^\sA, \sP(\xx_k^\sA - \alpha \bgg_k^\sA) - \xx_k^\sA \rangle + \frac{L_{g}}{2} \| P(\xx_k^\sA - \alpha \bgg_k^\sA) - \xx_k^\sA\|^2 \\
        &\leq (1- \rho) \langle \bgg_k^\sA, \sP(\xx_k^\sA - \alpha \bgg_k^\sA) - \xx_k^\sA \rangle - \frac{\alpha L_g}{2} \langle  \bgg_k^\sA,  \sP(\xx_k^\sA - \alpha \bgg_k^\sA) - \xx_k^\sA \rangle  \\
        &\leq \left( (1- \rho) - \frac{\alpha L_g}{2} \right) \langle \bgg_k^\sA, \sP(\xx_k^\sA - \alpha \bgg_k^\sA) - \xx_k^\sA \rangle.
    \end{align*}
    By $\langle \bgg_k^\sA, \sP(\xx_k^\sA - \alpha \bgg_k^\sA) - \xx_k^\sA \rangle \leq 0$, the active terms of \cref{eqn:stage1-termination-condition} are nonpositive if
    \begin{align*}
        (1- \rho) - \frac{\alpha L_g}{2} \geq 0 \implies \alpha \leq \frac{2(1 - \rho)}{L_g}. 
    \end{align*}
    If $\sI(\xx_k, \delta_k) = \emptyset$, then \eqref{eqn:type 1 line search lower bound Ik empty} follows directly from this bound.
    
    Now we consider the inactive terms of \eqref{eqn:stage1-termination-condition}. If $\SOL$, i.e., $\pp_k^\sI = \bss_k^\sI$, we apply \eqref{eqn:SOL MINRES descent curvature condition} and \cref{ass:KrylovSubspaceRegularity} to obtain
    \begin{align*}
        \alpha(1 - \rho) \langle \bgg_k^\sI, \bss_k^\sI \rangle + \frac{L_g \alpha^2}{2} \| \bss_k^\sI \|^2 &\leq \alpha \left( -(1-\rho) \langle \bss_k^\sI , \HH_k \bss_k^\sI\rangle + \frac{\alpha L_g }{2} \| \bss_k^\sI \|^2 \right)  \\
        &\leq \alpha \left( - (1-\rho) \sigma  \| \bss_k^\sI\|^2  + \frac{\alpha L_g }{2} \| \bss_k^\sI\|^2\right) \\
        &= \alpha \left( - (1-\rho) \sigma    + \frac{\alpha L_g }{2} \right)\| \bss_k^\sI \|^2.
    \end{align*}
    This upper bound will be negative for any step size satisfying
    \begin{align*}
        \alpha \leq \frac{2\sigma(1-\rho)}{L_g}.
    \end{align*}
    If $\NPC$, i.e., $\pp_k^\sI = \rr_k^\sI$, we apply \eqref{eqn:NPC step descent} to obtain
    \begin{align*}
        \alpha(1 - \rho) \langle \bgg_k^\sI, \rr_k^\sI \rangle + \frac{L_g \alpha^2}{2} \| \rr_k^\sI \|^2 &\leq -\alpha(1 - \rho)\| \rr_k^\sI\|^2 + \frac{L_g \alpha}{2}  \| \rr_k^\sI\|^2 \\
        &= \alpha \left( - (1 - \rho) + \frac{L_g \alpha}{2} \right)\| \rr_k^\sI\|^2,
    \end{align*}
    which is negative when 
    \begin{align*}
        \alpha \leq \frac{2(1- \rho)}{L_g}.
    \end{align*}
    If both the inactive and active terms of \cref{eqn:stage1-termination-condition} are nonpositive then the line search will certainly terminate. Collecting the bounds on the step size, we can see that the largest $\alpha_k$ which satisfies the line search criteria also satisfies the following lower bound
    \begin{align*}
        \alpha_k \geq \min \left\{\frac{2(1 - \rho)}{L_{g}} \min \{1, \sigma\},   \frac{\delta_k}{\| \pp_k^\sI \|}\right\}.
    \end{align*}
\end{proof}

\begin{lemma}[\textit{Type I} Step: Inactive Set Decrease] \label{lemma:type 1 inactive set decrease}
    Assume that $f$ satisfies \cref{ass:LipschitzGradient,ass:KrylovSubspaceRegularity,ass:ResidualLowerBoundRegularity}. Suppose that a \textit{Type I} step is taken at iteration k of \cref{alg:newton-mr-two-metric} but both $\sI(\xx_k, \delta_k) \neq \emptyset$ and $\| \bgg_k^\sI \| > \epsilon_k^{3/2} $. Let $\alpha_k$ be the largest step size satisfying the line search condition \eqref{eqn:line search criteria} so that $\xx_{k+1} = \sP(\xx_k + \alpha_k \pp_k )$. If $\SOL$ then
    \begin{align*}
        f(\xx_{k+1}) - f(\xx_k) < - \rho \sigma \min \left\{ \frac{2(1 - \rho)\min \{1, \sigma \} C_{\sigma,L_g}  ^2}{L_{g}} \epsilon_k^{3} , C_{\sigma,L_g}   \delta_k \epsilon_k^{3/2} \right\}.
    \end{align*}
    Otherwise, if $\NPC$,
    \begin{align*}
        f(\xx_{k+1}) - f(\xx_k) < -\rho \min \left\{\frac{2(1 - \rho)\omega^2}{L_{g}}  \epsilon_k^{3} , \omega \delta_k\epsilon_k^{3/2} \right\}.
    \end{align*}
\end{lemma}
\begin{proof}
    Line search criterion and the negativity of $\langle \bgg_{k}^\sA, \sP(\xx_k^\sA - \alpha_k \bgg_k^\sA) - \xx_k^\sA \rangle$ implies 
    \begin{align*}
        f(\xx_{k+1}) - f(\xx_k) &\leq \rho  \langle \bgg_{k}^\sA, \sP(\xx_k^\sA - \alpha_k \bgg_k^\sA) - \xx_k^\sA \rangle + \alpha_k \rho \langle \bgg_{k}^\sI, \pp_k^\sI \rangle  \leq \rho \alpha_k \langle \bgg_{k}^\sI, \pp_k^\sI \rangle. \tageq\label{eqn:type 1 inactive set line search upper bound}
    \end{align*}
    We now divide into two cases, depending on the step type selected by MINRES.
    
    If $\SOL$, then $ \pp_k^\sI= \bss_k^\sI$. Using the line search condition \cref{eqn:type 1 inactive set line search upper bound}, \cref{eqn:SOL MINRES descent curvature condition}, \cref{ass:KrylovSubspaceRegularity}, the lower bound on $\alpha_k$ from \cref{lemma:type 1 line search termination},  and the left-hand-side inequality of \cref{eqn:SOL step gradient related} with $\varrho=\sigma$, we have
    \begin{align*}
        f(\xx_{k+1}) - f(\xx_k) 
        &\leq \rho \alpha_k \langle \bgg_{k}^\sI, \bss_k^\sI \rangle \\
        &\leq -\rho \alpha_k \langle \bss_k^\sI , \HH_k^\sI \bss_k^\sI \rangle \\
        &\leq -\rho \sigma \alpha_k \| \bss_k^\sI \|^2 \\
        &\leq - \rho \sigma \min \left\{\frac{2(1 - \rho)}{L_{g}} \min \{1, \sigma \},   \frac{\delta_k}{\| \bss_k^\sI \|} \right\}  \| \bss_k^\sI \|^2  \\
        &\leq - \rho \sigma \min \left\{\frac{2(1 - \rho)}{L_{g}} \min \{1, \sigma \} \| \bss_k^\sI\|^2 ,  \delta_k \| \bss_k^\sI \| \right\}    \\
        &\leq - \rho \sigma \min \left\{\frac{2(1 - \rho)\min \{1, \sigma \} C_{\sigma,L_g}  ^2  }{L_{g}}   \| \bgg_k^\sI \|^2 ,  C_{\sigma,L_g}   \delta_k \| \bgg_k^\sI \| \right\}  \\
        &< - \rho \sigma \min \left\{ \frac{2(1 - \rho)\min \{1, \sigma \} C_{\sigma,L_g}  ^2}{L_{g}} \epsilon_k^{3} , C_{\sigma,L_g}   \delta_k \epsilon_k^{3/2} \right\},
    \end{align*}
    where for the last inequality, we used the fact that $\| \bgg_k^\sI \| > \epsilon_k^{3/2}$.
    
    If $\NPC$, then $\pp_k^\sI = \rr_k^\sI$. We use \cref{eqn:type 1 inactive set line search upper bound}, but apply \eqref{eqn:NPC step descent} and Assumption \ref{ass:ResidualLowerBoundRegularity} to get
    \begin{align*}
        f(\xx_{k+1}) - f(\xx_k) 
        &\leq \rho \alpha_k \langle \bgg_{k}^\sI, \rr_k^\sI \rangle \\
        &\leq - \rho \alpha_k \| \rr_k^\sI \|^2  \\
        &\leq - \rho \min \left\{\frac{2(1 - \rho)}{L_{g}}  \| \rr_k^\sI \|^2 ,  \delta_k \| \rr_k^\sI \|\right\} \\
        &\leq - \rho \min \left\{\frac{2(1 - \rho)\omega^2}{L_{g}}  \| \bgg_k^\sI \|^2 , \delta_k \omega \| \bgg_k^\sI \| \right\} \\
        &< - \rho \min \left\{\frac{2(1 - \rho)\omega^2}{L_{g}}  \epsilon_k^{3} ,\delta_k \omega \epsilon_k^{3/2} \right\},
    \end{align*}
    again, making use of $\| \bgg_k^\sI \| > \epsilon_k^{3/2}$ in the final line.
\end{proof}

\begin{lemma}[\textit{Type I} Step: Sufficient Reduction] \label{lemma:type 1 active set decrease}
    Assume that $f$ satisfies \cref{ass:LipschitzGradient,ass:KrylovSubspaceRegularity}. Suppose that a \textit{Type I} step is taken on iteration $k$ of \cref{alg:newton-mr-two-metric} so that $\sA(\xx_k, \delta_k) \neq \emptyset$ and either \cref{eqn:active gradient negativity condition} or \cref{eqn:active gradient norm termination condition} is unsatisfied. Let $\alpha_k$ be the largest step size satisfying the line search condition \eqref{eqn:line search criteria} so that $\xx_{k+1} = \sP(\xx_k + \alpha_k \pp_k )$. If $\sI(\xx_k, \delta_k) \neq \emptyset$ and $\| \bgg_k^\sI \| \leq \epsilon_k^{3/2}$, then
    \begin{align*}
        f(\xx_{k+1}) - f(\xx_k) <  - \rho \min \left\{ \frac{1}{2}, \min \{1, \sigma\} \min \left\{\frac{2(1 - \rho)}{L_{g}} , \frac{\delta_k}{ \epsilon_k^{3/2}} \right\} \min\left\{ 1 ,   \frac{\epsilon_k^2}{2\delta_k^2} \right\} \right\}\epsilon_k^2.
    \end{align*}
    Otherwise, if $\sI(\xx_k, \delta_k) = \emptyset$,
    \begin{align*}
        f(\xx_{k+1}) - f(\xx_k) < -\rho \min \left\{ \frac{1}{2},  \frac{2(1- \rho)}{L_g} \min \left\{ 1, \frac{\epsilon_k^2}{2\delta_k^2} \right\} \right\}\epsilon_k^2.
    \end{align*}
\end{lemma}
\begin{proof}
    Since $\alpha_k$ satisfies the line search sufficient decrease condition, the negativity of $\langle \bgg_{k}^\sI, \pp_k^\sI \rangle$, implied by \cref{eqn:SOL MINRES descent curvature condition,eqn:NPC step descent}, gives
    \begin{align*}
        f(\xx_{k+1}) - f(\xx_k) &\leq \rho \left( \langle \bgg_{k}^\sA, \sP(\xx_k^\sA - \alpha_k \bgg_k^\sA) - \xx_k^\sA \rangle + \alpha_k \langle \bgg_{k}^\sI, \pp_k^\sI \rangle \right) \\
        &\leq \rho\langle \bgg_{k}^\sA, \sP(\xx_k^\sA - \alpha \bgg_k^\sA) - \xx_k^\sA \rangle \\
        &= \rho \sum_{i \in \sA(\xx_k, \delta_k)} \bgg_k^i (\sP(\xx_k^i - \alpha \bgg_k^i) - \xx_k^i). \tageq\label{eqn:type 1 line search sum} 
    \end{align*}
    The analysis proceeds depending on which optimality condition is unsatisfied.
    
    Case 1 \cref{eqn:active gradient negativity condition}: $\bgg_k^i < -\epsilon_k$ for some $i \in \sA(\xx_k, \delta_k)$. In this case we can see that 
    \begin{align*}
        \bgg_k^i( \sP( \xx_k^i - \alpha_k \bgg_k^i) - \xx_k^i) = - \alpha_k (\bgg_k^i)^2  < - \alpha_k \epsilon_k^2.
    \end{align*}
    We immediately see from the term wise nonpositivity of \cref{eqn:type 1 line search sum} that 
    \begin{align*}
        f(\xx_{k+1}) - f(\xx_k) < -\rho \alpha_k\epsilon_k^2.
    \end{align*}
    Case 2 \cref{eqn:active gradient norm termination condition}: Continuing from \cref{eqn:type 1 line search sum} we obtain 
    \begin{align*}
        f(\xx_{k+1}) - f(\xx_k) &\leq \rho \sum_{i \in \sA(\xx_k, \delta_k)} \bgg_k^i (\sP(\xx_k^i - \alpha \bgg_k^i) - \xx_k^i) \\
        &= \rho \left( \sum_{\substack{i \in \sA(\xx_k, \delta_k) \ \\ \ \alpha_k \bgg_k^i\geq \xx_k^i }} -\bgg_k^i\xx_k^i +  \sum_{\substack{i \in \sA(\xx_k, \delta_k) \ \\ \ \alpha_k \bgg_k^i <  \xx_k^i} } -\alpha_k (\bgg_k^i)^2 \right). \tageq\label{eqn:type 1 line search upper bound}
    \end{align*}
    Note each sum in \cref{eqn:type 1 line search upper bound} is \textit{term-wise} negative. Since $\| \diag(\xx_k^\sA) \bgg_k^\sA\| > \epsilon_k^2$, we have 
    \begin{align*}
        \epsilon_k^4 &< \| \diag(\xx_k^\sA) \bgg_k^\sA \|^2 = \left( \sum_{\substack{i \in \sA(\xx_k, \delta_k) \ \\ \ \alpha_k \bgg_k^i\geq \xx_k^i} } (\bgg_k^i\xx_k^i)^2 +  \sum_{\substack{i \in \sA(\xx_k, \delta_k) \ \\ \ \alpha_k \bgg_k^i <  \xx_k^i} } (\xx_k^i\bgg_k^i)^2 \right).
    \end{align*}
    This implies two possible cases: either
    \begin{align*}
        \frac{\epsilon_k^4}{2} < \sum_{\substack{i \in \sA(\xx_k, \delta_k) \ \\ \ \alpha \bgg_k^i \geq \xx_k^i} } (\xx_k^i \bgg_k^i)^2 \implies \frac{\epsilon_k^2}{2} < \sum_{\substack{i \in \sA(\xx_k, \delta_k) \ \\ \ \alpha \bgg_k^i \geq \xx_k^i} } \xx_k^i \bgg_k^i,
    \end{align*}
    or
    \begin{align*}
        \frac{\epsilon_k^4}{2} < \sum_{\substack{i \in \sA(\xx_k, \delta_k)  \ \\ \  \alpha \bgg_k^i < \xx_k^i} } (\xx_k^i \bgg_k^i)^2 \leq \sum_{\substack{i \in \sA(\xx_k, \delta_k)  \ \\ \  \alpha \bgg_k^i < \xx_k^i} } (\delta_k \bgg_k^i)^2 \implies \frac{\epsilon_k^4}{2\delta_k^2} < \sum_{\substack{i \in \sA(\xx_k, \delta_k)  \ \\ \  \alpha \bgg_k^i < \xx_k^i} } (\bgg_k^i)^2.
    \end{align*}
    In either case, the negativity of each term of \cref{eqn:type 1 line search upper bound} implies  
    \begin{align*}
        f(\xx_{k+1}) - f(\xx_k) < - \rho \min \left\{ \frac{\epsilon_k^2}{2},  \frac{\alpha_k\epsilon_k^4}{2\delta_k^2} \right\} .
    \end{align*}
    Combining with Case 1 gives 
    \begin{align*}
        f(\xx_{k+1}) - f(\xx_k) &< - \rho \min \left\{ \frac{\epsilon_k^2}{2}, \alpha_k \epsilon_k^2,   \frac{\alpha_k\epsilon_k^4}{2\delta_k^2} \right\} \\
        &= - \rho \min \left\{ \frac{1}{2}, \alpha_k ,   \frac{\alpha_k\epsilon_k^2}{2\delta_k^2} \right\}\epsilon_k^2. \tageq\label{eqn:type 1 function decrease}
    \end{align*}
    If $\sI(\xx_k, \delta_k) = \emptyset$, we apply \eqref{eqn:type 1 line search lower bound Ik empty} to obtain 
    \begin{align*}
        f(\xx_{k+1}) - f(\xx_k) < - \rho \min \left\{ \frac{1}{2},  \frac{2(1- \rho)}{L_g} \min \left\{ 1, \frac{\epsilon_k^2}{2\delta_k^2} \right\} \right\}\epsilon_k^2.
    \end{align*}

    On the other hand, if $\sI(\xx_k, \delta_k) \neq \emptyset$, the lower bound for $\alpha_k$ in \eqref{eqn:type 1 line search lower bound} depends inversely on the inactive portion of the step $\| \pp_k^\sI\| $. The step size can therefore become small if $\| \pp_k^\sI \|$ is too large. To avoid this, we will make use of the fact that the gradient is bounded. In particular, by combining the right inequality of \eqref{eqn:SOL step gradient related} and \eqref{eqn:NPC residual upper bound gradient}, we obtain
    \begin{align*}
        \min\{1, \sigma\} \| \pp_k^\sI \| \leq \| \bgg_k^\sI \| \leq  \epsilon_k^{3/2},
    \end{align*} 
    which implies   
    \begin{align*}
          \frac{\delta_k\min\{1, \sigma\} }{ \epsilon_k^{3/2}} \leq \frac{\delta_k }{\| \pp_k^\sI \| }.
    \end{align*}
    Imposing this on the step size lower bound \eqref{eqn:type 1 line search lower bound} gives
    \begin{align*}
        \alpha_k \geq \min \{1, \sigma \} \min \left\{\frac{2(1 - \rho)}{L_{g}} , \frac{\delta_k}{\epsilon_k^{3/2}} \right\}.
    \end{align*}
    The decrease is therefore given by
    \begin{align*}
        f(\xx_{k+1}) - f(\xx_k) &< - \rho\min \left\{ \frac{1}{2}, \alpha_k ,   \frac{\alpha_k\epsilon_k^2}{2\delta_k^2} \right\}\epsilon_k^2 \\
        &\leq - \rho\min \left\{ \frac{1}{2}, \alpha_k \min\left\{ 1 ,   \frac{\epsilon_k^2}{2\delta_k^2} \right\} \right\}\epsilon_k^2 \\
        &\leq - \rho \min \left\{ \frac{1}{2}, \min \{1, \sigma\} \min \left\{\frac{2(1 - \rho)}{L_{g}} , \frac{\delta_k}{ \epsilon_k^{3/2}} \right\} \min\left\{ 1 ,   \frac{\epsilon_k^2}{2\delta_k^2} \right\} \right\}\epsilon_k^2.
    \end{align*} 
\end{proof}

The next three lemmas (\cref{lemma:type 2 line search termination,lemma:SOL alpha=1 gradient lower bound,lemma:type 2 sufficient decrease}) demonstrate the sufficient decrease of \textit{Type II} steps. Recall that a \textit{Type II} steps occurs once active set optimality is reached. \textit{Type II} steps are taken until the inactive set optimality \cref{eqn:inactive set termination condition} is satisfied (termination) or a new index falls into the active set and disrupts active set optimality, in which case we resume \textit{Type I} steps. The \textit{Type II} step consists of only a Newton-MR step in the inactive indices (no step is taken in the active indices). Indeed, a \textit{Type II} direction can be written (with possible reordering of indices) as  
\begin{align*}
    \xx_k(\alpha) -\xx_k = \vvec{0}{\sP(\xx_k^\sI + \alpha \pp_k^\sI) - \xx_k^\sI}.
\end{align*}
Eliminating the active portion of the step allows us to leverage a ``second-order analysis'' of the inactive indices without having to account for the curvature of the projected gradient portion of the step. Indeed, the analysis of the algorithm reverts to essentially that of unconstrained Newton-MR \citep{LiuNewtonMR}, with some minor modifications to account for the projection. Specifically, with possible reordering of the indices, we partition the Hessian into four blocks as 
\begin{align*}
    \HH_k = \begin{pmatrix}
        \HH_k^\sA  & \HH_k^O \\
        \HH_k^O & \HH_k^\sI
    \end{pmatrix},
\end{align*}
where $\HH_k^\sA$ and $\HH_k^\sI$ are the sub matrices corresponding to the active and inactive indices respectively and $\HH_k^O$ is the remaining off diagonal blocks of the Hessian. Under the Lipschitz Hessian condition (\cref{ass:LipschitzHessian}) and using $\alpha \leq \delta_k/\| \pp_k^\sI \| $ so that $\sP(\xx_k^\sI + \alpha \pp_k^\sI) = \xx_k^\sI + \alpha \pp_k^\sI$, we can write 
\begin{align*}
    f(\xx_k(\alpha)) &\leq f(\xx_k) +  \left \langle \begin{pmatrix}
        \bgg^\sA_k\\
        \bgg^\sI_k
    \end{pmatrix}, \begin{pmatrix}
        0 \\
        \alpha \pp_k^\sI
    \end{pmatrix} \right\rangle + \frac{1}{2}\left\langle \begin{pmatrix}
        0\\
        \alpha \pp_k^\sI
    \end{pmatrix} , \begin{pmatrix}
        \HH_k^\sA  & \HH_k^O \\
        \HH_k^O & \HH_k^\sI
    \end{pmatrix}  \begin{pmatrix}
        0\\
        \alpha\pp_k^\sI
    \end{pmatrix} \right \rangle  +  \frac{\alpha^3 L_{H}}{6} \left \| \begin{pmatrix}
        0\\
        \pp_k^\sI
    \end{pmatrix} \right\|^3 \\
    &= f(\xx_k) + \alpha \langle \bgg_k^\sI, \pp_k^\sI \rangle + \frac{\alpha^2}{2} \langle \pp_k^\sI, \HH_k^\sI \pp_k^\sI \rangle + \frac{\alpha^3 L_{H} }{6} \| \pp_k^\sI\|^3. \tageq\label{eqn:lipschitz upper bound of type 2 step}
\end{align*}
Our first lemma uses the expansion in \cref{eqn:lipschitz upper bound of type 2 step} to show that the largest step size satisfying the line search criterion is lower bounded.

\begin{lemma}[\textit{Type II} Step: Step-size Lower Bound] \label{lemma:type 2 line search termination}
    Assume that $f$ satisfies \cref{ass:LipschitzHessian}. If Algorithm \ref{alg:newton-mr-two-metric} selects a \textit{Type II} step at iteration $k$ and MINRES returns $\NPC$, then for the largest step size, $\alpha_k$, satisfying the line search criterion \cref{eqn:line search criteria}, we must have
    \begin{align*}
        \alpha_k \geq \min\left\{\sqrt{\frac{6 (1-\rho)}{ L_{H} \| \rr_k^\sI \| }}, \frac{\delta_k}{\| \rr_k^\sI \|}  \right\}. \tageq\label{eqn:type 2 NPC line search termination}
    \end{align*}
    Otherwise, if $\SOL$ and Assumption \ref{ass:KrylovSubspaceRegularity} holds, then 
    \begin{align*}
        \alpha_k \geq \min\left\{1, \sqrt{\frac{3 \sigma(1 - 2 \rho)}{L_{H} \| \bss_k^\sI \| } }, \frac{\delta_k}{\| \bss_k^\sI \|}  \right\}. \tageq\label{eqn:type 2 SOL line search termination}
    \end{align*}
\end{lemma}
\begin{proof}
    We have already seen that, if $\alpha \leq \delta_k/\| \pp_k^\sI \| $, \eqref{eqn:lipschitz upper bound of type 2 step} holds. From \cref{eqn:line search criteria}, the line search is satisfied for any $\alpha$ such that 
    \begin{align*}
        f(\xx_k(\alpha)) - f(\xx_k) - \rho \alpha \langle \bgg_k^\sI, \pp_k^\sI \rangle \leq 0.
    \end{align*}
    We now consider $\SOL$ and $\NPC$ cases. Let $\SOL$ so that $\pp_k^\sI = \bss_k^\sI$. Applying \eqref{eqn:lipschitz upper bound of type 2 step}, $\alpha \leq 1$, the MINRES curvature condition \eqref{eqn:SOL MINRES descent curvature condition} and \cref{ass:KrylovSubspaceRegularity} we have 
    \begin{align*}
        f(\xx_k(\alpha)) - f(\xx_k) - \rho \alpha \langle \bgg_k, \bss_k^\sI \rangle  
        &\leq \alpha \langle \bgg_k^\sI, \bss_k^\sI \rangle + \frac{\alpha^2}{2} \langle \bss_k^\sI, \HH_k^\sI \bss_k^\sI \rangle + \frac{\alpha^3 L_{H} }{6} \| \bss_k^\sI\|^3 - \rho \alpha \langle \bgg_k, \bss_k^\sI \rangle \\ 
        &\leq \alpha (1 - \rho) \langle \bgg_k^\sI, \bss_k^\sI \rangle + \frac{\alpha}{2} \langle \bss_k^\sI, \HH_k^\sI \bss_k^\sI \rangle + \frac{\alpha^3 L_{H} }{6} \| \bss_k^\sI \|^3 \\
        &= \alpha \left(\frac{1}{2} - \rho \right) \langle \bgg_k^\sI, \bss_k^\sI \rangle + \frac{\alpha}{2} (\langle \bgg_k^\sI, \bss_k^\sI \rangle +\langle \bss_k^\sI, \HH_k^\sI \bss_k^\sI \rangle) + \frac{\alpha^3 L_{H} }{6} \| \bss_k^\sI\|^3 \\
        &\leq \alpha \left(\frac{1}{2} - \rho \right)\langle \bgg_k^\sI, \bss_k^\sI \rangle + \frac{\alpha^3 L_{H} }{6} \| \bss_k^\sI\|^3\\
        &\leq -\alpha \left(\frac{1}{2} - \rho \right) \langle \bss_k^\sI , \HH^\sI_k \bss_k^\sI  \rangle + \frac{\alpha^3 L_{H} }{6} \| \bss_k^\sI \|^3 \\
        &\leq -\alpha \left(\frac{1}{2} - \rho \right)\sigma \| \bss_k^\sI  \|^2 + \frac{\alpha^3 L_{H} }{6} \| \bss_k^\sI \|^3\\
        &= \alpha  \left( -\left(\frac{1}{2} - \rho \right) \sigma + \frac{\alpha^2 L_{H} }{6} \| \bss_k^\sI \| \right) \| \bss_k^\sI \|^2.
    \end{align*}
    It can be seen that this upper bound is nonpositive if 
    \begin{align*}
        -\left(\frac{1}{2} - \rho\right) \sigma + \frac{\alpha^2 L_{H} }{6} \| \bss_k^\sI \|  \leq 0 \implies \alpha \leq  \sqrt{\frac{3 \sigma(1 - 2 \rho)}{L_{H} \| \bss_k^\sI \| } }.
    \end{align*}
    Collecting the bounds on $\alpha$, the largest step size that satisfies the line search condition can be lower bounded as 
    \begin{align*}
        \alpha_k \geq \min\left\{1, \sqrt{\frac{3 \sigma(1 - 2 \rho)}{L_{H} \| \bss_k^\sI \| } }, \frac{\delta_k}{\| \bss_k^\sI \|}  \right\}.
    \end{align*}

    Now let $\NPC$ so that $\pp_k^\sI = \rr_k^\sI$. Applying the negative curvature of $\rr_k^\sI$, \eqref{eqn:NPC step descent} and  \eqref{eqn:lipschitz upper bound of type 2 step}
    \begin{align*}
        f(\xx_k(\alpha)) - f(\xx_k) - \rho \alpha \langle \bgg_k, \rr_k^\sI \rangle  
        &\leq \alpha \langle \bgg_k^\sI, \rr_k^\sI \rangle + \frac{\alpha^2}{2} \langle \rr_k^\sI, \HH_k^\sI \rr_k^\sI\rangle + \frac{\alpha^3 L_{H} }{6} \| \rr_k^\sI\|^3 - \rho \alpha \langle \bgg_k, \rr_k^\sI \rangle \\
        &\leq \alpha (1 - \rho) \langle \bgg_k^\sI, \rr_k^\sI \rangle + \frac{\alpha^3 L_{H} }{6} \| \rr_k^\sI\|^3 \\
        &\leq -\alpha (1- \rho) \| \rr_k^\sI \|^2 + \frac{\alpha^3 L_{H} }{6} \| \rr_k^\sI\|^3 \\
        &= \alpha \left( -(1- \rho) + \frac{\alpha^2 L_{H} }{6} \| \rr_k^\sI\| \right) \| \rr_k^\sI\|^2.
    \end{align*}
    This upper bound is nonpositive if 
    \begin{align*}
         -(1- \rho) + \frac{\alpha^2 L_{H} }{6} \| \rr_k^\sI\| \leq 0 \implies \alpha \leq \sqrt{\frac{6 (1-\rho)}{ L_{H} \| \rr_k^\sI \| }}.
    \end{align*}
    Therefore the largest step size that satisfies the line search condition, in the NPC case, is lower bounded as 
    \begin{align*}
        \alpha_k \geq \min\left\{\sqrt{\frac{6 (1-\rho)}{ L_{H} \| \rr_k^\sI \| }}, \frac{\delta_k}{\| \rr_k^\sI \|}  \right\}.
    \end{align*}
\end{proof}

From \cref{lemma:type 2 line search termination} we can see that, for a judicious choice of $\delta_k$, the step size is inversely scaling with the step length, except for the $\alpha_k =1$ in $\SOL$ case. This inverse scaling is key to obtaining an improved rate. We therefore deal with the $\alpha_k = 1$ case separately. Indeed, in \cref{lemma:SOL alpha=1 gradient lower bound} we show that if $\alpha_k=1$ with $\SOL$ the step length must be lower bounded by norm of the gradient of the next iterate (over the same inactive set). This lemma is similar to the result in \citet[Lemma 7]{LiuNewtonMR}, we include it for completeness.
\begin{lemma}\label{lemma:SOL alpha=1 gradient lower bound}
    Suppose \cref{alg:newton-mr-two-metric} selects a \textit{Type II} step at iteration $k$ with $\SOL$ and $\alpha_k=1$, that is, an update of the form
    \begin{align*}
        \xx_{k+1} = \xx_k + \vvec{0}{\bss_k^\sI},
    \end{align*}
    with possible reordering. Under \cref{ass:LipschitzGradient,ass:LipschitzHessian,ass:KrylovSubspaceRegularity}, we have
    \begin{align*}
        \| \bss_k^{{\sI_k}} \| \geq c_0 \min\left\{ \left\| \bgg_{k+1}^{\sI_k} \right\|/\epsilon_k, \epsilon_k \right\},
    \end{align*}
    where 
    \begin{align*}
        c_0 \defeq \frac{ 2 \sigma  }{ \theta L_g   + \sqrt{ \theta^2 L_g^2 + 2 L_{H}\sigma^2} }.
    \end{align*}
\end{lemma}
\begin{proof}
    Since $(\rr_k^\sI)^{(t-1)} = - \HH_k^\sI (\bss_k^\sI)^{(t-1)} - \bgg_k^\sI \in \sK_t(\HH_k^\sI, \bgg_k^\sI)$ and NPC has not been detected, \cref{ass:KrylovSubspaceRegularity} implies 
    \begin{align*}
        \sigma \| \rr_k^\sI \|^2 \leq \langle \rr_k^\sI, \HH_k^\sI \rr_k^\sI \rangle \leq \| \rr_k^\sI \| \| \HH_k^\sI \rr_k^\sI \| \implies \| \rr_k^\sI \| \leq \frac{\| \HH_k^\sI \rr_k^\sI \|}{\sigma}. \tageq\label{eqn:residual Hr upper bound}
    \end{align*}
    For clarity, in the sequel we make the dependence of inactive set on the iteration explicit. Consider
    \begin{align*}
        \bgg_{k+1}^{\sI_k} = \left( \pdv{f(\xx_{k+1})}{\xx^i} \ | \ i \in \sI(\xx_k, \delta_k) \right),
    \end{align*}
    that is, the indices of the gradient evaluated at $\xx_{k+1}$ corresponding to the inactive set at $\xx_k$. This portion of the \textit{next} gradient ``lives'' in the same subset of the indices as $\bgg_k^{\sI_k}$. The mean value theorem therefore implies that 
    \begin{align*}
    \bgg_{k+1}^{\sI_k} - \bgg_k^{\sI_k}  - \HH_k^{\sI_k} \bss_k^{\sI_k} &= \int_0^1 \left( \HH \left(\xx_k + t \vvec{0}{\bss_k^\sI}\right)^{\sI_k} - \HH_k^{\sI_k} \right)\bss_k^{\sI_k} \, dt.
    \end{align*}
    \cref{ass:LipschitzHessian} implies
    \begin{align*}
        \left\| \bgg_{k+1}^{\sI_k} - \bgg_k^{\sI_k}  - \HH_k^{\sI_k} \bss_k^{\sI_k} \right\| \leq \frac{L_{H}}{2}\| \bss_k^{\sI_k} \| .
    \end{align*}
    Using this bound, \cref{eqn:MINRES termination tolerance}, and \cref{eqn:residual Hr upper bound}, we obtain
    \begin{align*}
        \left\| \bgg_{k+1}^{\sI_k} \right\| &= \left\| \bgg_{k+1}^{\sI_k} - \bgg_k^{\sI_k} - \HH_k^{\sI_k} \bss_k^{\sI_k} - \rr_k^{\sI_k} \right\| \\
        &\leq \left\| \bgg_{k+1}^{\sI_k} - \bgg_k^{\sI_k} - \HH_k^{\sI_k} \bss_k^{\sI_k} \right\|  + \| \rr_k^{\sI_k} \| \\
        &\leq \frac{L_{H}}{2} \| \bss_k^{\sI_k} \|^2  +  \frac{\| \HH_k^{\sI_k} \rr_k^{\sI_k}\|}{\sigma} \\
        &\leq \frac{L_{H}}{2} \| \bss_k^{\sI_k} \|^2  +  \frac{\theta  \epsilon_k \| \HH_k^{\sI_k} \bss_k^{\sI_k}\|}{\sigma} \\
         &\leq \frac{L_{H}}{2} \| \bss_k^{\sI_k} \|^2  +  \frac{\theta \epsilon_k L_g \| \bss_k^{\sI_k}\|}{\sigma},
    \end{align*}
    where the second to last line follows from the MINRES termination condition in \cref{alg:newton-mr-two-metric} and the last line follows from \cref{ass:LipschitzGradient}. Rearranging this expression, we obtain a quadratic inequality in $\| \bss_k^{\sI_k}\|$ as
    \begin{align*}
         0 &\leq L_{H} \sigma \| \bss_k^{\sI_k} \|^2  + 2 \theta \epsilon_k L_g \|\bss_k^{\sI_k}\| - 2\sigma \left\| \bgg_{k+1}^{\sI_k} \right\|.
    \end{align*}
    We can bound $\| \bss_k^{\sI_k} \| $ by the positive root of this quadratic as
    \begin{align*}
        \left\| \bss_k^{\sI_k} \right\| &\geq  \frac{-2 \theta \epsilon_k L_g  + \sqrt{4 \theta^2 \epsilon_k^2 L_g^2 + 8 L_{H}\sigma^2\| \bgg_{k+1}^{\sI_k} \| }}{2L_{H}\sigma} \\
        &= \left( \frac{- \theta  L_g  + \sqrt{ \theta^2 L_g^2 + 2 L_{H}\sigma^2\| \bgg_{k+1}^{\sI_k} \| /\epsilon_k^2 }}{L_{H}\sigma} \right) \epsilon_k \\
        &= \left( \frac{ \theta^2 L_g^2  - \left(  \theta^2 L_g^2 + 2 L_{H}\sigma^2 \| \bgg_{k+1}^{\sI_k} \|/\epsilon_k^2 \right) }{L_{H}\sigma \left( -\eta L_g  - \sqrt{ \theta^2 L_g^2 + 2 L_{H}\sigma^2 \| \bgg_{k+1}^{\sI_k} \|/\epsilon_k^2} \right)} \right) \epsilon_k \\
        &= \left( \frac{ 2 \sigma \| \bgg_{k+1}^{\sI_k} \|/\epsilon_k^2 }{ L_g  \theta  + \sqrt{ \theta^2 L_g^2 + 2 L_{H}\sigma^2 \| \bgg_{k+1}^{\sI_k} \|/\epsilon_k^2} } \right) \epsilon_k .
    \end{align*}
    We now consider two cases. If $\| \bgg_{k+1}^{\sI_k} \|/\epsilon_k^2 > 1$
    \begin{align*}
        \frac{ 2 \sigma \| \bgg_{k+1}^{\sI_k} \|/\epsilon_k^2 }{ \theta L_g + \sqrt{ \theta^2 L_g^2 + 2 L_{H}\sigma^2 \| \bgg_{k+1}^{\sI_k} \|/\epsilon_k^2} } 
        &= \frac{ 2 \sigma }{ \theta L_g\epsilon_k^2/\| \bgg_{k+1}^{\sI_k} \|  + \sqrt{ \theta^2 L_g^2 \epsilon_k^4/\| \bgg_{k+1}^{\sI_k} \|^2  +  2L_{H}\sigma^2 \epsilon_k^2/\| \bgg_{k+1}^{\sI_k} \| }} \\
        &\geq \frac{ 2 \sigma }{\theta L_g + \sqrt{ \theta^2 L_g^2  +  2L_{H}\sigma^2  }} .
    \end{align*}
    On the other hand, if $\| \bgg_{k+1}^{\sI_k} \|/\epsilon_k^2 \leq 1$
    \begin{align*}
        L_g \theta + \sqrt{ \theta^2 L_g^2 +  2 L_{H} \sigma^2\| \bgg_{k+1}^{\sI_k} \|/\epsilon_k^2}  \leq  L_g \theta + \sqrt{ \theta^2 L_g^2 +  2 L_{H} \sigma^2}.
    \end{align*}
    Together, these cases imply that 
    \begin{align*}
        \left\| \bss_k^{\sI_k} \right\| &= \left( \frac{ 2 \sigma \| \bgg_{k+1}^{\sI_k} \|/\epsilon_k^2 }{ L_g  \theta   + \sqrt{ \theta^2 L_g^2 + 2 L_{H}\sigma^2 \| \bgg_{k+1}^{\sI_k} \|/\epsilon_k^2} } \right) \epsilon_k \\
        &\geq \frac{ 2 \sigma  }{ L_g \theta   + \sqrt{ \theta^2 L_g^2 + 2 L_{H}\sigma^2} } \min\left\{ \left\| \bgg_{k+1}^{\sI_k} \right\|/\epsilon_k^2, 1 \right\} \epsilon_k.
    \end{align*}
\end{proof}

We now demonstrate the sufficient decrease of the \textit{Type II} step.
\begin{lemma}[\textit{Type II} Step: Sufficient Decrease] \label{lemma:type 2 sufficient decrease}
    Assume that $f$ satisfies \cref{ass:LipschitzGradient,ass:LipschitzHessian}. Suppose that a \text{Type II} step is taken on iteration $k$ of \cref{alg:newton-mr-two-metric} (i.e., $\sI(\xx_k, \delta_k) \neq \emptyset$ and $\| \bgg_k^\sI \| > \epsilon_k^2$). Let $\xx_{k+1} = \sP(\xx_k + \alpha_k \pp_k)$ where $\alpha_k$ is the largest step size satisfying the termination condition \eqref{eqn:line search criteria} (cf. \cref{lemma:type 2 line search termination}). Suppose that MINRES returns $\SOL$ and \cref{ass:KrylovSubspaceRegularity} is satisfied. Then, if $\|\bgg_{k+1}^\sI\| > 0$, we have 
    \begin{align*}
        f(\xx_{k+1}) - f(\xx_k) <  - \rho \sigma \min\left\{\sqrt{\frac{3 \sigma(1 - 2 \rho)}{L_{H}  } }C_{\sigma,L_g}  ^{3/2} \epsilon_k^3  , C_{\sigma,L_g}   \delta_k \epsilon_k^2, \frac{c_0^2 \left\| \bgg_{k+1}^{\sI} \right\|^2}{2\epsilon_k^2}, \frac{c_0^2\epsilon_k^2}{2} \right\}.
    \end{align*}
    where $c_0$ is defined in \cref{lemma:SOL alpha=1 gradient lower bound}. Note that if $\|\bgg_{k+1}^\sI\| =0$ strict inequality must be replaced with ``$\leq$''. On the other hand, if $\NPC$ and \cref{ass:ResidualLowerBoundRegularity} is satisfied, then
    \begin{align*}
        f(\xx_{k+1}) - f(\xx_k) < - \rho \min\left\{\sqrt{\frac{6 (1-\rho)}{ L_{H} }} \omega^{3/2} \epsilon_k^3, \omega \delta_k  \epsilon_k^2  \right\}.
    \end{align*}
\end{lemma}
\begin{proof}
    If $\SOL$, $\pp_k^\sI = \bss_k^\sI$. Combining the line search sufficient decrease \eqref{eqn:line search criteria}, the descent condition for the SOL step \eqref{eqn:SOL MINRES descent curvature condition} and \cref{ass:KrylovSubspaceRegularity}, we obtain
    \begin{align*}
        f(\xx_{k+1}) - f(\xx_k) 
        &\leq \alpha_k \rho  \langle \bss_k^\sI, \bgg_k^\sI \rangle \\
        &\leq - \alpha_k \rho  \langle \bss_k^\sI, \HH_k^\sI \bss_k^\sI \rangle \\ 
        &\leq - \alpha_k \rho \sigma \| \bss_k^\sI \|^2.
    \end{align*}
    Since $\alpha_k \leq 1$, if the step size returned by the line search satisfies $\alpha_k < 1$, then we must have  
    \begin{align*}
         \min\left\{\sqrt{\frac{3 \sigma(1 - 2 \rho)}{L_{H} \| \bss_k^\sI \| } }, \frac{\delta_k}{\| \bss_k^\sI \|}  \right\} \leq \alpha_k,
    \end{align*}
    as otherwise \eqref{eqn:type 2 SOL line search termination} would imply $\alpha_k \geq 1$. Therefore, by applying \cref{eqn:SOL step gradient related} with $\varrho= \sigma$, we obtain
    \begin{align*}
        f(\xx_{k+1}) - f(\xx_k) 
        &\leq - \rho \sigma \min\left\{\sqrt{\frac{3 \sigma(1 - 2 \rho)}{L_{H} \| \bss_k^\sI \| } }, \frac{\delta_k}{\| \bss_k^\sI \|}  \right\} \| \bss_k^\sI \|^2  \\
        &\leq - \rho \sigma \min\left\{\sqrt{\frac{3 \sigma(1 - 2 \rho)}{L_{H}  } } \| \bss_k^\sI \|^{3/2} , \delta_k\| \bss_k^\sI \|  \right\}  \\
        &\leq - \rho \sigma \min\left\{\sqrt{\frac{3 \sigma(1 - 2 \rho)}{L_{H}  } } C_{\sigma,L_g}  ^{3/2}\| \bgg_k^\sI \|^{3/2}  , C_{\sigma,L_g}   \delta_k\| \bgg_k^\sI \|  \right\} \\
        &< - \rho \sigma \min\left\{\sqrt{\frac{3 \sigma(1 - 2 \rho)}{L_{H}  } }C_{\sigma,L_g}  ^{3/2} \epsilon_k^3  , C_{\sigma,L_g}   \delta_k \epsilon_k^2   \right\},
    \end{align*}
    on the last line we use the fact that by assumption, $\| \bgg_k^\sI \| > \epsilon_k^2$. If the step size $\alpha_k = 1$ is selected by the line search, we can use Lemma \ref{lemma:SOL alpha=1 gradient lower bound} to obtain 
   \begin{align*}
       \| \bss_k^\sI\| \geq c_0 \min\left\{ \left\| \bgg_{k+1}^{\sI} \right\|/\epsilon_k, \epsilon_k \right\},
   \end{align*}
   which implies 
   \begin{align*}
        f(\xx_{k+1}) - f(\xx_k) 
        &\leq - \rho \sigma \| \bss_k^\sI \|^2  \\
        &<  - \frac{\rho \sigma c_0^2}{2} \min\left\{ \left\| \bgg_{k+1}^{\sI} \right\|^2/\epsilon_k^2, \epsilon_k^2 \right\} .
    \end{align*}
    If $\| \bgg_{k+1}^{\sI} \| =0$, the strict inequality must be replaced with ``$\leq$''. Combining the bounds we obtain the result.
    
    If $\NPC$, $\pp_k^\sI = \rr_k^\sI$. The line search condition \eqref{eqn:line search criteria}, the step size lower bound \eqref{eqn:type 2 NPC line search termination}, \eqref{eqn:NPC step descent} and Assumption \ref{ass:ResidualLowerBoundRegularity} imply  
    \begin{align*}
        f(\xx_{k+1}) - f(\xx_k) &\leq \alpha_k \rho \langle \rr_k^\sI, \bgg_k^\sI \rangle \\
        &\leq - \alpha_k \rho \| \rr_k^\sI \|^2 \\
        &\leq - \rho \min\left\{\sqrt{\frac{6 (1-\rho)}{ L_{H} \| \rr_k^\sI \| }}, \frac{\delta_k}{\| \rr_k^\sI \|}  \right\} \| \rr_k^\sI \|^2 \\
        &\leq - \rho \min\left\{\sqrt{\frac{6 (1-\rho)}{ L_{H} }} \| \rr_k^\sI \|^{3/2}, \delta_k \| \rr_k^\sI \|  \right\} \\
        &\leq - \rho \min\left\{\sqrt{\frac{6 (1-\rho)}{ L_{H} }} \omega^{3/2} \| \bgg_k^\sI \|^{3/2}, \delta_k \omega \| \bgg_k^\sI \|  \right\} \\
        &< - \rho \min\left\{\sqrt{\frac{6 (1-\rho)}{ L_{H} }} \omega^{3/2} \epsilon_k^3, \omega \delta_k  \epsilon_k^2  \right\},
    \end{align*}
    where the final inequality follows from the non-termination condition.
\end{proof}

We are finally ready to prove \cref{thm:first-order iteration complexity}.

\begin{proof}[Proof of \cref{thm:first-order iteration complexity}]
    Let $f^0 = f(\xx_0)$. We posit that the algorithm terminates in
    \begin{align*}
        K \defeq \left\lceil \frac{2(f^0 - f_*)\epsilon_g^{3/2}}{\min \{ c_{(1, 1)}, c_{(1, 2,)}, c_{(2, 2)} \} }   + 1 \right\rceil,
    \end{align*}
    iterations where $c_{(1, 1)}$, $c_{(1, 2,)}$ and $ c_{(2, 2)} $ are constants that will be defined later. Suppose, to the contrary, that the termination conditions is unsatisfied until at least iteration $K+1$. Then for iterations $k = 0, \ldots, K$ at least one of the termination \cref{eqn:active gradient negativity condition,eqn:active gradient norm termination condition,eqn:inactive set termination condition} conditions must be unsatisfied. We divide the \textit{Type I} iterations
    \begin{align*}
        \sK_1 = \{ k \in [K] \ | \ \sA(\xx_k, \epsilon_g^{1/2}) \neq \emptyset, \quad  \text{and} \quad  (\exists i \in \sA(\xx_k, \epsilon_g^{1/2}),  \quad  \bgg_k^i < - \sqrt{\epsilon_g}, \quad \text{or} \quad \| \diag(\xx_k^\sA) \bgg_k^\sA \| \geq \epsilon_g )\},
    \end{align*}
    into two sets
    \begin{align*}
        \sK_{1, 1} &= \sK_1 \cap \{ k \in [K] \ | \ \| \bgg_k^\sI\| > \epsilon_g^{3/4}, \quad \text{and} \quad \sI(\xx_k, \epsilon_g^{1/2}) \neq \emptyset \}, \\ 
        \sK_{1, 2} &= \sK_1 \cap \{ k \in [K] \ | \ \| \bgg_k^\sI\| \leq \epsilon_g^{3/4}, \quad \text{or} \quad \sI(\xx_k, \epsilon_g^{1/2} )=\emptyset \}.
    \end{align*}
    For the \textit{Type II} iterations $\sK_2 = [K] \setminus \sK_1$ we have $\sI(\xx_k, \epsilon_g^{1/2}) \neq \emptyset$ and $\|\bgg_k^{\sI_k} \| > \epsilon_g$. We divide them as follows
    \begin{align*}
        \sK_{2, 1} &= \sK_2 \cap \{ k \in [K] \ | \ \sI(\xx_{k+1}, \epsilon_g^{1/2}) \neq \emptyset, \quad \text{and} \quad  \| \bgg_{k+1}^{\sI_{k+1}} \| > \epsilon_g  \}, \\
        \sK_{2, 2} &= \sK_2 \cap \{ k \in [K] \ | \ \sI(\xx_{k+1}, \epsilon_g^{1/2}) = \emptyset, \quad \text{or} \quad \| \bgg_{k+1}^{\sI_{k+1}} \| \leq \epsilon_g  \}.
    \end{align*}
    We now restate the results obtained for per-iteration decrease.
    
    \textbf{\textit{Type I} step}. For $k \in \sK_{1, 1}$, Lemma \ref{lemma:type 1 inactive set decrease} applies and by combining the NPC and SOL cases and using $\epsilon_g<1$ we obtain, 
    \begin{align*}
        f(\xx_{k+1}) - f(\xx_k) &< - \min \left\{ c_{(1, 1)}^a \epsilon_g^{3/2} , c_{(1, 1)}^b \epsilon_g^{5/4} \right\} \leq - \min \left\{ c_{(1, 1)}^a, c_{(1, 1)}^b \right\} \epsilon_g^{3/2}
        = -c_{(1, 1)}\epsilon_g^{3/2}, \tageq\label{eqn:11 decrease}
    \end{align*}
    where 
    \begin{align*}
        c_{(1, 1)}^a &\defeq \rho\min\left\{ \frac{2(1 - \rho)\omega^2 }{L_{g}}, \frac{2(1 - \rho)\min \{1, \sigma \} \sigma C_{\sigma,L_g}  ^2 }{L_{g}} \right\}, \\
        c_{(1, 1)}^b &\defeq \rho \min\left\{\omega, \sigma C_{\sigma,L_g}   \right\}, \quad \text{and} \quad \ c_{(1, 1)} \defeq \min\{c_{(1, 1)}^a, c_{(1, 1)}^b \}.   
    \end{align*}

    \bigskip 

    For $k \in \sK_{1, 2}$, \cref{lemma:type 1 active set decrease} applies. Indeed, for $\sI(\xx_k, \epsilon_g^{1/2}) \neq \emptyset$ we obtain a decrease 
     \begin{align*}
        f(\xx_{k+1}) - f(\xx_k) &<  - \frac{\rho}{2} \min \left\{ 1, \min \{1, \sigma\} \min \left\{\frac{2(1 - \rho)}{L_{g}} , \frac{1}{ \epsilon_g^{1/4}} \right\} \right\}\epsilon_g \\
        &\leq  - \frac{\rho}{2} \min \left\{ 1, \min \{1, \sigma\} \min \left\{\frac{2(1 - \rho)}{L_{g}} , 1 \right\} \right\}\epsilon_g^{3/2},
    \end{align*}
    where on the second line we used $\epsilon_g < 1$. The decrease in the case where $\sI(\xx_k, \epsilon_g^{1/2}) = \emptyset$ is given by 
    \begin{align*}
        f(\xx_{k+1}) - f(\xx_k) < - \frac{\rho}{2} \min \left\{ 1,  \frac{2(1- \rho)}{L_g} \right\}\epsilon_g^{3/2}.
    \end{align*}
    Combining these results we obtain
    \begin{align*}
        f(\xx_{k+1}) - f(\xx_k) < - c_{(1,2)} \epsilon_g^{3/2}, \tageq\label{eqn:12 decrease}
    \end{align*}
    where
    \begin{align*}
        c_{(1, 2)} \defeq  \frac{\rho}{2}  \min\left\{1,  \min\{1 , \sigma \} \min\left\{ \frac{2(1- \rho)}{L_g}, 1 \right\} \right\}.
    \end{align*}

    \bigskip
    
    \textbf{\textit{Type II} Step}:  For $k \in \sK_{2, 1}$, we can apply $\| \bgg_{k+1}^{\sI_{k+1}} \| > \epsilon_g$ to further refine the bound for the SOL case. Note that because $\delta_k = \delta_{k+1} = \epsilon_g^{1/2}$ and a \textit{Type II} step is taken $\sI\left(\xx_{k+1}, \epsilon_g^{1/2}\right) \subseteq \sI\left(\xx_k, \epsilon_g^{1/2}\right)$. Indeed, if $i \in \sA\left(\xx_k, \epsilon_g^{1/2} \right)$ we have $\pp_k^i = 0$ and hence $\xx_{k+1}^i = \xx_k^i \leq \epsilon_g^{1/2}$ and $ \sA \left(\xx_k, \epsilon_g^{1/2}\right) \subseteq \sA \left(\xx_{k+1}, \epsilon_g^{1/2}\right) $. Together these results imply that
    \begin{align*}
        \epsilon_g < \| \bgg_{k+1}^{\sI_{k+1}} \| \leq \| \bgg_{k+1}^{\sI_k} \|.
    \end{align*} 
    With $\SOL$ and using $\epsilon_g < 1$, \cref{lemma:type 2 sufficient decrease} implies
    \begin{align*}
        f(\xx_{k+1}) - f(\xx_k) &< - \rho \sigma \min\left\{\sqrt{\frac{3 \sigma(1 - 2 \rho)C_{\sigma,L_g}  ^3}{L_{H}  } }\epsilon_g^{3/2} ,C_{\sigma,L_g}   \epsilon_g^{3/2}, \frac{c_0^2}{2} , \frac{c_0^2 \epsilon_g}{2}, \right\} \\    
        &\leq - \rho \sigma \min\left\{\sqrt{\frac{3 \sigma(1 - 2 \rho)C_{\sigma,L_g}  ^3}{L_{H}  } }  , C_{\sigma,L_g}    , \frac{c_0^2 }{2} \right\}\epsilon_g^{3/2}.
    \end{align*}
    With $\NPC$, this becomes 
    \begin{align*}
        f(\xx_{k+1}) - f(\xx_k) < - \rho \min\left\{\sqrt{\frac{6 (1-\rho)}{ L_{H} }} \omega^{3/2} , \omega \right\} \epsilon_g^{3/2},
    \end{align*}
    and so by combining the $\NPC$ and $\SOL$ cases, we have
    \begin{align*}
        f(\xx_{k+1}) - f(\xx_{k}) &< - c_{(2, 2)} \epsilon_g^{3/2}, \tageq\label{eqn:22 decrease}
    \end{align*}
    where
    \begin{align*}
        c_{(2, 2)} \defeq \rho \min\left\{\sqrt{\frac{3 \sigma^3 (1 - 2 \rho) C_{\sigma,L_g}  ^3}{L_{H}} }  , \sigma C_{\sigma,L_g}    , \frac{\sigma c_0^2 }{2}, \sqrt{\frac{6 (1-\rho) \omega^3}{ L_{H} }} , \omega \right\}.
    \end{align*}

    For $k \in \sK_{2, 2}$, the lower bound for the next gradient norm is no longer available. However, due to \cref{lemma:type 1 line search termination}, we have at least
    \begin{align*}
        f(\xx_{k+1}) - f(\xx_k) \leq 0.
    \end{align*}
    Additionally, due to the non-termination of the algorithm, $k \in \sK_{2,2}$ implies $k+1 \in \sK_{1}$ unless $k = K$, in which case $K+1$ could be the iteration the algorithm terminates. We can therefore write 
    \begin{align*}
        | \sK_{(2,1)} | \leq | \sK_1| + 1.
    \end{align*}
    We now bound the total decrease in terms of the number of iterations that must have occurred using \cref{eqn:11 decrease,eqn:12 decrease,eqn:22 decrease}
    \begin{align*}
        f^0 - f^* &\geq f^0 - f(\xx_{K+1})  \\
         &= \sum_{k=0}^K f(\xx_{k}) - f(\xx_{k+1}) \\
         &= \sum_{k \in \sK_{(1, 1)}} f(\xx_k) - f(\xx_{k+1}) + \sum_{k \in \sK_{(1, 2)}} f(\xx_k) - f(\xx_{k+1})\\
         &+ \sum_{k \in \sK_{(2, 1)}} f(\xx_k) - f(\xx_{k+1}) + \sum_{k \in \sK_{(2, 2)}} f(\xx_k) - f(\xx_{k+1}) \\
         &> \sum_{k \in \sK_{(1, 1)}} c_{(1,1)} \epsilon_g^{3/2} + \sum_{k \in \sK_{(1, 2)}} + c_{(1,2)} \epsilon_g^{3/2} + \sum_{k \in \sK_{(2, 2)}} c_{(2,2)} \epsilon_g^{3/2} \\
         &= |\sK_{(1, 1)}|c_{(1,1)}\epsilon_g^{3/2} + |\sK_{(1, 2)}|c_{(1,2)}\epsilon_g^{3/2} + |\sK_{(2, 2)}|c_{(2,2)}\epsilon_g^{3/2}.
    \end{align*}
    Since each term is positive, we get 
    \begin{align*}
        |\sK_{(1, 1)}| < \frac{ (f^0 - f^*)\epsilon_g^{-3/2} }{c_{(1,1)}}, \ |\sK_{(1, 2)}| < \frac{ (f^0 - f^*)\epsilon_g^{-3/2} }{c_{(1,2)}}, \ |\sK_{(2, 2)}| < \frac{ (f^0 - f^*)\epsilon_g^{-3/2} }{c_{(2,2)}}.
    \end{align*}
    Hence, if we add up the total number of iterations that must have been taken
    \begin{align*}
        K &= |\sK_{(1, 1)}| + |\sK_{(1, 2)}| + |\sK_{(2, 1)}| + |\sK_{(2, 2)}| \\
        &\leq 2(|\sK_{(1, 1)}| + |\sK_{(1, 2)}|) + |\sK_{(2, 2)}| + 1 \\
        &< \frac{2(f^0 - f^*)\epsilon_g^{3/2}}{c_{(1, 1)}} + \frac{2(f^0 - f^*)\epsilon_g^{3/2}}{c_{(1, 2)}} + \frac{(f^0 - f^*)\epsilon_g^{3/2}}{c_{(2, 2)}}  + 1 \\
        &\leq \left\lceil \frac{2(f^0 - f^*)\epsilon_g^{3/2}}{\min \{ c_{(1, 1)}, c_{(1, 2,)}, c_{(2, 2)} \} }   + 1\right\rceil\\
        &= K ,
    \end{align*}
    we arrive at a contradiction.
\end{proof}

\section{Operational Complexity} \label{apx:operational complexity}

The results in this section are corollaries of \cref{thm:minimal assumptions convergence theorem} and \cref{thm:first-order iteration complexity} and the MINRES iteration bounds in \citet{LiuNewtonMR}. The following definitions are included from \citep{LiuNewtonMR} for completeness.

Let $\Psi(\HH, \bgg)$ denote the set of $\bgg$-relevant eigenvalues\footnote{Eigenvalues outside of $\Psi(\HH, \bgg)$ are essentially ``invisible'' to the Krylov subspace built out of products of $\HH$ and $\bgg$.}, that is, the eigenvalues whose eigenspace is \textit{not} orthogonal to $\bgg$. Denote $\psi = |\Psi(\HH, \bgg)|$ and let $\psi_-$, $\psi_0$ and $\psi_+$ be the number of negative, zero and positive $\bgg$-relevant eigenvalues so that $\psi = \psi_- + \psi_0 + \psi_+$. We impose the following order on the eigenvalues 
\begin{align*}
    \lambda_1 > \lambda_2 > \ldots > \lambda_{\psi_+} > 0 > \lambda_{\psi_+ + \psi_0 + 1} > \ldots > \lambda_{\psi}.
\end{align*}
Denote by $\UU_i$ the matrix with columns which form an orthonormal basis of the $i^\text{th}$ eigenspace with the convention that the leading column is the only column onto which the gradient has nonzero projection. For $1 \leq i \leq \psi_+$ and $\psi_+ + \psi_0 + 1 \leq j \leq \psi$, define the following matrices
\begin{align*}
    \UU_{i+} = [\UU_1 \ldots \UU_{i} ], \ \UU_{j-} = [\UU_{j}, \ldots, \UU_{\psi}].
\end{align*}
The columns of $\UU_{i+}$ represent the eigenspaces of the $i$ most positive $\bgg$-relevant eigenvalues, while $\UU_{j-}$ represents the eigenspaces corresponding to the $j$ most negative $\bgg$-relevant eigenvalues. As a special case, let $\UU_+ = \UU_{\psi+}$ and $\UU_- = \UU_{(\psi_+ + \psi_0+ 1) -}$. Finally, let  
\begin{align*}
    \UU = [\UU_+, \UU_-].
\end{align*}
We now state a key assumption for the result.

\begin{assumption}{\citep[Assumption 5]{LiuNewtonMR}}\label{ass:uniform minres operational bound}
    There exists $\tau > 0$ and $L_g^2/(L_g^2 + \eta^2) < \nu \leq 1$ such that for any $\xx \in \real^d_+$ with $\bgg \notin \Null{(\HH)}$ at least one of the following statements (i)-(iii) must hold

    (i) If $\psi_+ \geq 1$ and $\psi_i \geq 1$ then there exists $1 \leq i \leq \psi_+$ and $\psi_{+} + \psi_0 + 1 \leq j \leq \psi$ such that 
    \begin{align*}
        \min\{ \lambda_i, -\lambda_j \} &\geq \tau, \\
        \| (\UU_{i+}\UU_{i+}^\transpose + \UU_{j-}\UU_{j-}^\transpose) \bgg \|^2 &\geq \nu \| \UU \UU^\transpose \bgg \|^2.
    \end{align*}

    (ii) If $\psi_+ \geq 1$ then there exists $1 \leq i \leq \psi_+$ such that 
    \begin{align*}
        \lambda_i &\geq \tau, \\
        \| \UU_{i+} \UU_{i+}^\transpose \bgg \|^2 &\geq \nu \| \UU \UU^\transpose \bgg\|^2.
    \end{align*}

    (iii) If $\psi_- \geq 1$ then there exists some $\psi_+ + \psi_0 +1 \leq j \leq \psi$ such that
    \begin{align*}
        -\lambda_j &\geq \tau, \\
        \| \UU_{j-} \UU_{j-}^\transpose \bgg \|^2  &\geq \nu \| \UU \UU^\transpose \bgg\|^2.
    \end{align*}
\end{assumption}

Recall that $\UU\UU^\transpose$, $\UU_{j-} \UU_{j-}^\transpose$ and $\UU_{i+} \UU_{i+}^\transpose$ represent a projection onto corresponding eigenspaces. Each part of \cref{ass:uniform minres operational bound} has a natural interpretation as a requirement that there is at least one \textit{large enough} magnitude $\bgg$-relevant eigenvalue for which the projection of the gradient onto the corresponding eigenspaces is not \textit{too small}. This is a significant relaxation of a more conventional uniform bound on the magnitude of the eigenvalues. For example, a uniform bound on the smallest magnitude eigenvalues 
\begin{align*}
      \min\{ \lambda_{\psi_+}, -\lambda_{\psi_+ + \psi_0 + +1}\} \geq \tau,
\end{align*}
immediately implies \cref{ass:uniform minres operational bound}-(i) with $i = \psi_+$, $j = \psi_+ + \psi_0 + 1$ and $\nu=1$. See \citet[Assumption 5]{LiuNewtonMR} for further discussion of this assumption.

\cref{ass:uniform minres operational bound} allows us to bound the number of Hessian vector products that are required for MINRES to satisfy \cref{eqn:MINRES termination tolerance}. Indeed, assuming $\bgg \notin \Null(\HH)$, we appeal to \citet[Eqn.\ (20)]{LiuNewtonMR} to bound the number of iterations until the MINRES termination tolerance \cref{eqn:MINRES termination tolerance} is satisfied as 
\begin{align*}
    T_{\text{SOL}} = \min \left\{  \left\lceil \frac{\sqrt{L_g/\mu}}{4} \log\left( 4/\left( \frac{\eta^2}{L_g^2 + \eta^2} - (1- \nu) \right)\right)  + 1\right\rceil , g \right\},
\end{align*}
where $g$ denotes the grade of $\bgg$ with respect to $\HH$ \citep[Definition 1.3]{LiuMinres}. We note that $T_{\text{SOL}}$ has a logarithmic dependence on the inexactness rolernance, $\eta$.

On the other hand if $\psi_- \geq 1 $ and \cref{ass:uniform minres operational bound}-(iii) holds, we appeal to \citet[Eqn.\ (19)]{LiuNewtonMR} to bound the iterations required to obtain a NPC direction as
\begin{align*}
    T_{\text{NPC}}= \min\left\{ \max \left\{ \left\lceil \left( \frac{\sqrt{2(L_g + \mu)/\mu}}{4}\right)\log\left( \frac{2(L_g + \mu)(1 - \nu)}{\mu\nu} \right)  + 1 \right\rceil , 1 \right\}, g\right\}.
\end{align*}
When $\nu = 1$, it is clear from the statement of \cref{ass:uniform minres operational bound}-(iii) that all $\bgg$-relevant eigenvalues are negative, which implies that negative curvature is detected at the very first iteration, i.e., $T_{\text{NPC}} = 1$. If we adopt the convention that $T_{\text{NPC}} = \infty$ when $\psi_- = 0$ or \cref{ass:uniform minres operational bound}-(iii) is unsatisfied we bound the number of MINRES iterations as $T = \min\{T_{\text{NPC}} , T_{\text{SOL}} \}$. If $\bgg \in \Null(\HH)$ then $\bgg$ is declared a zero curvature direction at the very first iteration. We now prove the operational complexity results. 

\begin{corollary}[First Order Operational Complexity \cref{alg:newton-mr-two-metric-minimal-assumptions-case}] \label{cor:first-order operational complexity minimal assumptions}
    Under the conditions of \cref{thm:minimal assumptions convergence theorem} and \cref{ass:uniform minres operational bound}, the total number of gradient evaluations and Hessian vector products in \cref{alg:newton-mr-two-metric-minimal-assumptions-case} to obtain an $\epsilon_g$-FO point is $\sO(\epsilon_g^{-2})$, for $d$ sufficiently large. 
\end{corollary}
\begin{proof}
    Due to \cref{thm:minimal assumptions convergence theorem}, the total number of outer iterations is $\sO(\epsilon_g^{-2})$. To obtain the operational result we simply need to count the total number of gradient evaluations and Hessian vector products \textit{per iteration}. The work required for each step of \cref{alg:newton-mr-two-metric-minimal-assumptions-case} is equivalent to the number of MINRES iterations (i.e. Hessian vector product) plus a single gradient evaluation. In the case of \cref{alg:newton-mr-two-metric-minimal-assumptions-case} the termination tolerance $\eta$ has no dependence on $\epsilon_g$. Considering the discussion above, for sufficiently large $d$, we bound the number of Hessian vector products as $\sO(1)$. The conclusion follows from the fact that $\sO(\epsilon_g^{-2}) (1 + \sO(1)) \in \sO(\epsilon_g^{-2})$.
\end{proof}

\begin{corollary}[First Order Operational Complexity \cref{alg:newton-mr-two-metric}] \label{cor:first-order operational complexity}
    Under the conditions of \cref{thm:first-order iteration complexity} and \cref{ass:uniform minres operational bound}, the total number of gradient evaluations and Hessian vector products in \cref{alg:newton-mr-two-metric} to obtain an $\epsilon_g$-FO point is $\Tilde{\sO}(\epsilon_g^{-3/2})$,  for $d$ sufficiently large. 
\end{corollary}
\begin{proof}
    The result is similar to \cref{cor:first-order operational complexity minimal assumptions}. We utilise \cref{thm:first-order iteration complexity} to bound the total number of outer iterations as $\sO(\epsilon_g^{-3/2})$. For \cref{alg:newton-mr-two-metric}, the MINRES termination tolerance is $\eta = \theta \sqrt{\epsilon_g}$, so we bound the total number of Hessian vector products as $\tilde{\sO}(1)$ for $d$ large. The conclusion follows.
\end{proof}

\section{Local Convergence} \label{apx:local convergence}

\begin{algorithm}[ht]
    \begin{algorithmic}[1]
        \FOR{$ k= 0, 1, \ldots$}
        \smallskip
        \STATE Update sets $\sA(\xx_k, \delta_k)$ and $\sI(\xx_k, \delta_k)$ as in \cref{eqn:inactive and active set definition}.
        \smallskip
        \IF{ Some termination condition is satisfied}
            \smallskip
            \STATE \textbf{Terminate}.
            \smallskip
        \ENDIF
        \smallskip
        \STATE $\ppk: \left\{\begin{array}{ll}
             \pp^{\sA}_k \gets - \bgg^{\sA}_k,&\\\\
            (\pp_k^\sI, \ \Dtype) \gets \text{MINRES}(\HH_k^\sI, \bgg_k^\sI, \eta, 0) \qquad \text{\# See \cref{remark:local convergence} regarding the choices for $\eta$.} &
        \end{array}\right.$

        \smallskip

        \smallskip
        \STATE $\alpha_{k} \gets$ \cref{alg:back tracking line search} with $\alpha_0=1$ and \eqref{eqn:line search criteria}.
        \smallskip
        \STATE $\xx_{k+1} = \sP(\xx_k + \alpha_k \pp_k)$
        \ENDFOR
    \end{algorithmic}
    \caption{Newton-MR TMP (Local Phase Version)}
    \label{alg:newton-mr-two-metric-local-phase}
\end{algorithm}

In this section we provide the detailed proof for \cref{thm:active set identification}. Our proof follows a similar line of reasoning as that in \citet[Proposition 3]{ProjectionNewtonMethodsForOptimizationProblemsBertsekas} but with several modifications and alterations specific to our setting and methodology. We assume in this section that $\sI(\xx_*, 0) \neq \emptyset$ as otherwise the analysis boils down to convergence of projected gradient to a trivial solution $\xx_* = 0$. Our main aim is to show that after a finite number of iterations, the iterates eventually end up in the following subspace
\begin{align*}
    X_* = \{ \xx \in \real^d \ | \ \xx^i = 0, \ i \in \sA(\xx_*, 0)\}.
\end{align*}
We start with a lemma to show that, by choosing our inexactness tolerance $\delta_k = \delta$, with 
\begin{align*}
    0 < \delta < \frac{1}{2}\min_{i \in \sI(\xx_*, 0)} \xx_*^i, \tageq\label{eqn:local convergence delta choice}
\end{align*} 
where $\xx_*$ is some local minima, we can properly ``separate'' the true active and inactive set if $\xx_k$ is close enough to $\xx_*$. That is, we apply the correct update to the true active and inactive indices.

\begin{lemma} \label{lemma:local phase active set separation}
    Let $\xx_*$ be a local minima of \cref{eqn:nonnegative constrained problem} and $\xx_k$ be an iterate of \cref{alg:newton-mr-two-metric-local-phase} with $\delta$ chosen according to \cref{eqn:local convergence delta choice}. There exists $\Delta_{\text{sep}}$ such that if $\xx_k \in B(\xx_*, \Delta_{\text{sep}})$, then $\sA(\xx_k, \delta) = \sA(\xx_*, 0)$.
\end{lemma}

\begin{proof}
Define
\begin{align*}
        \Delta_{\text{sep}} \defeq \min \left\{\frac{1}{2}\left( \min_{i \in \sI(\xx_*, 0)} \xx_*^i - \delta \right), \delta \right\} > 0.
    \end{align*}
We first we prove $\sI(\xx_k, \delta_k) \supseteq \sI(\xx_*, 0)$. For any $i \in \sI(\xx_*, 0)$ and $\xx_k \in B(\xx_*, \Delta_{\text{sep}})$ we have 
\begin{align*}
    \xx_*^i - \xx^i_k < \Delta_{\text{sep}} &\leq \frac{\xx_*^i - \delta}{2}\\
    \implies  \frac{\xx_*^i}{2} - \xx^i_k &\leq -\frac{\delta}{2} \\
    \implies   \frac{\xx_*^i}{2} + \frac{\delta}{2} &\leq \xx^i_k \\
    \implies \frac{3\delta}{2} &\leq \xx^i_k \\
    \implies \delta &< \xx^i_k,
\end{align*}
where the second to last line follows from \cref{eqn:local convergence delta choice}. 

Next we show that $\sI(\xx_k, \delta) \subseteq \sI(\xx_*, 0)$. In particular, we prove the contrapositive $i \in \sA(\xx_*, 0) \implies i \in \sA(\xx_k, \delta)$. For $i \in \sA(\xx_*, 0)$ we know that $\xx_*^i = 0$ and so for $\xx_k \in B(\xx_*, \Delta_{\text{sep}})$ we have
\begin{align*}
    \xx_k^i  = \xx_k^i - \xx_* < \Delta_{\text{sep}} \leq \delta .
\end{align*}   
That is, $i \in \sA(\xx_k, \delta)$.
\end{proof}

With this result in hand, we know that we apply the ``correct'' update to $\xx_k$. That is, true active indices receive a projected gradient update, while indices in the true inactive set receive a Newton-type step.   

Recall the the second-order sufficient condition in \cref{ass:local minima second-order sufficiency}. This condition is equivalent to 
\begin{align*}
    \langle \zz, \grad^2 f(\xx_*) \zz \rangle > 0, \ \zz \in X_*.
\end{align*}
By the continuity of the Hessian, we are free to choose $\Delta_{\text{cvx}} > 0$ such that for any $\xx \in B(\xx_*, \Delta_{\text{cvx}})$ the Hessian remains strongly positive definite on the subspace $X_*$. In other words, the constant, $\mu$, satisfying
\begin{align*}
    \mu \defeq \min_{\zz \in X_*, \ \xx \in B(\xx_*, \Delta_{\text{cvx}})} \frac{ \langle \zz, \grad^2 f(\xx) \zz \rangle }{\| \zz \|^2} > 0, \tageq\label{eqn:local strong convexity constant}
\end{align*}
is well defined. In the current notation, even if $\sI(\xx_k, \delta) = \sI(\xx_*, 0)$ we have $\pp_k^{\sI(\xx_k, \delta)} \notin X_*$ as $\pp_k^{\sI(\xx_k, \delta)} \in \real^{|{\sI(\xx_k, \delta)}|}$ is a subvector. Therefore, as a notational convenience, in this section we take subvectors and submatrices corresponding to a certain subset of indices, e.g., $\pp_k^{\sI(\xx_k, \delta)}$, to be padded with zeros in the removed indices. Note that this implies that $\pp_k^{\sI(\xx_k, \delta)}, \pp_k^{\sA(\xx_k, \delta)} \in \real^d $ and $\pp_k = \pp_k^{\sI(\xx_k, \delta)} + \pp_k^{\sA(\xx_k, \delta)}$ but leaves the mechanics of \cref{alg:newton-mr-two-metric-local-phase} unchanged. Now if $\sI(\xx_k, \delta) = \sI(\xx_*, 0)$ we have 
\begin{align*}
    \pp_k^{\sI(\xx_k, \delta)} = \pp_k^{\sI(\xx_*, 0)}  \in X_*.
\end{align*}
Indeed, it is easy to see that for any $t=0,\ldots, g$
\begin{align*}
    \sK_t(\HH_k^{\sI(\xx_k, \delta)}, \bgg_{k}^{\sI(\xx_k, \delta)}) \subseteq X_*.
\end{align*}
In addition, if $\xx_k \in B(\xx_*, \Delta_{\text{cvx}})$ then, by $\HH_k^{\sI(\xx_k, \delta)} = \HH_k^{\sI(\xx_*, 0)} $, $\mu$ plays the role of Krylov subspace regularity constant, $\sigma$, (cf. \cref{ass:KrylovSubspaceRegularity}) on $\sK_t(\HH_k^{\sI(\xx_*, 0)}, \bgg_k^{\sI(\xx_*, 0)})
$. Indeed, together, these results imply that, for any $t=0, \ldots, g$, we have
\begin{align*}
    \bss \in \sK_t(\HH_k^{\sI(\xx_k, \delta)}, \bgg_{k}^{\sI(\xx_k, \delta)}) \implies \langle \bss, \HH_k^{\sI(\xx_k, \delta)} \bss \rangle \geq \mu \| \bss \|^2. \tageq\label{eqn:local phase Krylov subspace regularity}
\end{align*}

From \cref{eqn:local phase Krylov subspace regularity} it is clear that we can do away with the $\NPC$ case and \cref{ass:KrylovSubspaceRegularity}. With this in mind, we now demonstrate that the step size produced by the line search in \cref{alg:newton-mr-two-metric-local-phase} is bounded. 

\begin{lemma}\label{lemma:local phase step size lower bound}
    Assume that $f$ satisfies \cref{ass:LipschitzGradient} and $\xx_*$ is a local minima of \cref{eqn:nonnegative constrained problem} satisfying \cref{ass:local minima second-order sufficiency}. Then if $\xx_k \in B(\xx_*, \min\{ \Delta_{\text{cvx}}, \Delta_{\text{sep}}\})$ the step size produced by the line search in \cref{alg:newton-mr-two-metric-local-phase} satisfies $\alpha_k \in [\bar{\alpha}, 1]$ where
    \begin{align*}
        \bar{\alpha} \defeq \min \left\{1, \frac{2(1 - \rho)\mu}{L_{g}} ,   \frac{\delta}{\| \pp_k^\sI \|}\right\}. \tageq\label{eqn:local phase step lower bound}
    \end{align*}
\end{lemma}
\begin{proof}
    $\xx_k \in B(\xx_*, \min\{ \Delta_{\text{cvx}}, \Delta_{\text{sep}}\})$ implies that $\sI(\xx_k, \delta) = \sI(\xx_*,  0) \neq \emptyset$ and \cref{eqn:local phase Krylov subspace regularity} holds. It follows that MINRES always selects $\SOL$ step.
    
    The result follows from the step size selection procedure in \cref{alg:newton-mr-two-metric-local-phase} and the analysis in the SOL case of \cref{lemma:type 1 line search termination}.

\end{proof}

Building on \cref{lemma:local phase active set separation}, our next result, \cref{lemma:close enough to minima lemma}, will show that, close enough to $\xx_*$, the active set update will be \textit{large enough} and the inactive set update \textit{small enough} that the zero bound constraints at $\xx_{k+1}$ coincide with the zero bound constraints at $\xx_*$, i.e., $\sA(\xx_{k+1}, 0) = \sA(\xx^*, 0)$. The intuition for this result is that the gradient and hence the step (cf. \cref{eqn:SOL step gradient related}) in the inactive indices should be going to zero as $\xx_k$ approaches $\xx_*$. By contrast, in the active set, a non-degeneracy condition (\cref{ass:local minima nondegeneracy}) ensures there is positive gradient in the active indices arbitrarily close to the boundary. When $\sA(\xx_{k+1}, 0) = \sA(\xx^*, 0)$, the fixed active set and small inactive step can also be used to ensure that our iterates do not drift too far from the starting point. This is the second part of \cref{lemma:close enough to minima lemma}.

\begin{lemma}
    \label{lemma:close enough to minima lemma}
    Suppose that $f$ satisfies \cref{ass:LipschitzGradient}. Let $\xx_*$ be a local minima satisfying \cref{ass:local minima nondegeneracy,ass:local minima second-order sufficiency}. If $\delta$ is chosen according to \cref{eqn:local convergence delta choice}, then the following two results hold:
    
    1.  There exists $\Delta_{\text{bnd}} > 0$ such that $\xx_k \in B(\xx_*, \Delta_{\text{bnd}})$ implies $\sA(\xx_k, \delta) = \sA(\xx_{k+1}, 0) = \sA(\xx_*, 0)$.

    2. Given a $\Delta > 0$, we can choose $\Delta_{\text{cls}} \in (0, \Delta_{\text{bnd}})$ such that $\|\xx_{k} - \xx_*\| < \Delta_{\text{cls}}$ implies that $\| \xx_{k+1} -\xx_* \| < \Delta$.
\end{lemma}
\begin{proof}

    We stipulate that $\Delta_{\text{bnd}} \leq \min\{\Delta_{\text{sep}}, \Delta_{\text{cvx}} \}$. Note that, in this case, $\sA(\xx_k, \delta) = \sA(\xx_*, 0)$ by \cref{lemma:local phase active set separation} and $\alpha_k \in [\bar{\alpha}, 1]$ where
   \begin{align*}
        \bar{\alpha} \defeq \min \left\{1, \frac{2(1 - \rho)\mu}{L_{g}} ,   \frac{\delta}{\| \pp_k^\sI \|}\right\}. \tageq\label{eqn:local phase step lower bound}
    \end{align*}
    by \cref{lemma:local phase step size lower bound}. It is also clear that MINRES selects $\SOL$. We first show that the step size, $\alpha_k$, can be uniformly lower bounded for $\xx_k$ close enough to $\xx_*$. We do this by showing that the step $\| \pp_k^\sI\|$ can be upper bounded. Specifically, due to \cref{eqn:local phase Krylov subspace regularity} the step, $\pp_k^\sI$, is upper bounded by the gradient magnitude (cf. \cref{eqn:SOL step gradient related} with $\varrho = \mu$)
   \begin{align*}
      \| \pp_k^\sI \| \leq \| \bgg_k^\sI\|/\mu. \tageq\label{eqn:local phase pk gradient bound}
   \end{align*}
   Next we use the continuity of $\grad f(\xx)$ and the fact that $\bgg_*^{\sI(\xx_*, 0)} = 0$ to choose $\Delta_0 \leq \min\{\Delta_{\text{sep}}, \Delta_{\text{cvx}} \}$ such that $\xx_k \in B(\xx_*, \Delta_0)$ implies 
   \begin{align*}
       \|\bgg_k^{\sI(\xx_k, \delta)}\|= \| \bgg_k^{\sI(\xx_*, 0)} \| \leq  \frac{\mu\delta}{2},
   \end{align*}
   where in the first equality we used \cref{lemma:local phase active set separation}. This implies
   \begin{align*}
       \| \pp_k^\sI \| \leq \delta/2, \tageq\label{eqn:local phase MR step bound}
   \end{align*}
   and hence by \cref{eqn:local phase step lower bound}
   \begin{align*}
       \bar{\alpha} \defeq \min \left\{1, \frac{2(1 - \rho)\mu}{L_{g}} \right\} .
   \end{align*}
   
   We now show that $\sA(\xx^*, 0) \subseteq \sA(\xx_{k+1}, 0)$. Let $i \in \sA(\xx^*, 0)$. Define $\ee_k = \xx_k - \xx_*$. By \cref{ass:local minima nondegeneracy} and the continuity of $\grad f(\xx)$, there exists $\Delta_1$ such that, for $\| \ee_k \| \leq \Delta_1$,
    \begin{align*}
        (\bgg(\xx_k))^j = (\bgg(\xx_* + \ee_k))^j > \frac{\gamma}{2}, \ \forall j \in \sA(\xx^*, 0).
    \end{align*}
    Consider $\Delta_2 = \min \left\{\Delta_0, \Delta_1, \overline{\alpha}\gamma/2 \right\}$. If $\xx_k \in B(\xx_*, \Delta_2)$, we have 
    \begin{align*}
        \xx_k^i = \xx_k^i - \xx_*^i < \Delta_2 \leq  \frac{\overline{\alpha}\gamma}{2}.
    \end{align*}
    Using this bound and the lower bound for the gradient and step size we compute the update as
    \begin{align*}
        \xx_k^i - \alpha_k \bgg_k^i  \leq \xx_k^i - \frac{\overline{\alpha} \gamma }{2} \leq 0 \implies \xx_{k+1}^i = \sP( \xx_k^i + \alpha_k \pp_k^i) = 0,
    \end{align*}
    which implies $i \in \sA(\xx_{k+1}, 0)$. 
    
    Next, we show $\sA(\xx^*, 0) \supseteq \sA(\xx_{k+1}, 0)$. In particular, we prove the contrapositive statement $i \in \sI(\xx_*, 0) \implies i \in \sI(\xx_{k+1}, 0)$. Suppose $i \in \sI(\xx_*, 0)$ the result will follow by showing that the $\xx_{k+1}$ remains bounded away from zero. Let $\Delta_{\text{bnd}} = \min\{ \Delta_0, \Delta_1, \Delta_2, \Delta_{\text{sep}}, \Delta_{\text{cvx}} \}$. Having $\xx_k \in B(\xx_*, \Delta_{\text{bnd}})$ implies $\sI(\xx_*, 0) = \sI(\xx_k, \delta) $. Additionally, the bound \cref{eqn:local phase MR step bound} applies and so $\alpha \in [\overline{\alpha}, 1]$ implies
    \begin{align*}
        \alpha |\pp_k^i|  \leq  \| \pp_k^{\sI(\xx_k, \delta)} \| \leq \delta/2,
    \end{align*}
    which yields
    \begin{align*}
        \xx_k^i + \alpha\pp_k^i \geq \xx_k^i - \alpha | \pp_k^i | \geq \xx_k^i - \frac{\delta}{2} > \frac{\delta}{2}, \tageq\label{eqn:local phase inactive step lower bound}
    \end{align*}
    where the final inequality follows from $i \in \sI(\xx_*, 0) =  \sI(\xx_k, \delta) \implies \xx_k^i > \delta$. Finally we compute the step as
    \begin{align*}
        \xx_{k+1} = \sP(\xx_k^i + \alpha \pp_k^i) = \xx_k^i + \alpha\pp_k^i > 0,  
    \end{align*}
    which is the result.  \\
    
    Now for the second part of the result. Fix $\Delta > 0$. From the first part of the result we know that for $\xx_k \in B(\xx_*, \Delta_{\text{bnd}})$ we have $\sA(\xx_k, \delta) = \sA(\xx_{k+1}, 0) = \sA(\xx_*, 0)$, which implies $\xx_{k+1}^{\sA(\xx_k, \delta)} = \xx_{k+1}^{\sA(\xx_{k+1}, 0)} = 0$ and $\xx_*^{\sA(\xx_k, \delta)} = \xx_*^{\sA(\xx_*, 0)} = 0$. Applying these equalities we obtain
    \begin{align*}
        \|\xx_{k+1} - \xx_*\| &= \left\|\xx_{k+1}^{\sI(\xx_k, \delta)}- \xx_*^{\sI(\xx_k, \delta)}\right\| \\
        &= \left\| [\sP(\xx_{k}  + \alpha_k\pp_k)]^{\sI(\xx_k,\delta)} -  \xx_*^{\sI(\xx_k, \delta)}\right\|   \\
        &= \left\| [\xx_{k}  + \alpha_k\pp_k]^{\sI(\xx_k, \delta)} -  \xx_*^{\sI(\xx_k, \delta)}\right\|  \\
        &\leq \left\| \xx_k^{\sI(\xx_k, \delta)} - \xx_*^{\sI(\xx_k, \delta)} \right\| + \alpha_k \left\|\pp_k^{\sI(\xx_k, \delta)} \right\| \\
        &\leq \|\xx_k - \xx_*\| + \left\|\bgg_k^{\sI(\xx_k, \delta)}\right\|/\mu,
    \end{align*}
    where we drop the projection on line three due to $[\xx_k + \alpha_k \pp_k]^{\sI(\xx_k,\delta)} > \delta/2$ when $\xx_k \in B(\xx_*, \Delta_{\text{bnd}}) $ (cf. \cref{eqn:local phase inactive step lower bound}). Again, $\bgg_k^{\sI(\xx_k, \delta)} = \bgg_k^{\sI(\xx_*, 0)}$ so, by the continuity of $\grad f(\xx)$, we are free to choose $\Delta_{3}$ so that $\xx_k \in B(\xx_*, \Delta_3)$ implies
    \begin{align*}
        \left\| \bgg_k^{\sI(\xx_k, \delta)} \right\| = \left\| \bgg_k^{\sI(\xx_*, 0)} \right\| < \frac{\mu \Delta}{2  },
    \end{align*}
    Finally, we can choose  $\Delta_{\text{cls}} = \min\{ \Delta_{\text{bnd}}, \Delta_3, \Delta/2 \}$ so that, if $\xx_k \in B(\xx_*, \Delta_{\text{cls}})$, we have 
    \begin{align*}
        \|\xx_{k+1} - \xx_*\| &\leq \|\xx_k - \xx_*\| + \left\|\bgg_k^{\sI(\xx_k, \delta)}\right\|/\mu \\
        &< \frac{\Delta}{2} + \frac{\Delta}{2} \\
        &= \Delta.
    \end{align*}
    
\end{proof}

The second part of \cref{lemma:close enough to minima lemma} can be used with the choice $\Delta = \Delta_{\text{bnd}}$ to obtain 
\begin{align*}
    \| \xx_{k} - \xx_* \| < \Delta_{\text{cls}} \implies \| \xx_{k+1} - \xx_* \| < \Delta_{\text{bnd}}.
\end{align*}
In this case, we can guarantee, due to $\Delta_{\text{cls}} \leq \Delta_{\text{bnd}}$ and the first part of \cref{lemma:close enough to minima lemma} applied to $\xx_{k}$, that 
\begin{align*}
    \sA(\xx_{k}, \delta) = \sA(\xx_{k+1}, 0) = \sA(\xx_*, 0),
\end{align*}
and from $\xx_{k+1} \in B(\xx_*, \Delta_{\text{bnd}}) $ and the first part of \cref{lemma:close enough to minima lemma} again
\begin{align*}
    \sA(\xx_{k+1}, \delta) = \sA(\xx_{k+2}, 0) = \sA(\xx_*, 0).
\end{align*}
Together, these results show that $\xx_k \in B(\xx_*, \Delta_{\text{cls}})$ implies $\xx_{k+1}, \xx_{k+2} \in X_*$. This means that the iterates of our algorithm essentially ``looks'' like unconstrained minimisation in this subspace. Unfortunately, with the results we have so far, we cannot guarantee that the iterates continue to stay close enough to the minima beyond iteration $k+2$. \cref{lemma:neighbourhood of minima lemma} will overcome this problem by using the second-order sufficient condition and adapting an unconstrained optimisation result \citep[Proposition 1.12]{bertsekasConstrained}. The main idea is that the ``strict convexity'' on $X_*$ induced by \cref{ass:local minima second-order sufficiency} (cf. \cref{eqn:local strong convexity constant}) implies that there exists a small ``basin'' (restricted to $X_*$) that the iterates will not leave once they enter. We can then use Lemma \ref{lemma:close enough to minima lemma} to show that our iterates eventually enter $X_*$ and the corresponding basin.

\begin{lemma}
    \label{lemma:neighbourhood of minima lemma}
    Let $f$ satisfy \cref{ass:LipschitzGradient} and $\xx_*$ be a local minima satisfying \cref{ass:local minima nondegeneracy,ass:local minima second-order sufficiency}. Let $\delta$ be chosen according to \cref{eqn:local convergence delta choice}. If there is an iterate, $\xx_k$, of \cref{alg:newton-mr-two-metric-local-phase} such that $\sA(\xx_k, \delta) = \sA(\xx_k, 0) = \sA(\xx_*, 0)$ and $\sA(\xx_{k+1}, 0) = \sA(\xx_*, 0)$ (i.e. $\xx_{k}, \xx_{k+1} \in X_*$) then there exists a neighbourhood (restricted to $X_*$) of $\xx_*$, $\sN(\xx_*)$, such that if $\xx_{k} \in \sN(\xx_*)$ and then $\xx_{k+1} \in \sN(\xx_*)$. Additionally, $\sN(\xx_*)$ is independent of the iterates and can be chosen arbitrarily small, i.e., for any $\Delta > 0$ we have $\sN(\xx_*) \subset B(\xx_*, \Delta)$.
\end{lemma}
\begin{proof}
    We fix $\Delta \leq \Delta_{\text{cvx}}$ and define
    \begin{align*}
        \sN(\xx_*) = \left\{ \xx \in B(\xx_*, \Delta) \cap X_* \ | \ f(\xx) \leq f(\xx_*) + \frac{ \mu }{2} \left( \frac{\Delta}{1 + L_g/\mu} \right)^2  \right\}.
    \end{align*}
    We will show that this set is the desired neighbourhood on $X_*$ in the sense there exists an open ball in the relative interior of $X_*$. The mean value theorem implies that there is a constant $t \in (0, 1)$ such that for any $\xx, \yy \in \real^d$
    \begin{align*}
        f(\yy) = f(\xx) + \langle \grad f(\xx), \yy - \xx \rangle + \frac{1}{2} \langle \yy - \xx , \grad^2 f (\xx + t(\yy - \xx)) (\yy - \xx) \rangle. \tageq\label{eqn:mean value theorem}
    \end{align*}
    We obtain, by \cref{ass:LipschitzGradient},
    \begin{align*}
        f(\yy) \leq f(\xx) + \langle \grad f(\xx), \yy - \xx \rangle + \frac{L_g}{2}\| \yy - \xx \|^2 .
    \end{align*}
    Let $\xx = \xx_*$ and $\yy \in B(\xx_*, \Delta) \cap X_*$. The fact that $\yy$ and $\xx_*$ only have nonzero components in $\sI(\xx_*, 0)$, while the optimality condition \cref{eqn:inactive set termination condition} implies $\grad f(\xx_*)$ only has zero components in $\sI(\xx_*, 0)$, allows us to write
    \begin{align*}
        f(\yy) \leq f(\xx_*) + \frac{L_g}{2} \| \yy - \xx_* \|^2,
    \end{align*}
    so by choosing $\yy$ close enough to $\xx_*$ we have $\yy \in \sN(\xx_*)$. This implies that $\sN(\xx_*)$ is a neighbourhood of $\xx_*$, in $X_*$.

    Let $\xx = \xx_*$ and $\yy = \xx_k$ for $\xx_k \in B(\xx_*, \Delta) \cap X_* $. Then, $\xx_* + t(\xx_k - \xx_*) \in B(\xx_*, \Delta) \cap X_* $ for any $t \in [0, 1]$. Hence, $\cref{eqn:local strong convexity constant}$ applied to \cref{eqn:mean value theorem} yields
    \begin{align*}
        \frac{\mu}{2} \| \xx_k - \xx_* \|^2 \leq f(\xx_k) - f(\xx_*). \tageq\label{eqn:fstar lower bound }
    \end{align*}
    Next, we seek to bound the distance between subsequent errors. Since $\sA(\xx_k, \delta) = \sA(\xx_{k+1}, 0) = \sA(\xx_*, 0)$, we have
    \begin{align*}
        \xx_{k+1}^{\sA(\xx_{k}, \delta)} = \xx_{k+1}^{\sA(\xx_{k+1}, 0)} = 0,
    \end{align*}
    and 
    \begin{align*}
        \xx_*^{\sA(\xx_k, \delta)} = \xx_*^{\sA(\xx_*, 0)} = 0.
    \end{align*}
    We compute 
    \begin{align*}
        \| \xx_{k+1} - \xx_* \| &= \| \xx_{k+1}^{\sI(\xx_{k+1}, 0)} - \xx_*^{\sI(\xx_*, 0)} \| \\
        &= \| \xx_{k+1}^{\sI(\xx_k, \delta)} - \xx_*^{\sI(\xx_k, \delta)} \| \\
        &= \| [\sP(\xx_k + \alpha_k \pp_k)]^{\sI(\xx_{k}, \delta)} - \xx_*^{\sI(\xx_k, \delta)} \| \\
        &= \| \xx_k^{\sI(\xx_k, \delta)} - \xx_*^{\sI(\xx_k, \delta)} + \alpha_k \pp_k^{\sI(\xx_k, \delta)}  \| \\
        &\leq \| \xx_k^{\sI(\xx_k, \delta)} - \xx_*^{\sI(\xx_k, \delta)}\| + \alpha_k\|\pp_k^{\sI(\xx_k, \delta)} \|, \tageq\label{eqn:local phase error bound}
    \end{align*}
    where the fourth line follows from $\sI(\xx_k, \delta) = \sI(\xx_{k+1}, 0)$ and $i \in \sI(\xx_{k+1}, 0)$ implying that $0 < \xx_{k+1}^i = \sP(\xx_k^i + \alpha_k \pp_k^i) \implies \sP(\xx_k^i + \alpha_k \pp_k^i) = \xx_{k}^i + \alpha_k \pp_k^i $. Since $\sI(\xx_k, \delta) = \sI(\xx_*, 0)$ and $\xx_k \in B(\xx_*, \Delta_{\text{cvx}})$, we know \cref{eqn:local phase pk gradient bound} holds. We can refine \cref{eqn:local phase pk gradient bound} by combining \cref{ass:LipschitzGradient}, $\sI(\xx_k, \delta) = \sI(\xx_k, 0) = \sI(\xx_*, 0)$ and \cref{eqn:inactive set termination condition} to obtain
    \begin{align*}
        \left\| \bgg_k^{\sI(\xx_k, \delta)}\right\| = \left\| \bgg_k^{\sI(\xx_k, 0)} - \bgg_*^{\sI(\xx_*,0)} \right\| \leq \left\| \bgg_k- \bgg_*\right\| \leq L_g\| \xx_k - \xx_*\| = L_g\left\| \xx_k^{\sI(\xx_k, \delta)} - \xx_*^{\sI(\xx_k, \delta)} \right\|.
    \end{align*}
    Combining this bound, $\alpha_k \leq 1$, \cref{eqn:local phase pk gradient bound,,eqn:local phase error bound} we have 
    \begin{align*}
        \| \xx_{k+1} - \xx_* \| &\leq \| \xx_k^{\sI(\xx_k, \delta)} - \xx_*^{\sI(\xx_k, \delta)}\| + \frac{L_g}{\mu}\left\| \xx_k^{\sI(\xx_k, \delta)} - \xx_*^{\sI(\xx_k, \delta)} \right\| \\
        &= \left( 1 + \frac{L_g}{\mu}\right)\| \xx_k^{\sI(\xx_k, \delta)} - \xx_*^{\sI(\xx_k, \delta)}\|\\
        &= \left( 1 + \frac{L_g}{\mu}\right)\| \xx_k - \xx_*\|. \tageq\label{eqn:xk+1-xstar bound}
    \end{align*}
    We now show that this is enough to guarantee that $\xx_{k+1} \in \sN(\xx_*)$. In particular, if $\xx_k \in \sN(\xx_*)$ then $\xx_k \in B(\xx_*, \Delta_{cvx})$, so by combining the definition of $\sN(\xx_*)$ and \cref{eqn:fstar lower bound } we have 
    \begin{align*}
        \frac{\mu}{2} \| \xx_k - \xx_*  \|^2 \leq f(\xx_k) - f(\xx_*) \leq \frac{\mu}{2} \left( \frac{\Delta}{1 + L_g/\mu} \right)^2  
        \implies  \| \xx_k - \xx_*  \| < \frac{\Delta }{1 + L_g/\mu}.
    \end{align*}
    Applying \cref{eqn:xk+1-xstar bound} we have
    \begin{align*}
        \| \xx_{k+1} - \xx_*\| < \Delta,
    \end{align*}
    which implies that $\xx_{k+1}\in B(\xx_*, \Delta) \cap X_* $. In addition, $\alpha_k$ satisfies the line search criterion, which guarantees that $f(\xx_{k+1}) \leq f(\xx_k)$ and so
    \begin{align*}
        f(\xx_{k+1}) - f(\xx_*) \leq f(\xx_k) - f(\xx_*) \leq \frac{\mu}{2} \left( \frac{\Delta}{1 + L_g/\mu} \right)^2,
    \end{align*}
    which implies $\xx_{k+1} \in \sN(\xx_*)$, as needed. In the above argument we are free to replace $\Delta$ with any $\Delta' \in (0, \Delta]$ which implies that we can always choose $\sN(\xx_*)$ sufficiently small.
\end{proof}

We our now ready to prove \cref{thm:active set identification}.

\begin{proof}[Proof of \cref{thm:active set identification}]
    Note that we are free to choose the neighbourhood $\sN(\xx_*)$ of $\xx_*$ on $X_*$ from Lemma \ref{lemma:neighbourhood of minima lemma} small. We therefore select $\Delta_0 < \Delta_{\text{cvx}}$ and $ \sN(\xx_*)$ to satisfy the following inclusions
    \begin{align*}
        B(\xx_*, \Delta_0) \cap X_* \subseteq \sN(\xx_*) \subseteq B(\xx_*\Delta_{\text{bnd}}) \cap X_*. \tageq\label{eqn:local phase neighbourhood inclusions}
    \end{align*}
    By the second part of Lemma \ref{lemma:close enough to minima lemma}, there exists $\Delta_{\text{cls}} \leq \Delta_{\text{bnd}}$ such that the following inclusions hold
    \begin{align*}
        \xx_{k} \in B(\xx_*, \Delta_{\text{cls}}) \implies \xx_{k+1} \in B(\xx_*, \Delta_{0}). \tageq\label{eqn:local phase xk close inclusions}
    \end{align*}
    Choose $\Delta_{\text{actv}} = \Delta_{\text{cls}}$ and suppose that $\xx_{\bar{k}} \in B(\xx_*, \Delta_{\text{actv}})$. The first inclusion of \cref{eqn:local phase xk close inclusions}, implies $\sA(\xx_{\bar{k}}, \delta) = \sA(\xx_{\bar{k}+1}, 0) = \sA(\xx_*, 0)$, i.e. $\xx_{\bar{k}+1} \in X_*$,  by $\Delta_\text{cls} \leq \Delta_{\text{bnd}}$ and the first part of Lemma \ref{lemma:close enough to minima lemma}. This fact and the second inclusion of \cref{eqn:local phase xk close inclusions}, implies $\xx_{\bar{k}+1} \in \sN(\xx_*)$ and therefore, by the second inclusion of \cref{eqn:local phase neighbourhood inclusions}, $\sA(\xx_{\bar{k}+1}, \delta) = \sA(\xx_{\bar{k}+2}, 0) = \sA(\xx_*, 0)$, Again by the first part of Lemma \ref{lemma:close enough to minima lemma}. Combining what we have so far, we obtain $\sA(\xx_{\bar{k}+1}, \delta) = \sA(\xx_{\bar{k}+1}, 0) = \sA(\xx_*, 0)$, which is the result for $\bar{k}+1$. Additionally, however, we can apply \cref{lemma:neighbourhood of minima lemma} applied to the iterate $\bar{k} + 1$ to obtain $\xx_{\bar{k}+2} \in \sN(\xx_*)$. The argument for $\bar{k} + 1$ may now be repeated for $k \geq \bar{k} + 2$. For instance, \cref{eqn:local phase neighbourhood inclusions} and $\xx_{\bar{k}+2} \in \sN(\xx_*)$ implies $\xx_{\bar{k}+2} \in B(\xx_*, \Delta_\text{bnd})$ and so $\sA(\xx_{\bar{k}+2}, \delta) = \sA(\xx_{\bar{k}+3}, 0) = \sA(\xx_*, 0)$ by Lemma \ref{lemma:close enough to minima lemma} and $\xx_{\bar{k}+3} \in \sN(\xx_*)$ by Lemma \ref{lemma:neighbourhood of minima lemma}, which yields the result for $\bar{k}+2$ and sets up the argument for $\xx_{\bar{k}+3}$. Continuing in this fashion yields the result for the given $\Delta_{\text{actv}}$. 
\end{proof}

\section{Further Details and Extended Numerical Results}\label{apx:Numerical Results}

In this section we provide some additional elements of our proposed methods, further details on our experimental setup, and also give a more complete description of various problems we consider for our numerical simulations.

\subsection{Line Search Algorithms}
Here, we gather the line search algorithms used for the theoretical analysis as well as the empirical evaluations of our methods. 
\begin{algorithm}[htbp]
    \begin{algorithmic}[1]
        \STATE \textbf{input} Initial step size $\alpha_0$, Line search criterion, Scaling parameter $0 < \zeta < 1$.
        \vspace{1mm}
        \STATE $\alpha \gets \alpha_0$.
        \vspace{1mm}
        \WHILE{Line search criterion is not satisfied}
        \vspace{1mm}
            \STATE $\alpha \gets \zeta \alpha$.
            \vspace{1mm}
        \ENDWHILE
        \vspace{1mm}
        \STATE \textbf{return} $\alpha$.
        \vspace{1mm}
    \end{algorithmic}
    \caption{Backward Tracking Line Search.}
    \label{alg:back tracking line search}
\end{algorithm}

\begin{algorithm}[htbp]
    \begin{algorithmic}[1]
        \STATE \textbf{input} Initial step size $\alpha_0$, Line search criterion, Scaling parameter $0 < \zeta < 1$.
        \vspace{1mm}
        \STATE $\alpha \gets \alpha_0$.
        \vspace{1mm}
        \IF {Line search criterion is not satisfied} 
		\vspace{1mm}
		\STATE \text{Call \cref{alg:back tracking line search}}
		\vspace{1mm}
		\ELSE
		\vspace{1mm}
		\WHILE {Line search criterion is satisfied}
		\vspace{1mm}
		\STATE $ \alpha = \alpha/\zeta $
		\vspace{1mm}
		\ENDWHILE
        \vspace{1mm}
		\STATE \textbf{return} $\zeta\alpha$.
		\vspace{1mm}
        \ENDIF
    \end{algorithmic}
    \caption{Forward/Backward Tracking Line Search}
    \label{alg:forward tracking line search}
\end{algorithm}

\subsection{Smooth Reformulation of Nonsmooth $\ell_1$ Regression} \label{apx:smooth L1 regularisation}

Consider $\ell_1$ regularisation of a smooth function, $f$, as given in \cref{eqn:L1 objective}. Unfortunately, even when $f$ is smooth, the objective \cref{eqn:L1 objective} is non-differentiable when $\xx^i = 0$ for some $i = 1, \ldots, d$. However, it was shown in \citet{Schmidz2007FastOptimizationMethodsForL1Regularization} that one can reformulate \cref{eqn:L1 objective} into a smooth problem by splitting $\xx$ into positive and negative parts, i.e., $\xx_+ = \max(0, \xx) $ and $\xx_- = - \min(0, \xx)$, where ``$\max$'' and ``$\min$'' are taken elementwise. Indeed, we have the identities 
\begin{align*}
    \xx^i = \xx_+^i - \xx_-^i,
\end{align*}
and 
\begin{align*}
    |\xx^i| = \xx_+^i + \xx_-^i,
\end{align*}
which we can use to reformulate \cref{eqn:smooth L1 objective} as a constrained problem on $\real^{2d}$. In particular, the following auxiliary function is equivalent to the objective of \cref{eqn:L1 objective} 
\begin{align*}
     F(\xx_+, \xx_-) \defeq f(\xx_+ - \xx_-) + \lambda \sum_{i=1}^d (\xx_+^i + \xx_-^i). 
\end{align*}
If we make the identification $\zz = [\xx_+, \xx_-] \in \real^{2d}$, we obtain the auxiliary minimisation problem defined by
\begin{align}
    \min_{\zz \in \real^{2d}}\; F(\zz) \quad \text{subject to } \quad \zz \geq \zero. \tageq\label{eqn:smooth L1 objective}
\end{align}
The nonpositivity condition in \cref{eqn:smooth L1 objective} ensures that $\zz$ can be interpreted as the positive and negative part of the underlying variable, $\xx$.
The gradient and Hessian of the auxiliary function, $F$, are given by
\begin{align*}
    \grad F(\xx_+, \xx_-) =  \begin{pmatrix}
        \grad  f(\xx_+ - \xx_-) + \lambda \one_{d \times 1} \\
        -\grad f(\xx_+ - \xx_-) + \lambda \one_{d \times 1}
    \end{pmatrix}, 
    \quad \grad^2 F(\xx_+, \xx_-) = \begin{pmatrix}
        \grad^2 f(\xx_+ - \xx_-) & -\grad^2 f(\xx_+ - \xx_-)\\
        -\grad^2 f(\xx_+ - \xx_-) & \grad^2 f(\xx_+ - \xx_-)
    \end{pmatrix}.
\end{align*}

\begin{remark}[Evaluating the gradients and Hessian-vector products]
    Clearly, evaluating the gradient of $F$ requires only a single evaluation of the original gradient, $\grad f $. On the other hand, for computing a Hessian-vector product of $F$ with a vector $\vv = ( \vv_1^\transpose, \vv_2^\transpose)^\transpose \in \real^{2d}$, we have
\begin{align*}
    \grad^2 F(\xx_+, \xx_-) \vv 
    &= \begin{pmatrix}
        \grad^2 f(\xx_+ - \xx_-)\vv_1 - \grad^2 f(\xx_+ - \xx_-)\vv_2\\
        -\grad^2 f(\xx_+ - \xx_-)\vv_1 + \grad^2 f(\xx_+ - \xx_-)\vv_2
    \end{pmatrix},
\end{align*}
which requires two Hessian-vector products of the original function $f$ in the form $\grad^2 f(\xx) \ww$ where $\ww \in \real^{d}$.
\end{remark}

As a sanity check, we show that the first order stationary points of \cref{eqn:smooth L1 objective,eqn:L1 objective} coincide. The first order necessary conditions for \cref{eqn:smooth L1 objective} imply that for $j = 1, \ldots, d $ and $i = j$,
\begin{align*}
    [\xx_+^*]^i \geq \zero, \quad \text{and} \quad
        \begin{cases}
            [\grad f(\xx^*_+ - \xx_-^*)]^i + \lambda  = 0, \ &\text{if } [\xx_+^*]^i > 0,  \\
            [\grad f(\xx^*_+ - \xx_-^*)]^i + \lambda \geq 0, \ &\text{if } [\xx_+^*]^i = 0,\\
        \end{cases} \tageq\label{eqn:smooth L1 optimality condition - positive}
\end{align*}
and for $j = d+1, \ldots, 2d$ and $i = j-d$,
\begin{align*}
    [\xx_-^*]^i \geq \zero, \quad \text{and} \quad
        \begin{cases}
            -[\grad f(\xx^*_+ - \xx_-^*)]^i + \lambda  = 0, \ &\text{if } [\xx_-^*]^i > 0, \\
            -[\grad f(\xx^*_+ - \xx_-^*)]^i + \lambda \geq 0, \ &\text{if } [\xx_-^*]^i = 0. 
        \end{cases} \tageq\label{eqn:smooth L1 optimality condition - negative}
\end{align*}
On the other hand, the first order stationary points of the problem \eqref{eqn:L1 objective} can be expressed in terms of  the Clarke subdifferential \citep[Chapter 2]{clarke1990optimization} as those points $\xx^*$ for which $\zero \in \grad f(\xx^*) + \partial \| \xx^*\|_1$. That is, for $i =1, \ldots, d$, we have
\begin{align}
      \begin{cases}
        [\grad f(\xx^*)]^i + \lambda = 0   \ &\text{if } [\xx^*]^i > 0, \\
        [\grad f(\xx^*)]^i - \lambda = 0  \ &\text{if } [\xx^*]^i < 0, \\
        \left| [\grad f(\xx^*)]^i \right|  \leq \lambda  \ &\text{if } [\xx^*]^i = 0. 
    \end{cases}    \label{eqn:L1 optimality condition}
\end{align}
We first show that if $\zz^* = [\xx_+^*, \xx_-^*]$ satisfies \cref{eqn:smooth L1 optimality condition - positive,eqn:smooth L1 optimality condition - negative} then $\xx^* = \xx^*_+ - \xx^*_-$ satisfies \eqref{eqn:L1 optimality condition}. First, suppose $[\xx^*]^i > 0$. In this case, we must have $[\xx^*]^{i} = [\xx_+^*]^i > 0 = [\xx_-^*]^i$, which from the first case of \cref{eqn:smooth L1 optimality condition - positive} implies the first case of \cref{eqn:L1 optimality condition}. When $[\xx^*]^i < 0$, since $[\xx^*]^{i} = [\xx_-^*]^i > 0 = [\xx_+^*]^i$, the first case of \cref{eqn:smooth L1 optimality condition - negative} implies the first case of \cref{eqn:L1 optimality condition}. Finally, when $[\xx^*]^i = 0$, we have $[\xx_+^*]^i = [\xx_-^*]^i = 0$, and we appeal to the second case of both \cref{eqn:smooth L1 optimality condition - positive,eqn:smooth L1 optimality condition - negative} to obtain 
\begin{align*}
    [\grad f(\xx_*)]^i \geq - \lambda, \quad \text{and}  \quad [\grad f(\xx_*)]^i \leq \lambda,
\end{align*}
which implies $\left|[\grad f(\xx_*)]^i\right| \leq \lambda$, i.e., the third case of \cref{eqn:L1 optimality condition}.

We now show  that if  $\xx^* = \xx^*_+ - \xx^*_-$ satisfies \eqref{eqn:L1 optimality condition}, then if $\zz^* = [\xx_+^*, \xx_-^*]$ satisfies \cref{eqn:smooth L1 optimality condition - positive,eqn:smooth L1 optimality condition - negative}. Consider the first case of \eqref{eqn:L1 optimality condition}. Noting again that $[\xx^*]^{i} = [\xx_+^*]^i > 0 = [\xx_-^*]^i$, it clearly implies the first and the second cases of \cref{eqn:smooth L1 optimality condition - positive,,eqn:smooth L1 optimality condition - negative}, respectively (recall $\lambda > 0$). Similarly, the second case of \eqref{eqn:L1 optimality condition} implies the second and the first cases of \cref{eqn:smooth L1 optimality condition - positive,,eqn:smooth L1 optimality condition - negative}, respectively. Finally, it is clear that the third case of \eqref{eqn:L1 optimality condition} implies the second case for both \cref{eqn:smooth L1 optimality condition - positive,,eqn:smooth L1 optimality condition - negative}.

\subsection{Additional Experimental Details}\label{apx:parameter settings}
\paragraph{Oracle Calls as Complexity Measure}
Following the typical convection in the optimisation literature, in all our experiments, we plot the objective value against the total number of oracle calls for function, gradient, and Hessian-vector product evaluations. We adopt this approach because the measurement of ``wall-clock'' time can be heavily dependent on specific implementation details and computational platform. In contrast, counting the number of equivalent function evaluations, as an implementation and system independent unit of complexity is more appropriate and fair. More specifically, upon evaluating the function, computing its gradient is equivalent to one additional function evaluation, and computing a Hessian-vector product requires two additional function evaluations compared to a gradient evaluation \cite{pearlmutter1994fast}. For example, in neural networks, for a given data at the input layer, evaluation of network's output, i.e., function evaluation, involves one forward propagation. The corresponding gradient is computed by performing one additional backward propagation. After computing the gradient, an additional forward followed by a backward propagation give the corresponding Hessian-vector product \cite{goodfellow2016deep}. 

\paragraph{Parameter Settings} In all experiments we set $\epsilon_k = \delta_k = \sqrt{\epsilon_g}$ as per \cref{thm:first-order iteration complexity}. For the Newton-MR TMP we set the inexactness condition for MINRES, i.e., \cref{eqn:MINRES termination tolerance}, to $\eta = 10^{-2}$ for convex problems and $\eta = 1$ for nonconvex problems. We apply a less stringent tolerance in the nonconvex case to maximise the chances of terminating early with a ''good enough'' SOL type solution. Indeed, running the solver too long increases the odds that spurious negative curvature direction will arise as part of iterations. Since such directions never occur in convex settings, one can afford to solve the subproblems more accurately. 

For projected Newton-CG, we use the parameter settings from the experiments in \citet{XieWright2021ComplexityOfProjectedNewtonCG}. Specifically, in the notation of \citet{XieWright2021ComplexityOfProjectedNewtonCG}, we set the accuracy parameter and back tracking parameter to $\zeta=\theta=0.5$ and the step acceptance parameter to $\eta =0.2$. Furthermore, following the algorithmic description of \citet{XieWright2021ComplexityOfProjectedNewtonCG}, and to have equivalent termination conditions, we modify the gradient negativity check from $\bgg_k^i < -\epsilon_k^{3/2}$ to $\bgg_k^i < -\epsilon_k$ for this method. 

For projected gradient and Newton-MR TMP, we set the scaling parameter in \cref{alg:back tracking line search,alg:forward tracking line search} to $\zeta = 0.5$ and the sufficient decrease parameter to $\rho=10^{-4}$. All line searches are initialised from $\alpha_0 = 1$. We note that, for both FISTA and PGM, we terminate the iterations when $|(f(\xx_{k}) + \lambda\| \xx_k\|_1 - (f(\xx_{k-1}) + \lambda\| \xx_{k-1}\|_1)| < 10^{-8}$ on the $\ell_1$ problem and $|f(\xx_{k}) - f(\xx_{k-1})| < 10^{-8}$ otherwise. We set the momentum term in PGM to $\beta=0.9$ and select the fixed step size by starting from $\alpha=1$ and successively shrinking the step size by a factor of 10 until the iterates are stable for the duration of the experiment, i.e., no divergence or large scale oscillations. This procedure resulted in a step size of $\alpha=10^{-3}$ for the $\ell_1$ MLP (\cref{fig:MLP fashion MNIST}) and $\alpha=1$ for the NNMF problems (\cref{fig:NNMF cosine text,fig:NNMF nonconvex regularisation}). 

We now give a more complete description of each of the objective functions.

\paragraph{Multinomial Regression}

We first consider is the problems of multinomial regression on $C$ classes. Specifically, consider a set of data items $\{\aa_i, b_i \}_{i=1}^n \subset \real^d \times \{1, \ldots C\}$. Denote the weights of each class as $\xx_1, \ldots, \xx_{C}$ and define $\xx = [ \xx_1, \ldots, \xx_{C-1}]$. We are free to take $\xx_C = \zero$ as class $C$ is identifiable from the weights of the other classes. The objective, $f$, is given by 
\begin{align*}
    f(\xx) = \frac{1}{n} \sum_{i=1}^n  \sum_{c=1}^{C-1} -\one(b_i = c) \log{(\text{softmax}(\xx_c, \aa_i))}, \tageq\label{eqn:multinomial regression objective}
\end{align*}
where $\one(\cdot)$ is the indicator function and 
\begin{align*}
    \text{softmax}(\xx_c, \aa_i) = \frac{\exp{(\langle \xx_c, \aa_i \rangle)}}{\sum_{c=1}^C \exp{(\langle \xx_c, \aa_i \rangle)}}.
\end{align*}
In this case, the objective is convex. We allow for a constant term in each set of weights, $\xx_c$, which we do not apply the $\ell_1$ penalisation to. 

All methods for this example are initialised from $\xx_{0} = \zero$.

\paragraph{Neural Network} Again, suppose we have a set of data items $\{\aa_i, b_i \}_{i=1}^n \subset \real^d \times \{1, \ldots C\}$. We consider a small two layer network with a smooth activation function. Specifically, we consider the sigmoid weighted linear unit (SiLU) activation \citep{elfwing2017sigmoidweighted} defined by
\begin{align*}
    \sigma(x) = \frac{x}{1 + e^{-x}}.
\end{align*}
We note that the SiLU activation is similar to ReLU and is the product of a linear activation with a standard sigmoid activation. We define a network, $\hh(\cdot; \xx)$ parameterised by the weights, $\xx$, with the following architecture 
\begin{align*}
    \text{Input (d)} \to \text{Linear (100)} \to \text{SiLU} \to \text{Linear (100)} \to \text{SiLU} \to \text{Linear (10)},
\end{align*}
where the number in brackets denotes the size of the output from the layer. Note that we allow for a bias term in each linear layer which we do not apply the $\ell_1$ penalty to. The objective function, $f$, is given by cross entropy loss incurred by the network over the entire dataset 
\begin{align*}
    f(\xx) = - \frac{1}{n}\sum_{i=1}^n \log \left( \frac{ \exp([\hh(\aa_i ; \xx)]^{b_i}) }{\sum_{c=1}^C \exp([\hh(\aa_i; \xx)]^c ) } \right). \tageq\label{eqn:MLP objective function}
\end{align*}
The weights for layer $i$, denoted $\xx_i$, are initialised with the default PyTorch initialisation, that is, via independent uniform draw
\begin{align*}
    \xx_i \sim U(-\sqrt{k}, \sqrt{k}),
\end{align*}
where $k = 1/(\#\text{Inputs})$ with $(\#\text{Inputs})$ the number of input features into the layer.

\paragraph{NNMF Problem}

A common choice for \cref{eqn:NNMF} is a standard Euclidean distance function
\begin{align*}
    D(\YY, \WW \HH) = \frac{1}{n m} \vnorm{\YY - \WW \HH}^2_F, \tageq\label{eqn:NNMF euclidean distance}
\end{align*}
where $\| \cdot \|_F$ is the Frobenius matrix norm. In this case, \cref{eqn:NNMF} is nonconvex in both $\WW$ and $\HH$, when considered simultaneously, but convex so long as one of the variables is held fixed. This motivates the standard approach to solving \cref{eqn:NNMF} based on alternating updates to $\WW$ and $\HH$ \citep{gillis2014NonnegativeMatrixFactorisation} where one variable is fixed while optimise over the other (e.g., alternating nonnegative least squares). 

By contrast, to test our algorithm, we specifically consider solving \cref{eqn:NNMF} as a nonconvex problem in $\WW$ and $\HH$ simultaneously\footnote{Our algorithm could be employed as a subproblem solver in alternating schemes on $\WW$ and $\HH$. Indeed, the original Bertsekas TMP has been applied for this purpose \citep{Kuang2012SymmetricNNMF}.}. For our first experiment (\cref{fig:NNMF cosine text}), we consider a text data application. When comparing text documents, we aim to have a similarity measure that is \textit{independent} of document length. Indeed, we consider documents similar if they have similar word frequency \textit{ratios}. This notion of similarity is naturally captured by measuring alignment between vectors, which motivates the use of a loss function based on cosine similarity as
\begin{align*}
    D(\YY, \WW\HH) = \frac{1}{n}\sum_{i=1}^n 1 - \cos{(\theta(\yy_i, (\WW\HH)_i) )}, \tageq\label{eqn:NNMF cosine loss}
\end{align*}
where $\theta(\yy_i, (\WW\HH)_i)$ is the angle between the $i$th predicted and true document. This loss function only considers the alignment between documents. Indeed, we can write
\begin{align*}
    \cos{(\theta(\yy_i, (\WW\HH)_i) )} = \frac{ \left \langle \YY_i, (\WW\HH)_i \right\rangle }{ \| \YY_i \| \| (\WW \HH)_i \| }.
\end{align*}
However, using this representation it is clear that, due to the nonnegativity of $\YY$ and $\WW \HH$, \cref{eqn:NNMF cosine loss} ranges between 0 and 1. It is also clear that \cref{eqn:NNMF cosine loss} is equivalent to a Euclidean distance with normalisation
\begin{align*}
    \frac{1}{n} \sum_{i=1}^n \left \| \frac{\YY_i}{\| \YY_i \|} - \frac{(\WW \HH)_i}{\|(\WW \HH)_i\|}\right \|^2.
\end{align*}

In our second example (\cref{fig:NNMF nonconvex regularisation}), we consider \cref{eqn:NNMF} with a standard Euclidean distance function \cref{eqn:NNMF euclidean distance} and a nonconvex regularisation term $R_\lambda$. Specifically, we consider a version of the smooth clipped absolute deviation regularisation (SCAD) first proposed in \citet{fan2001SCAD}. SCAD uses a quadratic function to smoothly interpolate between a regular $\ell_1$ penalty and a constant penalty
\begin{align*}
    \text{SCAD}_{\lambda, a} (x)  = \begin{cases}
        \lambda |x|, \quad  & |x| < \lambda, \\
        \frac{a \lambda |x| - x^2 - \lambda^2}{a-1}, \quad & \lambda \leq |x| < a \lambda,  \\
        \frac{\lambda^2(a + 1)}{2}, \quad & |x| \geq a \lambda. 
    \end{cases}
\end{align*}
The SCAD penalty reduces the downward bias on large parameters typical of the $\ell_1$ penalty while still allowing for sparsification of small parameters. We consider a twice smooth clipped absolute deviation, which we call TSCAD. TSCAD replaces the quadratic interpolant with a quartic, $Q_{\lambda, a}(x)$, which allows for a twice continuously differentiable penalty
\begin{align*}
    \text{TSCAD}_{\lambda, a}(x) = \begin{cases} 
        \lambda |x|,  \quad & |x| < \lambda, \\
        Q_{\lambda, a}(x), \quad & \lambda \leq | x| < a \lambda, \\
        \frac{(a + 1) \lambda^2}{2}, \quad & |x| \geq a \lambda.
    \end{cases}
\end{align*}
The regularisation term is simply given by
\begin{align*}
    R_\lambda( \WW, \HH)  = \sum_{i, j} \text{TSCAD}_{\lambda, a}(\WW_{ij}) + \sum_{i, j} \text{TSCAD}_{\lambda, a}(\HH_{ij}).
\end{align*}
Due to the inherent nonconvexity of the NNMF problem, initialisation is key to obtaining good results. We utilised a simple half normal initialisation. Indeed, because the data matrix for each NNMF example (\cref{fig:NNMF cosine text,fig:NNMF nonconvex regularisation}) satisfies $0 \leq \YY \leq 1$, we produced the initialisation by drawing $(\WW_0')_{ij} \sim \sN(0, 1)$ and $(\HH_0')_{ij} \sim \sN(0, 1)$ and normalising in the following manner
\begin{align*}
     \WW_0 \gets \frac{|\WW_0'|}{\sqrt{\max{(|\WW_0'||\HH_0'|)} }}, \quad \HH_0 \gets \frac{|\HH_0'|}{\sqrt{\max{(|\WW_0'||\HH_0'|)}}},
\end{align*}
where $| \cdot|$ is taken elementwise. This initialisation was found to result in nontrivial solutions (i.e., visually reasonable low rank representations $\HH$) to \cref{eqn:NNMF}.

\subsection{Simulations For Fast Local Convergence} \label{apx:numerical local convergence}
In \cref{fig:local convergence of MNIST logistic regression,fig:local convergence of CIFAR multinomial regression}, we consider an extended version of the results in \cref{fig:MNIST logistic regression,fig:CIFAR multinomial}, respectively. Specifically, we plot the progress in each of the termination conditions \cref{eqn:epsilon optimality conditions}. Part (a) of all figures depict the gradient norm on the inactive set. For Newton-MR TMP, this is the termination condition associated with the Newton-MR portion of the step. We see in both \cref{fig:local convergence of MNIST logistic regression,fig:local convergence of CIFAR multinomial regression} that, for our method, the inactive set termination condition is steadily reduced until a point is reached where the convergence becomes extremely rapid. This is consistent with the theoretical predictions in \cref{thm:active set identification} and \cref{cor:local convergence guarantee}. We note that projected Newton-CG exhibits similar behaviour once it reaches Newton-CG step phase but to a lesser extent.

\begin{figure}[H]
    \centering
    \includegraphics[scale=0.5]{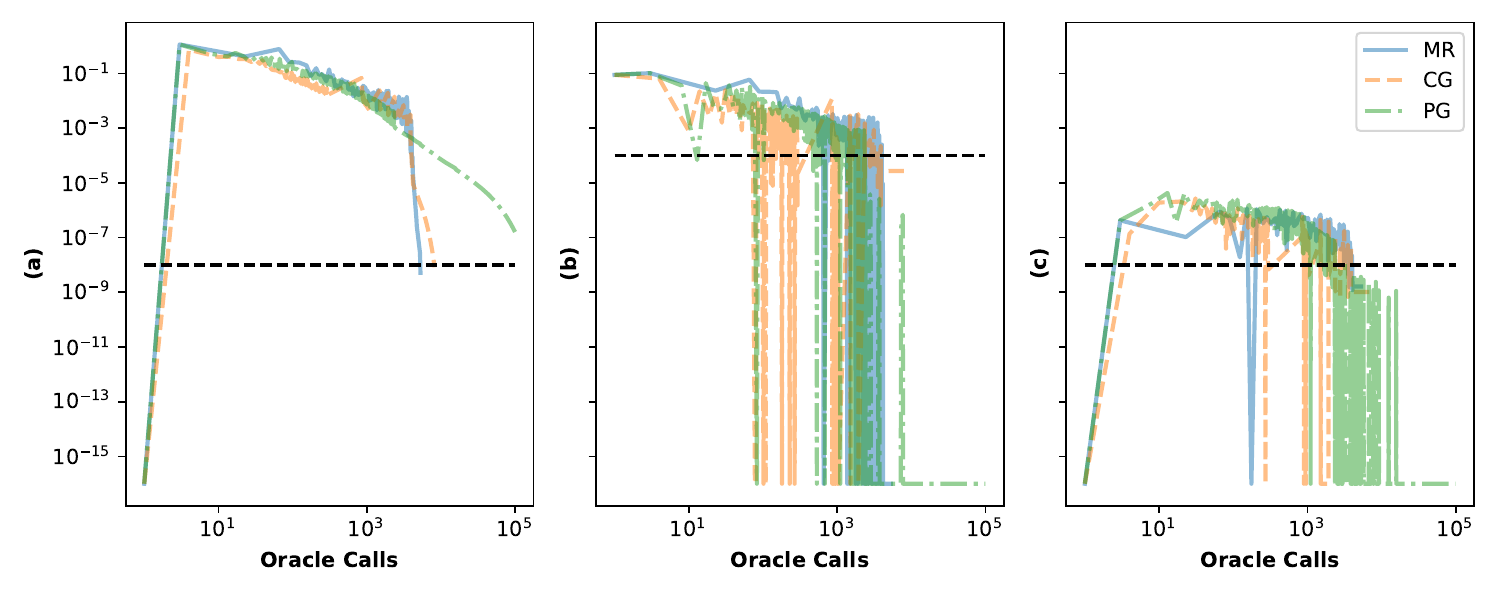}
    \caption{Termination conditions in \cref{eqn:epsilon optimality conditions} corresponding to experiment of \cref{fig:MNIST logistic regression}. (a) $\| \bgg_k^\sI \| $, (b) $-\min(\bgg_k^\sA, \zero)$ ($\min$ is taken elementwise) and (c) $\|\diag(\xx_k^\sA) \bgg_k^\sA \|$. The dashed line indicates the termination threshold for each of the respective conditions.}
    \label{fig:local convergence of MNIST logistic regression}
\end{figure}

\begin{figure}[H]
    \centering
    \includegraphics[scale=0.5]{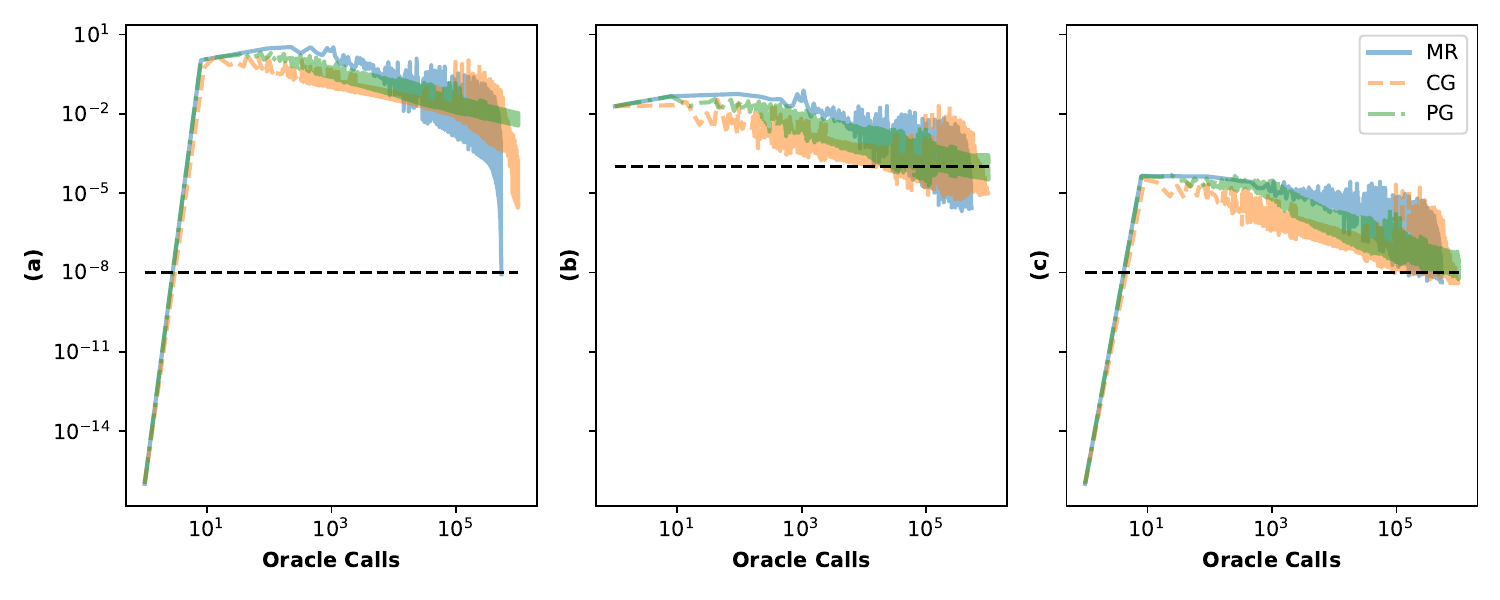}
    \caption{Termination conditions in \cref{eqn:epsilon optimality conditions} corresponding to experiment of  \cref{fig:CIFAR multinomial}. (a) $\| \bgg_k^\sI \| $, (b) $-\min(\bgg_k^\sA, \zero)$ ($\min$ is taken elementwise) and (c) $\|\diag(\xx_k^\sA) \bgg_k^\sA \|$. The dashed line indicates the termination threshold for each of the respective conditions.}
    \label{fig:local convergence of CIFAR multinomial regression}
\end{figure}

\subsection{Timing Results} \label{apx:timing results}

For completeness, in the following section we give results presented in \cref{sec:numerical results} in terms of ``wall-clock'' time. As noted earlier, wall-clock timing results are implementation and platform dependent. In particular, results are unreliable for small time scales. However, we note that, over larger time scales, the wall-clock time results generally conform with the corresponding oracle call results.

\begin{figure}[ht]
    \centering
    \includegraphics[scale=0.5]{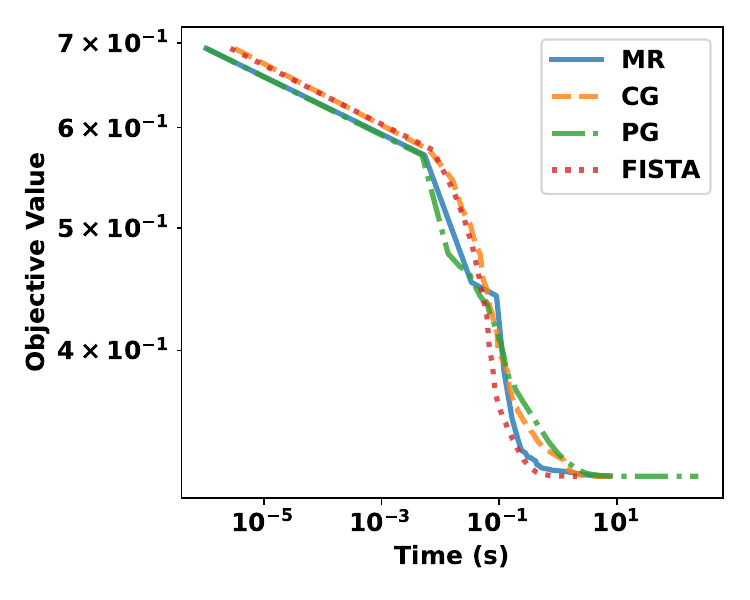}
    \caption{Wall-clock timing results for logistic regression ($C=2$) on the binarised \texttt{MNIST} dataset \citep{Lecun1998MNIST} ($d=785$) with $\lambda = 10^{-3}$.}
    \label{fig:MNIST logistic regression - time}
\end{figure}

\begin{figure}[ht]
    \centering
    \includegraphics[scale=0.5]{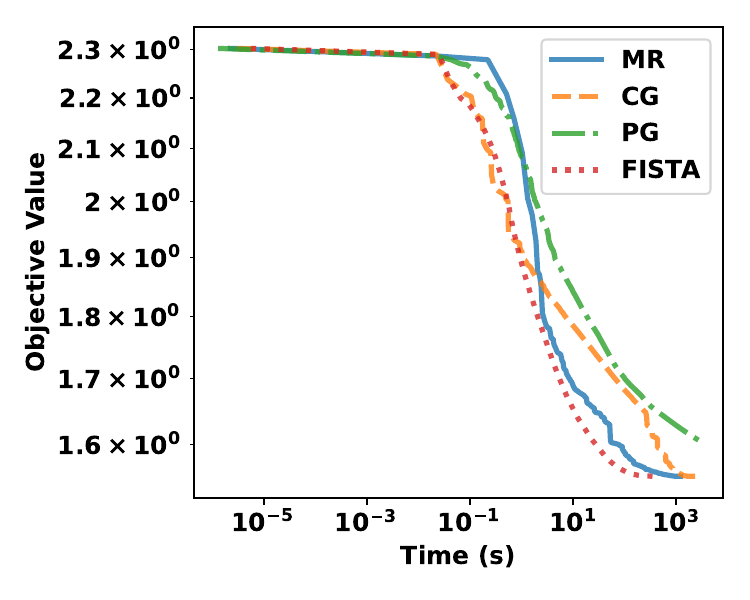}
    \caption{Wall-clock timing results for multinomial regression ($C=10$) on \texttt{CIFAR10} dataset \citep{Krizhevsky2009CIFAR10} ($d=27,657$) with $\lambda = 10^{-4}$. }
    \label{fig:CIFAR multinomial - time}
\end{figure}

\begin{figure}
    \centering
    \includegraphics[scale=0.5]{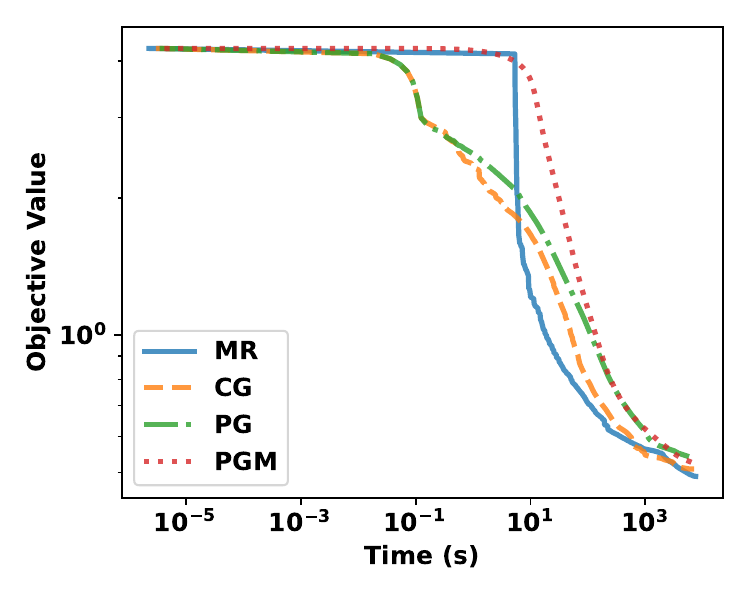}
    \caption{Wall-clock timing results for training a two-layer neural network on the \texttt{Fashion MNIST} dataset \citep{xiao2017fashionmnist} ($d=89,610$) with $\lambda = 10^{-3}$.} 
    \label{fig:MLP fashion MNIST - time}
\end{figure}

\begin{figure}
    \centering
    \includegraphics[scale=0.5]{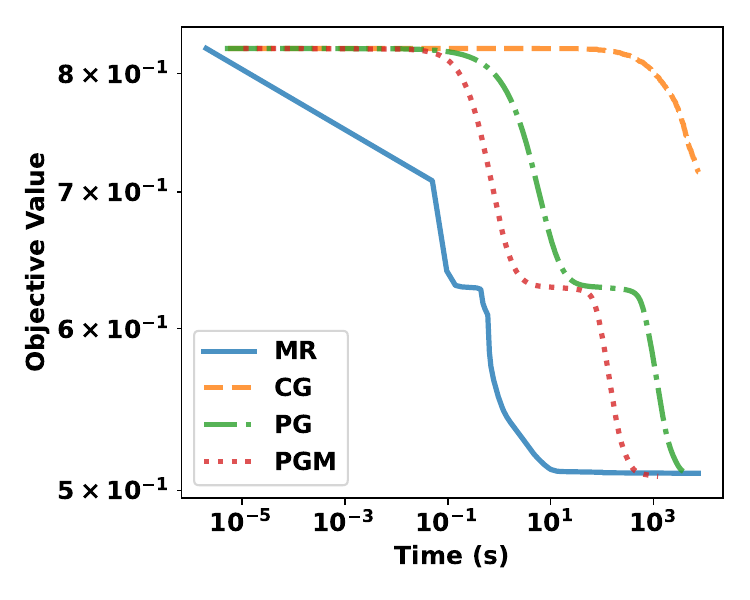}
    \caption{Wall-clock timing results for NNMF ($r=20$) with cosine distance on top 1000 TF-IDF features of the \texttt{20 Newsgroup} dataset \citep{20newsgroups} ($d=385,220$).} 
    \label{fig:NNMF cosine text - time}
\end{figure}

\begin{figure}
    \centering
    \includegraphics[scale=0.5]{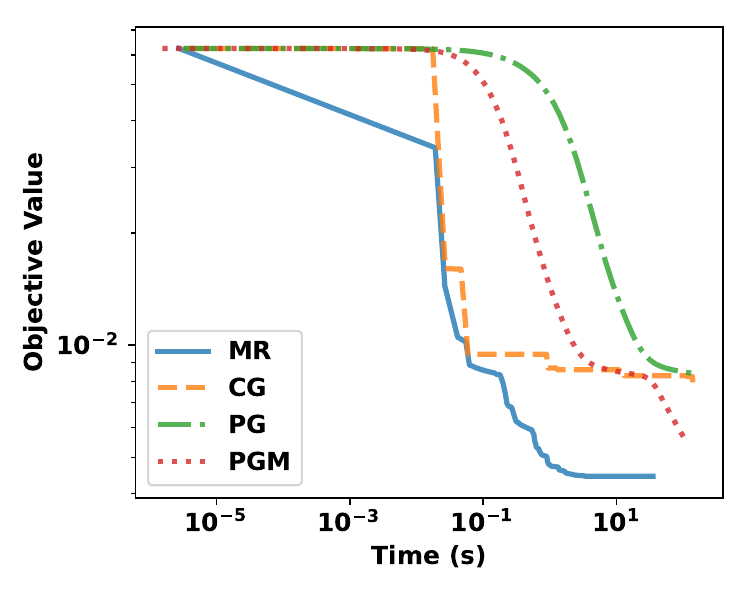}
    \caption{Wall-clock timing results for NNMF ($r=10$) with nonconvex TSCAD regulariser on the \texttt{Olivetti faces} dataset \citep{scikit-learn} ($d=44,960$). We used $a=3$ and $\lambda=10^{-4}$ for the TSCAD regulariser.}
    \label{fig:NNMF nonconvex regularisation - time}
\end{figure}

%%%%%%%%%%%%%%%%%%%%%%%%%%%%%%%%%%%%%%%%%%%%%%%%%%%%%%%%%%%%%%%%%%%%%%%%%%%%%%%
%%%%%%%%%%%%%%%%%%%%%%%%%%%%%%%%%%%%%%%%%%%%%%%%%%%%%%%%%%%%%%%%%%%%%%%%%%%%%%%

\end{document}